\documentclass[pdftex,a4paper,12pt]{scrartcl}

\usepackage[T1]{fontenc}
\usepackage{lmodern} 

\usepackage[pdftex]{graphicx}
\usepackage{amsmath,amssymb,amsthm}
\usepackage{mathtools}
\usepackage{exscale}
\usepackage{hyperref}
\usepackage[noabbrev,capitalize]{cleveref}
\usepackage{enumitem}
\usepackage{comment}
\usepackage{color} 

\usepackage{tikz}
\usetikzlibrary{cd,calc,decorations.markings}

\tikzset{
  quadratic/.style={
    to path={
      (\tikztostart) .. controls
      ($#1!1/3!(\tikztostart)$) and ($#1!1/3!(\tikztotarget)$)
      .. (\tikztotarget)
    }
  },
  quadinit/.style={
    to path={
      (\tikztostart)
      to[quadratic={($#1!1/2!(\tikztostart)$)}]
      ($1/4*(\tikztostart)+1/2*#1+1/4*(\tikztotarget)$)
    }
  },
  quadtail/.style={
    to path={
      ($1/4*(\tikztostart)+1/2*#1+1/4*(\tikztotarget)$)
      to[quadratic={($#1!1/2!(\tikztotarget)$)}]
      ($(\tikztotarget)$)
    }
  }
}

\tikzset{
  ->-/.style={
    decoration={
      markings,
      mark=at position .5 with {\arrow{>}}
    },
    postaction={decorate}
  },
  -stealth-/.style={
    decoration={
      markings,
      mark=at position .5 with {\arrow{stealth}}
    },
    postaction={decorate}
  }
}

\tikzset{
  x=.9cm,
  y=.9cm,
  xlen/.style={
    x={(0pt,#1)}
  },
  ylen/.style={
    y={(#1,0pt)}
  }
}

\makeatletter
\DeclareFontFamily{OT1}{lzc}{}
\DeclareFontShape{OT1}{lzc}{m}{it}{<->[1.15]pzcmi}{}
\DeclareMathAlphabet{\mathlzc}{OT1}{lzc}{m}{it}

\providecommand{\overbar}[2][1.5]{\mkern #1mu\overline{\mkern-#1mu#2\mkern-#1mu}\mkern #1mu}

\newcommand{\blank}{\mkern1mu\mathchar"7B\mkern1mu}

\newcommand\opposite{\mathrm{op}}

\newcommand\Cob{\mathlzc{Cob}_2^\ell}
\newcommand\Hsmooth[2][]{%
  \ifx\relax#1\relax%
  \expandafter\@firstoftwo%
  \else%
  \expandafter\@secondoftwo%
  \fi%
  {#2_{\mathsf{H}}}%
  {#2_{\mathsf{H},#1}}%
}
\newcommand\Vsmooth[1]{#1_{\mathsf{V}}}
\newcommand\Dresol[1]{#1_{\mathord{\times}}}
\newcommand\Presol[1]{#1_{\mathord{+}}}
\newcommand\Nresol[1]{#1_{\mathord{-}}}
\newcommand\PNresol[1]{#1_{\mathord{\pm}}}

\newcommand\mnphantom@impl[2]{%
  \sbox0{\ensuremath{#1#2}}\kern-.6\wd0}

\newcommand\mnphantom[1]{\mathpalette\mnphantom@impl{#1}}

\DeclarePairedDelimiterX\dblBrac[1]{\lbrack}{\rbrack}{%
  \expandafter\ifx\delimsize\middle%
  \expandafter\@firstoftwo%
  \else%
  \expandafter\@secondoftwo%
  \fi{\mnphantom{\left\lbrack\vphantom{#1}\right.}}{\mnphantom{\delimsize\lbrack}}%
  \delimsize\lbrack\mathopen{}#1\mathclose{}\delimsize\rbrack%
  \expandafter\ifx\delimsize\middle%
  \expandafter\@firstoftwo%
  \else%
  \expandafter\@secondoftwo%
  \fi{\mnphantom{\left.\vphantom{#1}\right\rbrack}}{\mnphantom{\delimsize\rbrack}}
}

\def\@replace#1#2#3#4\@mid@repl#5\@end@repl{%
  \romannumeral-`0%
  \ifx\relax#4\relax%
  \expandafter\@firstoftwo%
  \else%
  \expandafter\@secondoftwo%
  \fi%
  {%
    \ifx#1#3%
    \expandafter\@firstoftwo%
    \else%
    \expandafter\@secondoftwo%
    \fi%
    {#5#2}{#5#3}%
  }{%
    \ifx#1#3%
    \expandafter\@firstoftwo%
    \else%
    \expandafter\@secondoftwo%
    \fi%
    {\@replace#1#2#4\@mid@repl#5#2\@end@repl}%
    {\@replace#1#2#4\@mid@repl#5#3\@end@repl}%
  }%
}

\newcommand\labelseq[2]{%
  \expandafter\expandafter\expandafter\ifcase\@replace\\\or#1\@mid@repl#2\relax\or\@end@repl\else\@ctrerr\fi%
}

\makeatother

\makeatletter
\newcommand{\pgfsize}[2]{
 \pgfextractx{\@tempdima}{\pgfpointdiff{\pgfpointanchor{current bounding box}{south west}}
 {\pgfpointanchor{current bounding box}{north east}}}
 \global#1=\@tempdima
 \pgfextracty{\@tempdima}{\pgfpointdiff{\pgfpointanchor{current bounding box}{south west}}
 {\pgfpointanchor{current bounding box}{north east}}}
 \global#2=\@tempdima
}

\newcommand\saveTikzBox[3][]{%
  \expandafter\newlength\csname tikz@box@#2@wid\endcsname%
  \expandafter\newlength\csname tikz@box@#2@hei\endcsname%
  \expandafter\newsavebox\csname tikz@box@#2\endcsname%
  \expandafter\savebox\csname tikz@box@#2\endcsname{%
    \tikz[#1]{%
      #3%
      \pgfextractx{\@tempdima}{\pgfpointdiff{\pgfpointanchor{current bounding box}{south west}}{\pgfpointanchor{current bounding box}{north east}}}
      \expandafter\global\csname tikz@box@#2@wid\endcsname=\@tempdima
      \pgfextracty{\@tempdima}{\pgfpointdiff{\pgfpointanchor{current bounding box}{south west}}{\pgfpointanchor{current bounding box}{north east}}}
      \expandafter\global\csname tikz@box@#2@hei\endcsname=\@tempdima
    }%
  }%
}

\newcommand\useTikzBox[1]{\expandafter\usebox\csname tikz@box@#1\endcsname}

\newcommand\defTikzBox[3][]{%
  \saveTikzBox[#1]{#2}{#3}%
  \expandafter\newcommand\csname#2\endcsname{\useTikzBox{#2}}%
}
\makeatother

\defTikzBox[baseline=-.5ex]{diagCrossPosUp}{%
  \draw[red,very thick,-stealth] (.5,-.6) -- (-.5,.6);
  \fill[white] (0,0) circle(.15);
  \draw[red,very thick,-stealth] (-.5,-.6) -- (.5,.6);
}

\newcommand\diagCrossPosUpWith[2][]{%
  \ooalign{\diagCrossPosUp\crcr\hss\ensuremath{\mathllap{#1}\hspace{1em}\mathrlap{#2}}\hss}
}

\defTikzBox[baseline=-.5ex]{diagCrossNegUp}{%
  \draw[red,very thick,-stealth] (-.5,-.6) -- (.5,.6);
  \fill[white] (0,0) circle(.15);
  \draw[red,very thick,-stealth] (.5,-.6) -- (-.5,.6);
}

\newcommand\diagCrossNegUpWith[2][]{%
  \ooalign{\diagCrossNegUp\crcr\hss\ensuremath{\mathllap{#1}\hspace{1em}\mathrlap{#2}}\hss}
}

\defTikzBox[baseline=-.5ex]{diagSmoothV}{%
  \draw[red,very thick] (-.5,-.6) .. controls(0,0) .. (-.5,.6);
  \draw[red,very thick] (.5,-.6) .. controls(0,0) .. (.5,.6);
}

\defTikzBox[baseline=-.5ex]{diagSmoothUp}{%
  \draw[red,very thick,-stealth] (-.5,-.6) .. controls(0,0) .. (-.5,.6);
  \draw[red,very thick,-stealth] (.5,-.6) .. controls(0,0) .. (.5,.6);
}

\defTikzBox[baseline=-.5ex]{diagSmoothUptwL}{%
  \draw[blue,very thick,-stealth] (-.5,-.6) .. controls(0,0) .. (-.5,.6);
  \draw[red,very thick,-stealth] (.5,-.6) .. controls(0,0) .. (.5,.6);
}

\defTikzBox[baseline=-.5ex]{diagSmoothUptwLRA}{%
  \useasboundingbox (-.5,-.8) rectangle (.9,.8);
  \draw[blue,very thick,-stealth] (-.5,-.6) .. controls(0,0) .. (-.5,.6);
  \draw[red,very thick,-stealth] (.5,-.6) .. controls(0,0) .. (.5,.6);
  \draw[red,very thick,dotted] (.5,.6) ..controls+(.5,.6)and+(.5,-.6) .. (.5,-.6);
}

\defTikzBox[baseline=-.5ex]{diagSmoothUptwR}{%
  \draw[red,very thick,-stealth] (-.5,-.6) .. controls(0,0) .. (-.5,.6);
  \draw[blue,very thick,-stealth] (.5,-.6) .. controls(0,0) .. (.5,.6);
}

\defTikzBox[baseline=-.5ex]{diagSmoothUptwRLA}{%
  \useasboundingbox (-.9,-.8) rectangle (.5,.8);
  \draw[red,very thick,-stealth] (-.5,-.6) .. controls(0,0) .. (-.5,.6);
  \draw[red,very thick,dotted] (-.5,.6) ..controls+(-.5,.6)and+(-.5,-.6) .. (-.5,-.6);
  \draw[blue,very thick,-stealth] (.5,-.6) .. controls(0,0) .. (.5,.6);
}

\defTikzBox[baseline=-.5ex]{diagSmoothUptwB}{%
  \draw[blue,very thick,-stealth] (-.5,-.6) .. controls(0,0) .. (-.5,.6);
  \draw[blue,very thick,-stealth] (.5,-.6) .. controls(0,0) .. (.5,.6);
}

\defTikzBox[baseline=-.5ex]{diagSmoothDown}{%
  \draw[red,very thick,stealth-] (-.5,-.6) .. controls(0,0) .. (-.5,.6);
  \draw[red,very thick,stealth-] (.5,-.6) .. controls(0,0) .. (.5,.6);
}

\defTikzBox[baseline=(current bounding box.center)]{diagSmoothUptwBDA}{%
  \useasboundingbox (-.6,-1.2) rectangle (.6,.5);
  \draw[blue,very thick,-stealth] (-.5,-.6) .. controls(0,0) .. (-.5,.6);
  \draw[blue,very thick,-stealth] (.5,-.6) .. controls(0,0) .. (.5,.6);
  \draw[blue,very thick,dotted] (.4,-.5) .. controls+(.6,-.75)and+(-.6,-.75) .. (-.4,-.5);
}

\defTikzBox[baseline=(current bounding box.center)]{diagSmoothUptwBUA}{%
  \useasboundingbox (-.6,-.7) rectangle (.6,1);
  \draw[blue,very thick,-stealth] (-.5,-.6) .. controls(0,0) .. (-.5,.6);
  \draw[blue,very thick,-stealth] (.5,-.6) .. controls(0,0) .. (.5,.6);
  \draw[blue,very thick,dotted] (.4,.5) .. controls+(.6,.75)and+(-.6,.75) .. (-.4,.5);
}

\defTikzBox[baseline=-.5ex]{diagSmoothH}{%
  \draw[red,very thick] (-.5,-.6) .. controls(0,0) .. (.5,-.6);
  \draw[red,very thick] (-.5,.6) .. controls(0,0) .. (.5,.6);
}

\defTikzBox[baseline=-.5ex]{diagSmoothW}{%
  \draw[gray,line width=3,-stealth] (0,-.3) -- (0,.3);
  \draw[red,very thick,-stealth] (-.5,-.6) to[out=60,in=180] (0,-.3);
  \draw[red,very thick,-stealth] (.5,-.6) to[out=120,in=0] (0,-.3);
  \draw[red,very thick,-stealth] (0,.3) to[out=180,in=-60] (-.5,.6);
  \draw[red,very thick,-stealth] (0,.3) to[out=0,in=-120] (.5,.6);
}

\defTikzBox[baseline=-.5ex]{diagSmoothWtwU}{%
  \draw[gray,line width=3,-stealth] (0,-.3) -- (0,.3);
  \draw[red,very thick,-stealth] (-.5,-.6) to[out=60,in=180] (0,-.3);
  \draw[red,very thick,-stealth] (.5,-.6) to[out=120,in=0] (0,-.3);
  \draw[blue,very thick,-stealth] (0,.3) to[out=180,in=-60] (-.5,.6);
  \draw[blue,very thick,-stealth] (0,.3) to[out=0,in=-120] (.5,.6);
}

\defTikzBox[baseline=(current bounding box.center)]{diagSmoothWtwUDA}{%
  \useasboundingbox (-.6,-1.2) rectangle (.6,.5);
  \draw[gray,line width=3,-stealth] (0,-.3) -- (0,.3);
  \draw[red,very thick,-stealth] (-.5,-.6) to[out=60,in=180] (0,-.3);
  \draw[red,very thick,-stealth] (.5,-.6) to[out=120,in=0] (0,-.3);
  \draw[red,very thick,dotted] (.4,-.5) .. controls+(.6,-.75)and+(-.6,-.75) .. (-.4,-.5);
  \draw[blue,very thick,-stealth] (0,.3) to[out=180,in=-60] (-.5,.6);
  \draw[blue,very thick,-stealth] (0,.3) to[out=0,in=-120] (.5,.6);
}

\defTikzBox[baseline=-.5ex]{diagSmoothWtwD}{%
  \draw[gray,line width=3,-stealth] (0,-.3) -- (0,.3);
  \draw[blue,very thick,-stealth] (-.5,-.6) to[out=60,in=180] (0,-.3);
  \draw[blue,very thick,-stealth] (.5,-.6) to[out=120,in=0] (0,-.3);
  \draw[red,very thick,-stealth] (0,.3) to[out=180,in=-60] (-.5,.6);
  \draw[red,very thick,-stealth] (0,.3) to[out=0,in=-120] (.5,.6);
}

\defTikzBox[baseline=(current bounding box.center)]{diagSmoothWtwDUA}{%
  \useasboundingbox (-.6,-.7) rectangle (.6,1);
  \draw[gray,line width=3,-stealth] (0,-.3) -- (0,.3);
  \draw[blue,very thick,-stealth] (-.5,-.6) to[out=60,in=180] (0,-.3);
  \draw[blue,very thick,-stealth] (.5,-.6) to[out=120,in=0] (0,-.3);
  \draw[red,very thick,-stealth] (0,.3) to[out=180,in=-60] (-.5,.6);
  \draw[red,very thick,-stealth] (0,.3) to[out=0,in=-120] (.5,.6);
  \draw[red,very thick,dotted] (.4,.5) .. controls+(.6,.75)and+(-.6,.75) .. (-.4,.5);
}

\defTikzBox[baseline=-.5ex]{diagSmoothWtwB}{%
  \draw[gray,line width=3,-stealth] (0,-.3) -- (0,.3);
  \draw[blue,very thick,-stealth] (-.5,-.6) to[out=60,in=180] (0,-.3);
  \draw[blue,very thick,-stealth] (.5,-.6) to[out=120,in=0] (0,-.3);
  \draw[blue,very thick,-stealth] (0,.3) to[out=180,in=-60] (-.5,.6);
  \draw[blue,very thick,-stealth] (0,.3) to[out=0,in=-120] (.5,.6);
}

\defTikzBox[baseline=-.5ex]{diagSmoothWtwBLA}{%
  \useasboundingbox (-.9,-.8) rectangle (.5,.8);
  \draw[gray,line width=3,-stealth] (0,-.3) -- (0,.3);
  \draw[blue,very thick,-stealth] (-.5,-.6) to[out=60,in=180] (0,-.3);
  \draw[blue,very thick,-stealth] (.5,-.6) to[out=120,in=0] (0,-.3);
  \draw[blue,very thick,-stealth] (0,.3) to[out=180,in=-60] (-.5,.6);
  \draw[blue,very thick,-stealth] (0,.3) to[out=0,in=-120] (.5,.6);
  \draw[blue,very thick,dotted] (-.5,.6) ..controls+(-.5,.6)and+(-.5,-.6) .. (-.5,-.6);
}

\defTikzBox[baseline=-.5ex]{diagSmoothWtwBRA}{%
  \useasboundingbox (-.5,-.8) rectangle (.9,.8);
  \draw[gray,line width=3,-stealth] (0,-.3) -- (0,.3);
  \draw[blue,very thick,-stealth] (-.5,-.6) to[out=60,in=180] (0,-.3);
  \draw[blue,very thick,-stealth] (.5,-.6) to[out=120,in=0] (0,-.3);
  \draw[blue,very thick,-stealth] (0,.3) to[out=180,in=-60] (-.5,.6);
  \draw[blue,very thick,-stealth] (0,.3) to[out=0,in=-120] (.5,.6);
  \draw[blue,very thick,dotted] (.5,.6) ..controls+(.5,.6)and+(.5,-.6) .. (.5,-.6);
}

\defTikzBox[baseline=-.5ex]{diagSingUp}{%
  \draw[red,very thick,-stealth] (-.5,-.6) -- (.5,.6);
  \draw[red,very thick,-stealth] (.5,-.6) -- (-.5,.6);
  \fill[blue] (0,0) circle (.1);
}

\defTikzBox[baseline=-.5ex,scale=.75]{diagFiSing}{%
  \draw[red,very thick,-stealth] (-.5,-1) -- (0,0) .. controls+(.5,1)and+(0,.5) .. (0.7,0) .. controls+(0,-.5)and+(.5,-1) .. (0,0) -- (-.5,1);
  \fill[blue] (0,0) circle(.1);
}

\defTikzBox[baseline=-.5ex,scale=.75]{diagFiH}{%
  \draw[red,very thick] (.7,0) .. controls+(0,-.5)and+(.4,-.8) .. (.15,-.3) .. controls(.05,-.1) and (-.05,-.1) .. (-.15,-.3) -- (-.5,-1);
  \draw[red,very thick] (.7,0) .. controls+(0,.5)and+(.4,.8) .. (.15,.3) .. controls(.05,.1) and (-.05,.1) .. (-.15,.3) -- (-.5,1);
}

\defTikzBox[baseline=-.5ex,scale=.75]{diagFiV}{%
  \draw[red,very thick] (.7,0) .. controls+(0,-.5)and+(.4,-.8) .. (.15,-.3) .. controls(.05,-.1) and (.05,.1) .. (.15,.3) .. controls+(.4,.8)and+(0,.5) .. (.7,0);
  \draw[red,very thick] (-.5,1) -- (-.15,.3) .. controls(-.05,.1)and(-.05,-.1) .. (-.15,-.3) -- (-.5,-1);
}

\defTikzBox[baseline=-.5ex,scale=.75]{diagFiNil}{%
  \draw[red,very thick] (-.5,1) -- (-.15,.3) .. controls(-.05,.1)and(-.05,-.1) .. (-.15,-.3) -- (-.5,-1);
}

\defTikzBox[baseline=(current bounding box.center)]{diagTrefoilCruxOl}{%
  \draw[blue,very thick] (-.325,.3125) arc (210:-30:.375);
  \draw[red,very thick] (-.433,.125) arc(90:390:.375);
  \draw[blue,very thick] (.433,.125) arc(90:-210:.375);
  \draw[red,very thick] (-.433,.125) .. controls (-.361,.125) and (-.315,.141) .. (-.297,.172);
  \draw[blue,very thick] (-.297,.172) .. controls (-.279,.203) and (-.289,.25) .. (-.325,.3125);
  \draw[red,very thick] (-.108,-.0625) .. controls (-.144,0) and (-.153,.047) .. (-.135,.078);
  \draw[blue,very thick] (-.135,.078) .. controls (-.117,.109) and (-.072,.125) .. (0,.125);
  \draw[blue,very thick] (.433,.125) .. controls (.288,.125)and(.252,.1875) .. (.325,.3125);
  \draw[blue,very thick] (.108,-.0625) .. controls (.18,.0625)and(.144,.125) .. (0,.125);
}

\defTikzBox[baseline=(current bounding box.center)]{diagTrefoilCruxlO}{%
  \draw[blue,very thick] (-.325,.3125) arc (210:-90:.375);
  \draw[red,very thick] (-.433,.125) arc(90:330:.375);
  \draw[red,very thick] (.108,-.4375) arc(-150:150:.375);
  \draw[red,very thick] (-.433,.125) .. controls (-.361,.125) and (-.315,.141) .. (-.297,.172);
  \draw[blue,very thick] (-.297,.172) .. controls (-.279,.203) and (-.289,.25) .. (-.325,.3125);
  \draw[red,very thick] (-.108,-.0625) .. controls (-.144,0) and (-.153,.047) .. (-.135,.078);
  \draw[blue,very thick] (-.135,.078) .. controls (-.117,.109) and (-.072,.125) .. (0,.125);
  \draw[red,very thick] (-.108,-.0625) .. controls (-.036,-.1875)and(.036,-.1875) .. (.108,-.0625);
  \draw[red,very thick] (-.108,-.4375) .. controls (-.036,-.3125)and(.036,-.3125) .. (.108,-.4375);
}

\defTikzBox[baseline=(current bounding box.center)]{diagTrefoilCruxll}{%
  \draw[blue,very thick] (-.325,.3125) arc (210:-90:.375);
  \draw[red,very thick] (-.433,.125) arc(90:390:.375);
  \draw[red,very thick] (.433,-.25) circle (.375);
  \draw[red,very thick] (-.433,.125) .. controls (-.361,.125) and (-.315,.141) .. (-.297,.172);
  \draw[blue,very thick] (-.297,.172) .. controls (-.279,.203) and (-.289,.25) .. (-.325,.3125);
  \draw[red,very thick] (-.108,-.0625) .. controls (-.144,0) and (-.153,.047) .. (-.135,.078);
  \draw[blue,very thick] (-.135,.078) .. controls (-.117,.109) and (-.072,.125) .. (0,.125);
}

\defTikzBox[baseline=-.5ex]{diagRvSingP}{%
  \draw[red,very thick,-stealth] (-.5,-.7) .. controls (.5,0) .. (-.5,.7);
  \fill[white] (0,.35) circle(.15);
  \draw[red,very thick,-stealth] (.5,-.7) .. controls (-.5,0) .. (.5,.7);
  \fill[blue] (0,-.35) circle(.1);
}

\defTikzBox[baseline=-.5ex]{diagRvSingN}{%
  \draw[red,very thick,-stealth] (.5,-.7) .. controls (-.5,0) .. (.5,.7);
  \fill[white] (0,.35) circle(.15);
  \draw[red,very thick,-stealth] (-.5,-.7) .. controls (.5,0) .. (-.5,.7);
  \fill[blue] (0,-.35) circle(.1);
}

\defTikzBox[baseline=-.5ex]{diagCruxCompCrx}{%
  \useasboundingbox (-.5,-1) rectangle (.5,1);
  \draw[red,very thick] (.1,0) .. controls+(0,-.125)and+(-.2,.25) .. (.4,-.5);
  \draw[red,very thick,dotted] (.4,-.5) .. controls+(.6,-.75)and+(-.6,-.75) .. (-.4,-.5);
  \draw[red,very thick] (-.4,-.5) .. controls+(.2,.25)and+(0,-.125) .. (-.1,0);
  \draw[blue,very thick] (.1,0) .. controls+(0,.125)and+(-.2,-.25) .. (.4,.5);
  \draw[blue,very thick,dotted] (.4,.5) .. controls+(.6,.75)and+(-.6,.75) .. (-.4,.5);
  \draw[blue,very thick] (-.4,.5) .. controls+(.2,-.25)and+(0,.125) .. (-.1,0);
}

\defTikzBox[baseline=-.5ex]{diagCruxCompH}{%
  \useasboundingbox (-.5,-1) rectangle (.5,1);
  \draw[red,very thick,dotted] (.4,-.5) .. controls+(.6,-.75)and+(-.6,-.75) .. (-.4,-.5);
  \draw[red,very thick] (.4,-.5) .. controls(0,0) .. (-.4,-.5);
  \draw[red,very thick,dotted] (.4,.5) .. controls+(.6,.75)and+(-.6,.75) .. (-.4,.5);
  \draw[red,very thick] (.4,.5) .. controls(0,0) .. (-.4,.5);
}

\defTikzBox[baseline=-.5ex]{diagCruxCompV}{%
  \useasboundingbox (-.5,-1) rectangle (.5,1);
  \draw[red,very thick] (.4,-.5) .. controls(0,0) .. (.4,.5);
  \draw[red,very thick] (-.4,-.5) .. controls(0,0) .. (-.4,.5);
  \draw[red,very thick,dotted] (.4,-.5) .. controls+(.6,-.75)and+(-.6,-.75) .. (-.4,-.5);
  \draw[red,very thick,dotted] (.4,.5) .. controls+(.6,.75)and+(-.6,.75) .. (-.4,.5);
}

\defTikzBox[baseline=-.5ex]{diagNonCruxH}{%
  \useasboundingbox (-.8,-.6) rectangle (.8,.6);
  \draw[red,very thick] (.4,-.5) .. controls(0,0) .. (-.4,-.5);
  \draw[red,very thick] (.4,.5) .. controls(0,0) .. (-.4,.5);
  \draw[red,very thick,dotted] (.4,-.5) .. controls+(.6,-.75)and+(.6,.75) .. (.4,.5);
  \draw[red,very thick,dotted] (-.4,-.5) .. controls+(-.6,-.75)and+(-.6,.75) .. (-.4,.5);
}

\defTikzBox[baseline=-.5ex]{diagNonCruxV}{%
  \useasboundingbox (-.8,-.6) rectangle (.8,.6);
  \draw[red,very thick] (.4,-.5) .. controls(0,0) .. (.4,.5);
  \draw[red,very thick] (-.4,-.5) .. controls(0,0) .. (-.4,.5);
  \draw[red,very thick,dotted] (.4,-.5) .. controls+(.6,-.75)and+(.6,.75) .. (.4,.5);
  \draw[red,very thick,dotted] (-.4,-.5) .. controls+(-.6,-.75)and+(-.6,.75) .. (-.4,.5);
}

\defTikzBox[baseline=-.5ex,xlen=.4pt,ylen=-.5pt]{BordDeltaP}{%
  \draw[red,very thick] (47.54,16.43) -- (33.46,16.43);
  \draw[red,very thick,dotted] (33.46,16.43) -- (-32.46,16.43);
  \draw[red,very thick] (33.46,24.57) -- (-46.54,24.57);
  \draw[black] (18.66,-17.66) .. controls (18.66,9.01)and(-17.66,5.324) .. (-17.66,-21.34);
  \draw[black] (47.54,16.43) -- (47.54,-23.57);
  \draw[black] (33.46,24.57) -- (33.46,-15.43);
  \draw[black,thick,dotted] (-32.46,16.43) -- (-32.46,-16.1);
  \draw[black] (-32.46,-16.1) -- (-32.46,-23.57);
  \draw[black] (-46.54,24.57) -- (-46.54,-15.43);
  \draw[red,very thick] (47.54,-23.57) to[quadratic={(27.54,-23.57)}] (20.5,-19.5) to[quadratic={(13.46,-15.43)}] (33.46,-15.43);
  \draw[red,very thick]  (-32.46,-23.57) to[quadratic={(-12.46,-23.57)}] (-19.5,-19.5) to[quadratic={(-26.54,-15.43)}] (-46.54,-15.43);
}

\defTikzBox[baseline=-.5ex,xlen=.4pt,ylen=-.5pt]{BordDeltaN}{%
  \draw[red,very thick]  (47.54,16.43) to[quadratic={(39.43,16.43)}] (33.46,17.1);
  \draw[red,very thick,dotted]  (33.46,17.1) to[quadratic={(24.69,18.08)}] (20.5,20.5) to[quadratic={(18.67,21.56)}] (18.67,22.34);
  \draw[red,very thick] (18.67,22.34) to[quadratic={(18.67,24.57)}] (33.46,24.57);
  \draw[red,very thick,dotted]  (-32.46,16.43) to[quadratic={(-17.67,16.43)}] (-17.67,18.66);
  \draw[red,very thick] (-17.67,18.66) to[quadratic={(-17.67,19.44)}] (-19.5,20.5) to[quadratic={(-26.54,24.57)}] (-46.54,24.57);
  \draw[black]  (18.66,22.34) .. controls (18.66,-4.324)and(-17.66,-8.01) .. (-17.66,18.66);
  \draw[black] (47.54,16.43) -- (47.54,-23.57);
  \draw[black] (33.46,24.57) -- (33.46,-15.43);
  \draw[black,thick,dotted] (-32.46,16.43) -- (-32.46,-15.43);
  \draw[black] (-32.46,-15.43) -- (-32.46,-23.57);
  \draw[black] (-46.54,24.57) -- (-46.54,-15.43);
  \draw[red,very thick] (47.54,-23.57) -- (-32.46,-23.57);
  \draw[red,very thick] (33.46,-15.43) -- (-46.54,-15.43);
}

\defTikzBox[baseline=-.5ex,xlen=.4pt,ylen=-.4pt]{BordPhiFst}{%
  \draw[red,very thick] (47.54,36.43) -- (33.46,36.43);
  \draw[red,very thick,dotted]  (33.46,36.43) -- (-32.46,36.43);
  \draw[red,very thick] (33.46,44.57) -- (-46.54,44.57);
  \draw[black] (18.66,2.343) .. controls (18.66,29.01)and(-17.66,25.32) .. (-17.66,-1.343) .. controls (-17.66,-28.01)and(18.66,-24.32) .. (18.66,2.343);
  \draw[black] (47.54,36.43) -- (47.54,-43.57);
  \draw[black] (33.46,44.57) -- (33.46,-35.43);
  \draw[black,thick,dotted] (-32.46,36.43) -- (-32.46,-35.43);
  \draw[black] (-32.46,-35.43) -- (-32.46,-43.57);
  \draw[black] (-46.54,44.57) -- (-46.54,-35.43);
  \draw[red,very thick] (47.54,-43.57) -- (-32.46,-43.57);
  \draw[red,very thick] (33.46,-35.43) -- (-46.54,-35.43);
}

\defTikzBox[baseline=-.5ex,xlen=.4pt,ylen=-.4pt]{BordPhiSnd}{%
  \draw[red,very thick] (47.54,36.43) -- (33.46,36.43);
  \draw[red,very thick,dotted]  (33.46,36.43) -- (-32.46,36.43);
  \draw[red,very thick] (33.46,44.57) -- (-46.54,44.57);
  \draw[black] (23.92,32.57) .. controls+(0,-10)and+(0,15) .. (50.92,4.57) .. controls+(0,-15)and+(0,10) .. (23.92,-23.43);
  \draw[black] (23.92,10.57) .. controls+(0,5)and+(0,7) .. (37.92,4.57) .. controls+(0,-7)and+(0,-5) .. (23.92,-1.43);
  \draw[black] (47.54,36.43) -- (47.54,13);
  \draw[black,thick,dotted] (47.54,13) -- (47.54,-3.7);
  \draw[black] (47.54,-3.7) -- (47.54,-43.57);
  \draw[black] (33.46,44.57) -- (33.46,22.9);
  \draw[black,thick,dotted] (33.46,22.9) -- (33.46,11.1);
  \draw[black] (33.46,11.1) -- (33.46,-1.8);
  \draw[black,thick,dotted] (33.46,-1.8) -- (33.46,-13.7);
  \draw[black] (33.46,-13.7) -- (33.46,-35.43);
  \draw[black,thick,dotted] (-32.46,36.43) -- (-32.46,-35.43);
  \draw[black] (-32.46,-35.43) -- (-32.46,-43.57);
  \draw[black] (-46.54,44.57) -- (-46.54,-35.43);
  \draw[red,very thick] (47.54,-43.57) -- (-32.46,-43.57);
  \draw[red,very thick] (33.46,-35.43) -- (-46.54,-35.43);
}

\defTikzBox[baseline=-.5ex,xlen=.4pt,ylen=-.4pt]{BordPhiSndVar}{%
  \draw[red,very thick] (47.54,36.43) -- (33.46,36.43);
  \draw[red,very thick,dotted]  (33.46,36.43) -- (-32.46,36.43);
  \draw[red,very thick] (33.46,44.57) -- (-46.54,44.57);
  \draw[black] (38,24.43) .. controls+(0,-10)and+(0,15) .. (65,-3.57) .. controls+(0,-15)and+(0,10) .. (38,-31.57);
  \draw[black] (38,2.43) .. controls+(0,5)and+(0,7) .. (52,-3.57) .. controls+(0,-7)and+(0,-5) .. (38,-9.57);
  \draw[black] (47.54,36.43) -- (47.54,14.7);
  \draw[black,thick,dotted] (47.54,14.7) -- (47.54,2.8);
  \draw[black] (47.54,2.8) -- (47.54,-10);
  \draw[black,thick,dotted] (47.54,-10) -- (47.54,-21.8);
  \draw[black] (47.54,-21.8) -- (47.54,-43.57);
  \draw[black] (33.46,44.57) -- (33.46,-35.43);
  \draw[black,thick,dotted] (-32.46,36.43) -- (-32.46,-35.43);
  \draw[black] (-32.46,-35.43) -- (-32.46,-43.57);
  \draw[black] (-46.54,44.57) -- (-46.54,-35.43);
  \draw[red,very thick] (47.54,-43.57) -- (-32.46,-43.57);
  \draw[red,very thick] (33.46,-35.43) -- (-46.54,-35.43);
}

\defTikzBox[xlen=.4pt,ylen=-.5pt]{BordFourTuL}{%
  \draw[red,very thick] (62.33,-61.34) to[quadratic={(62.33,-60.56)}] (60.5,-59.5) to[quadratic={(53.46,-55.43)}] (33.46,-55.43) to[quadratic={(18.67,-55.43)}] (18.67,-57.66);
  \draw[red,very thick] (18.67,-57.66) to[quadratic={(18.67,-58.44)}] (20.5,-59.5) to[quadratic={(27.54,-63.57)}] (47.54,-63.57) to[quadratic={(62.33,-63.57)}] (62.33,-61.34);
  \draw[red,very thick] (-17.33,-61.34) to[quadratic={(-17.33,-60.56)}] (-19.5,-59.5) to[quadratic={(-26.46,-55.43)}] (-46.46,-55.43) to[quadratic={(-61.67,-55.43)}] (-61.67,-57.66);
  \draw[red,very thick] (-61.67,-57.66) to[quadratic={(-61.67,-58.44)}] (-59.5,-59.5) to[quadratic={(-52.54,-63.57)}] (-32.54,-63.57) to[quadratic={(-17.33,-63.57)}] (-17.33,-61.34);
  \draw[black] (62.34,-61.34) .. controls (62.34,-34.68)and(18.66,-30.99) .. (18.66,-57.66);
  \draw[black]  (-17.66,-61.34) -- (-17.66,58.66);
  \draw[black]  (-61.34,-57.66) -- (-61.34,62.34);
  \draw[black]  (62.34,58.66) .. controls (62.34,31.99)and(18.66,35.68) .. (18.66,62.34);
  \draw[red,very thick] (62.33,58.34) to[quadratic={(62.33,59.56)}] (60.5,60.5) to[quadratic={(53.46,64.43)}] (33.46,64.43) to[quadratic={(18.67,64.43)}] (18.67,62.66);
  \draw[red,very thick,dotted] (18.67,62.66) to[quadratic={(18.67,61.44)}] (20.5,60.5) to[quadratic={(27.54,56.57)}] (47.54,56.57) to[quadratic={(62.33,56.57)}] (62.33,58.34);
  \draw[red,very thick] (-17.33,58.34) to[quadratic={(-17.33,59.56)}] (-19.5,60.5) to[quadratic={(-26.46,64.43)}] (-46.46,64.43) to[quadratic={(-61.67,64.43)}] (-61.67,62.66);
  \draw[red,very thick,dotted] (-61.67,62.66) to[quadratic={(-61.67,61.44)}] (-59.5,60.5) to[quadratic={(-52.54,56.57)}] (-32.54,56.57) to[quadratic={(-17.33,56.57)}] (-17.33,58.34);
}

\defTikzBox[xlen=.4pt,ylen=-.5pt]{BordFourTuR}{%
  \draw[red,very thick] (62.33,-61.34) to[quadratic={(62.33,-60.56)}] (60.5,-59.5) to[quadratic={(53.46,-55.43)}] (33.46,-55.43) to[quadratic={(18.67,-55.43)}] (18.67,-57.66);
  \draw[red,very thick] (18.67,-57.66) to[quadratic={(18.67,-58.44)}] (20.5,-59.5) to[quadratic={(27.54,-63.57)}] (47.54,-63.57) to[quadratic={(62.33,-63.57)}] (62.33,-61.34);
  \draw[red,very thick] (-17.33,-61.34) to[quadratic={(-17.33,-60.56)}] (-19.5,-59.5) to[quadratic={(-26.46,-55.43)}] (-46.46,-55.43) to[quadratic={(-61.67,-55.43)}] (-61.67,-57.66);
  \draw[red,very thick] (-61.67,-57.66) to[quadratic={(-61.67,-58.44)}] (-59.5,-59.5) to[quadratic={(-52.54,-63.57)}] (-32.54,-63.57) to[quadratic={(-17.33,-63.57)}] (-17.33,-61.34);
  \draw[black] (-17.66,-61.34) .. controls (-17.66,-34.68)and(-61.34,-30.99) .. (-61.34,-57.66);
  \draw[black] (62.34,-61.34) -- (62.34,58.66);
  \draw[black] (18.66,-57.66) -- (18.66,62.34);
  \draw[black] (-17.66,58.66) .. controls (-17.66,31.99)and(-61.34,35.68) .. (-61.34,62.34);
  \draw[red,very thick] (62.33,58.34) to[quadratic={(62.33,59.56)}] (60.5,60.5) to[quadratic={(53.46,64.43)}] (33.46,64.43) to[quadratic={(18.67,64.43)}] (18.67,62.66);
  \draw[red,very thick,dotted] (18.67,62.66) to[quadratic={(18.67,61.44)}] (20.5,60.5) to[quadratic={(27.54,56.57)}] (47.54,56.57) to[quadratic={(62.33,56.57)}] (62.33,58.34);
  \draw[red,very thick] (-17.33,58.34) to[quadratic={(-17.33,59.56)}] (-19.5,60.5) to[quadratic={(-26.46,64.43)}] (-46.46,64.43) to[quadratic={(-61.67,64.43)}] (-61.67,62.66);
  \draw[red,very thick,dotted] (-61.67,62.66) to[quadratic={(-61.67,61.44)}] (-59.5,60.5) to[quadratic={(-52.54,56.57)}] (-32.54,56.57) to[quadratic={(-17.33,56.57)}] (-17.33,58.34);
}

\defTikzBox[xlen=.4pt,ylen=-.5pt]{BordFourTuD}{%
  \draw[red,very thick] (62.33,-61.34) to[quadratic={(62.33,-60.56)}] (60.5,-59.5) to[quadratic={(53.46,-55.43)}] (33.46,-55.43) to[quadratic={(18.67,-55.43)}] (18.67,-57.66);
  \draw[red,very thick] (18.67,-57.66) to[quadratic={(18.67,-58.44)}] (20.5,-59.5) to[quadratic={(27.54,-63.57)}] (47.54,-63.57) to[quadratic={(62.33,-63.57)}] (62.33,-61.34);
  \draw[red,very thick] (-17.33,-61.34) to[quadratic={(-17.33,-60.56)}] (-19.5,-59.5) to[quadratic={(-26.46,-55.43)}] (-46.46,-55.43) to[quadratic={(-61.67,-55.43)}] (-61.67,-57.66);
  \draw[red,very thick] (-61.67,-57.66) to[quadratic={(-61.67,-58.44)}] (-59.5,-59.5) to[quadratic={(-52.54,-63.57)}] (-32.54,-63.57) to[quadratic={(-17.33,-63.57)}] (-17.33,-61.34);
  \draw[black]  (18.66,-57.66) .. controls (18.66,-30.99)and(-17.66,-34.68) .. (-17.66,-61.34);
\draw[black]  (62.34,-61.34) .. controls (62.34,14.34)and(-61.34,18.66) .. (-61.34,-57.66);
\draw[black]  (62.34,58.66) .. controls (62.34,31.99)and(18.66,35.68) .. (18.66,62.34);
\draw[black]  (-17.66,58.66) .. controls (-17.66,31.99)and(-61.34,35.68) .. (-61.34,62.34);
  \draw[red,very thick] (62.33,58.34) to[quadratic={(62.33,59.56)}] (60.5,60.5) to[quadratic={(53.46,64.43)}] (33.46,64.43) to[quadratic={(18.67,64.43)}] (18.67,62.66);
  \draw[red,very thick,dotted] (18.67,62.66) to[quadratic={(18.67,61.44)}] (20.5,60.5) to[quadratic={(27.54,56.57)}] (47.54,56.57) to[quadratic={(62.33,56.57)}] (62.33,58.34);
  \draw[red,very thick] (-17.33,58.34) to[quadratic={(-17.33,59.56)}] (-19.5,60.5) to[quadratic={(-26.46,64.43)}] (-46.46,64.43) to[quadratic={(-61.67,64.43)}] (-61.67,62.66);
  \draw[red,very thick,dotted] (-61.67,62.66) to[quadratic={(-61.67,61.44)}] (-59.5,60.5) to[quadratic={(-52.54,56.57)}] (-32.54,56.57) to[quadratic={(-17.33,56.57)}] (-17.33,58.34);
}

\defTikzBox[xlen=.4pt,ylen=-.5pt]{BordFourTuU}{%
  \draw[red,very thick] (62.33,-61.34) to[quadratic={(62.33,-60.56)}] (60.5,-59.5) to[quadratic={(53.46,-55.43)}] (33.46,-55.43) to[quadratic={(18.67,-55.43)}] (18.67,-57.66);
  \draw[red,very thick] (18.67,-57.66) to[quadratic={(18.67,-58.44)}] (20.5,-59.5) to[quadratic={(27.54,-63.57)}] (47.54,-63.57) to[quadratic={(62.33,-63.57)}] (62.33,-61.34);
  \draw[red,very thick] (-17.33,-61.34) to[quadratic={(-17.33,-60.56)}] (-19.5,-59.5) to[quadratic={(-26.46,-55.43)}] (-46.46,-55.43) to[quadratic={(-61.67,-55.43)}] (-61.67,-57.66);
  \draw[red,very thick] (-61.67,-57.66) to[quadratic={(-61.67,-58.44)}] (-59.5,-59.5) to[quadratic={(-52.54,-63.57)}] (-32.54,-63.57) to[quadratic={(-17.33,-63.57)}] (-17.33,-61.34);
  \draw[black]  (62.34,-61.34) .. controls (62.34,-34.68)and(18.66,-30.99) .. (18.66,-57.66);
  \draw[black]  (-17.66,-61.34) .. controls (-17.66,-34.68)and(-61.34,-30.99) .. (-61.34,-57.66);
  \draw[black]  (62.34,58.66) .. controls (62.34,-17.66)and(-61.34,-13.34) .. (-61.34,62.34);
  \draw[black]  (18.66,62.34) .. controls (18.66,35.68)and(-17.66,31.99) .. (-17.66,58.66);
  \draw[red,very thick] (62.33,58.34) to[quadratic={(62.33,59.56)}] (60.5,60.5) to[quadratic={(53.46,64.43)}] (33.46,64.43) to[quadratic={(18.67,64.43)}] (18.67,62.66);
  \draw[red,very thick,dotted] (18.67,62.66) to[quadratic={(18.67,61.44)}] (20.5,60.5) to[quadratic={(27.54,56.57)}] (47.54,56.57) to[quadratic={(62.33,56.57)}] (62.33,58.34);
  \draw[red,very thick] (-17.33,58.34) to[quadratic={(-17.33,59.56)}] (-19.5,60.5) to[quadratic={(-26.46,64.43)}] (-46.46,64.43) to[quadratic={(-61.67,64.43)}] (-61.67,62.66);
  \draw[red,very thick,dotted] (-61.67,62.66) to[quadratic={(-61.67,61.44)}] (-59.5,60.5) to[quadratic={(-52.54,56.57)}] (-32.54,56.57) to[quadratic={(-17.33,56.57)}] (-17.33,58.34);
}

\defTikzBox[baseline=-.5ex,xlen=.3pt,ylen=-.5pt]{BordROneDeltaP}{%
  \draw[red,very thick] (-93.59,28.63) -- (66.41,28.63);
  \draw[red,very thick,dotted] (-52.46,16.43) to[quadratic={(-59.5,20.5)}] (-39.5,20.5) -- (40.5,20.5) to[quadratic={(58.99,20.5)}] (66.41,17.02);
  \draw[red,very thick] (66.41,17.02) to[quadratic={(67.01,16.74)}] (67.54,16.43) to[quadratic={(71.71,14.03)}] (66.41,13.04);
  \draw[red,very thick,dotted] (66.41,13.04) to[quadratic={(62.75,12.37)}] (54.59,12.37) -- (-25.41,12.37) to[quadratic={(-45.41,12.37)}] (-52.46,16.43);
  \draw[black] (-24.71,-17.28) .. controls (-24.71,9.391)and(11.62,13.08) .. (11.62,-13.59);
  \draw[black,thick,dotted] (-54.29,18.28) -- (-54.29,-11.37);
  \draw[black] (-54.29,-11.37) -- (-54.29,-21.72);
  \draw[black] (69.38,14.59) -- (69.38,-25.41);
  \draw[black] (-93.59,28.63) -- (-93.59,-11.37);
  \draw[black] (66.41,28.63) -- (66.41,-11.37);
  \draw[red,very thick] (-93.59,-11.37) -- (-53.59,-11.37) to[quadratic={(-33.59,-11.37)}] (-26.54,-15.43) to[quadratic={(-19.5,-19.5)}] (-39.5,-19.5) to[quadratic={(-59.5,-19.5)}] (-52.46,-23.57) to[quadratic={(-45.41,-27.63)}] (-25.41,-27.63) -- (54.59,-27.63) to[quadratic={(74.59,-27.63)}] (67.54,-23.57) to[quadratic={(60.5,-19.5)}] (40.5,-19.5) to[quadratic={(20.5,-19.5)}] (13.46,-15.43) to[quadratic={(6.41,-11.37)}] (26.41,-11.37) -- (66.41,-11.37);
}

\defTikzBox[baseline=-.5ex,xlen=.3pt,ylen=-.5pt]{BordROneDeltaN}{%
  \draw[red,very thick] (-93.59,28.63) -- (-53.59,28.63) to[quadratic={(-33.59,28.63)}] (-26.54,24.57) to[quadratic={(-24.71,23.51)}] (-24.71,22.73);
  \draw[red,very thick,dotted] (-24.71,22.73) to[quadratic={(-24.71,20.5)}] (-39.5,20.5) to[quadratic={(-59.5,20.5)}] (-52.46,16.43) to[quadratic={(-45.41,12.37)}] (-25.41,12.37) -- (-22.3,12.37);
  \draw[red,very thick] (-22.3,12.37) -- (7.225,12.37);
  \draw[red,very thick,dotted] (7.225,12.37) -- (54.59,12.37) to[quadratic={(62.75,12.37)}] (66.41,13.04);
  \draw[red,very thick] (66.41,13.04) to[quadratic={(71.71,14.03)}] (67.54,16.43) to[quadratic={(67.01,16.74)}] (66.41,17.02);
  \draw[red,very thick,dotted] (66.41,17.02) to[quadratic={(58.99,20.5)}] (40.5,20.5) to[quadratic={(20.5,20.5)}] (13.46,24.57) to[quadratic={(11.62,25.63)}] (11.62,26.41);
  \draw[red,very thick] (11.62,26.41) to[quadratic={(11.62,28.63)}] (26.41,28.63) -- (66.41,28.63);
  \draw[black] (-24.71,22.72) .. controls (-24.71,-3.942)and(11.62,-0.2563) .. (11.62,26.41);
  \draw[black,thick,dotted] (-54.29,18.28) -- (-54.29,-11.37);
  \draw[black] (-54.29,-11.37) -- (-54.29,-21.72);
  \draw[black] (69.38,14.59) -- (69.38,-25.41);
  \draw[black] (-93.59,28.63) -- (-93.59,-11.37);
  \draw[black] (66.41,28.63) -- (66.41,-11.37);
  \draw[red,very thick] (-93.59,-11.37) -- (66.41,-11.37);
  \draw[red,very thick] (-52.46,-23.57) to[quadratic={(-59.5,-19.5)}] (-39.5,-19.5) -- (40.5,-19.5) to[quadratic={(60.5,-19.5)}] (67.54,-23.57) to[quadratic={(74.59,-27.63)}] (54.59,-27.63) -- (-25.41,-27.63) to[quadratic={(-45.41,-27.63)}] (-52.46,-23.57);
}

\defTikzBox[baseline=-.5ex,xlen=.3pt,ylen=-.5pt]{BordROnePBarFstVarV}{%
  \draw[red,very thick] (-93.59,48.63) -- (-53.59,48.63) to[quadratic={(-33.59,48.63)}] (-26.54,44.57) to[quadratic={(-24.71,43.51)}] (-24.71,42.73);
  \draw[red,very thick,dotted] (-24.71,42.73) to[quadratic={(-24.71,40.5)}] (-39.5,40.5) to[quadratic={(-59.5,40.5)}] (-52.46,36.43) to[quadratic={(-45.41,32.37)}] (-25.41,32.37) -- (-22.3,32.37);
  \draw[red,very thick] (-22.3,32.37) -- (7.225,32.37);
  \draw[red,very thick,dotted] (7.225,32.37) -- (54.59,32.37) to[quadratic={(62.75,32.37)}] (66.41,33.04);
  \draw[red,very thick] (66.41,33.04) to[quadratic={(71.71,34.03)}] (67.54,36.43) to[quadratic={(67.01,36.74)}] (66.41,37.02);
  \draw[red,very thick,dotted] (66.41,37.02) to[quadratic={(58.99,40.5)}] (40.5,40.5) to[quadratic={(20.5,40.5)}] (13.46,44.57) to[quadratic={(11.62,45.63)}] (11.62,46.41);
  \draw[red,very thick] (11.62,46.41) to[quadratic={(11.62,48.63)}] (26.41,48.63) -- (66.41,48.63);
  \draw[black] (-24.71,42.72) .. controls (-24.71,16.06)and(11.62,19.74) .. (11.62,46.41);
  \draw[black,thick,dotted] (-54.29,38.28) .. controls(-54.29,1.97)and(47.59,-4.42) .. (66.41,24.62);
  \draw[black] (66.41,24.62) .. controls(68.33,27.57)and(69.38,30.9) .. (69.38,34.59);
  \draw[black] (-54.29,-41.72) .. controls(-54.29,-37.87)and(-53.14,-34.42) .. (-51.07,-31.37);
  \draw[black,thick,dotted] (-51.07,-31.37) .. controls(-34.48,-6.99)and(41.37,-8.11) .. (63.38,-31.37);
  \draw[black] (63.38,-31.37) .. controls(67.18,-35.39)and(69.38,-40.08) .. (69.38,-45.41);
  \draw[black] (30.92,32.57) .. controls+(0,-10)and+(0,15) .. (77.92,4.57) .. controls+(0,-15)and+(0,10) .. (30.92,-23.43);
  \draw[black] (30.92,10.57) .. controls+(0,5)and+(0,7) .. (54.92,4.57) .. controls+(0,-7)and+(0,-5) .. (30.92,-1.43);
  \draw[black] (-93.59,48.63) -- (-93.59,-31.37);
  \draw[black] (66.41,48.63) -- (66.41,16);
  \draw[black,thick,dotted] (66.41,16) -- (66.41,-6.8);
  \draw[black] (66.41,-6.8) -- (66.41,-31.37);
  \draw[red,very thick] (-93.59,-31.37) -- (66.41,-31.37);
  \draw[red,very thick] (-52.46,-43.57) to[quadratic={(-59.5,-39.5)}] (-39.5,-39.5) -- (40.5,-39.5) to[quadratic={(60.5,-39.5)}] (67.54,-43.57) to[quadratic={(74.59,-47.63)}] (54.59,-47.63) -- (-25.41,-47.63) to[quadratic={(-45.41,-47.63)}] (-52.46,-43.57);
}

\defTikzBox[baseline=-.5ex,xlen=.3pt,ylen=-.5pt]{BordROneNSndV}{%
  \draw[red,very thick] (-93.59,48.63) -- (66.41,48.63);
  \draw[red,very thick] (66.41,37.02) to[quadratic={(67.01,36.74)}] (67.54,36.43) to[quadratic={(71.71,34.03)}] (66.41,33.04);
  \draw[red,very thick,dotted] (66.41,33.04) to[quadratic={(62.75,32.37)}] (54.59,32.37) -- (-25.41,32.37) to[quadratic={(-45.41,32.37)}] (-52.46,36.43) to[quadratic={(-59.5,40.5)}] (-39.5,40.5) -- (40.5,40.5) to[quadratic={(58.99,40.5)}] (66.41,37.02);
  \draw[black,thick,dotted] (-54.29,38.28) .. controls(-54.29,1.97)and(47.59,-4.42) .. (66.41,24.62);
  \draw[black] (66.41,24.62) .. controls(68.33,27.57)and(69.38,30.9) .. (69.38,34.59);
  \draw[black] (-54.29,-41.72) .. controls(-54.29,-37.93)and(-53.41,-34.48) .. (-51.8,-31.37);
  \draw[black,thick,dotted] (-51.8,-31.37) ..controls(-36.54,-1.93)and(44.38,-3.19) .. (64.73,-31.37);
  \draw[black] (64.73,-31.37) .. controls(67.7,-35.48)and(69.38,-40.17) .. (69.38,-45.41);
  \draw[black] (-24.71,-37.28) .. controls +(0,20)and+(0,20) .. (11.62,-33.59);
  \draw[black] (30.92,32.57) .. controls+(0,-10)and+(0,15) .. (77.92,4.57) .. controls+(0,-15)and+(0,10) .. (30.92,-23.43);
  \draw[black] (30.92,10.57) .. controls+(0,5)and+(0,7) .. (54.92,4.57) .. controls+(0,-7)and+(0,-5) .. (30.92,-1.43);
  \draw[black] (-93.59,48.63) -- (-93.59,-31.37);
  \draw[black] (66.41,48.63) -- (66.41,16);
  \draw[black,thick,dotted] (66.41,16) -- (66.41,-6.8);
  \draw[black] (66.41,-6.8) -- (66.41,-31.37);
  \draw[red,very thick] (-93.59,-31.37) -- (-53.59,-31.37) to[quadratic={(-33.59,-31.37)}] (-26.54,-35.43) to[quadratic={(-19.5,-39.5)}] (-39.5,-39.5) to[quadratic={(-59.5,-39.5)}] (-52.46,-43.57) to[quadratic={(-45.41,-47.63)}] (-25.41,-47.63) -- (54.59,-47.63) to[quadratic={(74.59,-47.63)}] (67.54,-43.57) to[quadratic={(60.5,-39.5)}] (40.5,-39.5) to[quadratic={(20.5,-39.5)}] (13.46,-35.43) to[quadratic={(6.41,-31.37)}] (26.41,-31.37) -- (66.41,-31.37);
}

\defTikzBox[baseline=-.5ex,xlen=.3pt,ylen=-.5pt]{BordFIVId}{%
  \draw[red,very thick] (-93.59,48.63) -- (66.41,48.63);
  \draw[red,very thick] (66.41,37.02) to[quadratic={(67.01,36.74)}] (67.54,36.43) to[quadratic={(71.71,34.03)}] (66.41,33.04);
  \draw[red,very thick,dotted] (66.41,33.04) to[quadratic={(62.75,32.37)}] (54.59,32.37) -- (-25.41,32.37) to[quadratic={(-45.41,32.37)}] (-52.46,36.43) to[quadratic={(-59.5,40.5)}] (-39.5,40.5) -- (40.5,40.5) to[quadratic={(58.99,40.5)}] (66.41,37.02);
  \draw[black,thick,dotted] (-54.29,38.28) -- (-54.29,-31.37);
  \draw[black] (-54.29,-31.37) -- (-54.29,-41.72);
  \draw[black] (69.38,34.59) -- (69.38,-45.41);
  \draw[black] (-93.59,48.63) -- (-93.59,-31.37);
  \draw[black] (66.41,48.63) -- (66.41,-31.37);
  \draw[red,very thick] (-93.59,-31.37) -- (66.41,-31.37);
  \draw[red,very thick] (-52.46,-43.57) to[quadratic={(-59.5,-39.5)}] (-39.5,-39.5) -- (40.5,-39.5) to[quadratic={(60.5,-39.5)}] (67.54,-43.57) to[quadratic={(74.59,-47.63)}] (54.59,-47.63) -- (-25.41,-47.63) to[quadratic={(-45.41,-47.63)}] (-52.46,-43.57);
}

\defTikzBox[baseline=-.5ex,xlen=.3pt,ylen=-.5pt]{BordFIPhiFst}{%
  \draw[red,very thick] (-93.59,48.63) -- (66.41,48.63);
  \draw[red,very thick] (66.41,37.02) to[quadratic={(67.01,36.74)}] (67.54,36.43) to[quadratic={(71.71,34.03)}] (66.41,33.04);
  \draw[red,very thick,dotted] (66.41,33.04) to[quadratic={(62.75,32.37)}] (54.59,32.37) -- (-25.41,32.37) to[quadratic={(-45.41,32.37)}] (-52.46,36.43) to[quadratic={(-59.5,40.5)}] (-39.5,40.5) -- (40.5,40.5) to[quadratic={(58.99,40.5)}] (66.41,37.02);
  \draw[black] (21.62,3.28) .. controls+(0,20)and+(0,20) .. (-34.71,-0.41);
  \draw[black] (21.62,3.28) .. controls+(0,-20)and+(0,-20) .. (-34.71,-0.41);
  \draw[black,thick,dotted] (-54.29,38.28) -- (-54.29,-31.37);
  \draw[black] (-54.29,-31.37) -- (-54.29,-41.72);
  \draw[black] (69.38,34.59) -- (69.38,-45.41);
  \draw[black] (-93.59,48.63) -- (-93.59,-31.37);
  \draw[black] (66.41,48.63) -- (66.41,-31.37);
  \draw[red,very thick] (-93.59,-31.37) -- (66.41,-31.37);
  \draw[red,very thick] (-52.46,-43.57) to[quadratic={(-59.5,-39.5)}] (-39.5,-39.5) -- (40.5,-39.5) to[quadratic={(60.5,-39.5)}] (67.54,-43.57) to[quadratic={(74.59,-47.63)}] (54.59,-47.63) -- (-25.41,-47.63) to[quadratic={(-45.41,-47.63)}] (-52.46,-43.57);
}

\defTikzBox[baseline=-.5ex,xlen=.3pt,ylen=-.5pt]{BordFIPhiSnd}{%
  \draw[red,very thick] (-93.59,48.63) -- (66.41,48.63);
  \draw[red,very thick] (66.41,37.02) to[quadratic={(67.01,36.74)}] (67.54,36.43) to[quadratic={(71.71,34.03)}] (66.41,33.04);
  \draw[red,very thick,dotted] (66.41,33.04) to[quadratic={(62.75,32.37)}] (54.59,32.37) -- (-25.41,32.37) to[quadratic={(-45.41,32.37)}] (-52.46,36.43) to[quadratic={(-59.5,40.5)}] (-39.5,40.5) -- (40.5,40.5) to[quadratic={(58.99,40.5)}] (66.41,37.02);
  \draw[black] (3.92,32.57) .. controls+(0,-10)and+(0,15) .. (50.92,4.57) .. controls+(0,-15)and+(0,10) .. (3.92,-23.43);
  \draw[black] (3.92,10.57) .. controls+(0,5)and+(0,7) .. (27.92,4.57) .. controls+(0,-7)and+(0,-5) .. (3.92,-1.43);
  \draw[black,thick,dotted] (-54.29,38.28) -- (-54.29,-31.37);
  \draw[black] (-54.29,-31.37) -- (-54.29,-41.72);
  \draw[black] (69.38,34.59) -- (69.38,-45.41);
  \draw[black] (-93.59,48.63) -- (-93.59,-31.37);
  \draw[black] (66.41,48.63) -- (66.41,-31.37);
  \draw[red,very thick] (-93.59,-31.37) -- (66.41,-31.37);
  \draw[red,very thick] (-52.46,-43.57) to[quadratic={(-59.5,-39.5)}] (-39.5,-39.5) -- (40.5,-39.5) to[quadratic={(60.5,-39.5)}] (67.54,-43.57) to[quadratic={(74.59,-47.63)}] (54.59,-47.63) -- (-25.41,-47.63) to[quadratic={(-45.41,-47.63)}] (-52.46,-43.57);
}

\defTikzBox[baseline=-.5ex,xlen=.3pt,ylen=-.5pt]{BordFIHmtpFst}{%
  \draw[red,very thick] (-93.59,48.63) -- (66.41,48.63);
  \draw[red,very thick] (66.41,37.02) to[quadratic={(67.01,36.74)}] (67.54,36.43) to[quadratic={(71.71,34.03)}] (66.41,33.04);
  \draw[red,very thick,dotted] (66.41,33.04) to[quadratic={(62.75,32.37)}] (54.59,32.37) -- (-25.41,32.37) to[quadratic={(-45.41,32.37)}] (-52.46,36.43) to[quadratic={(-59.5,40.5)}] (-39.5,40.5) -- (40.5,40.5) to[quadratic={(58.99,40.5)}] (66.41,37.02);
  \draw[black,thick,dotted] (-54.29,38.28) .. controls(-54.29,11.05)and(47.59,5.5) .. (66.41,27.13);
  \draw[black] (66.41,27.13) .. controls(68.33,29.34)and(69.38,31.82) .. (69.38,34.59);
  \draw[black] (-54.29,-41.72) -- (-54.29,-31.37);
  \draw[black,thick,dotted] (-54.29,-31.37) .. controls(-54.29,14.59)and(47.59,14.9) .. (66.41,-15.26);
  \draw[black] (66.41,-15.26) .. controls(66.33,-18.33)and(69.38,-21.72) .. (69.38,-25.41) -- (69.38,-45.41);
  \draw[black] (11.62,-11.72) .. controls+(0,10)and+(0,10) .. (-24.71,-15.41);
    \draw[black] (11.62,-11.72) .. controls+(0,-10)and+(0,-10) .. (-24.71,-15.41);
  \draw[black] (-93.59,48.63) -- (-93.59,-31.37);
  \draw[black] (66.41,48.63) -- (66.41,-31.37);
  \draw[red,very thick] (-93.59,-31.37) -- (66.41,-31.37);
  \draw[red,very thick] (-52.46,-43.57) to[quadratic={(-59.5,-39.5)}] (-39.5,-39.5) -- (40.5,-39.5) to[quadratic={(60.5,-39.5)}] (67.54,-43.57) to[quadratic={(74.59,-47.63)}] (54.59,-47.63) -- (-25.41,-47.63) to[quadratic={(-45.41,-47.63)}] (-52.46,-43.57);
}

\defTikzBox[baseline=-.5ex,xlen=.3pt,ylen=-.5pt]{BordFIHmtpSnd}{%
  \draw[red,very thick] (-93.59,48.63) -- (66.41,48.63);
  \draw[red,very thick] (66.41,37.02) to[quadratic={(67.01,36.74)}] (67.54,36.43) to[quadratic={(71.71,34.03)}] (66.41,33.04);
  \draw[red,very thick,dotted] (66.41,33.04) to[quadratic={(62.75,32.37)}] (54.59,32.37) -- (-25.41,32.37) to[quadratic={(-45.41,32.37)}] (-52.46,36.43) to[quadratic={(-59.5,40.5)}] (-39.5,40.5) -- (40.5,40.5) to[quadratic={(58.99,40.5)}] (66.41,37.02);
  \draw[black] (-24.71,12.72) .. controls+(0,-10)and+(0,-10) .. (11.62,16.41);
  \draw[black] (-24.71,12.72) .. controls+(0,10)and+(0,10) .. (11.62,16.41);
  \draw[black,thick,dotted] (-54.29,38.28) -- (-54.29,23.28) .. controls(-54.29,-13.03)and(47.59,-19.42) .. (66.41,9.62);
  \draw[black] (66.41,9.62) .. controls(68.33,12.57)and(69.38,15.9) .. (69.38,19.59) -- (69.38,34.59);
  \draw[black] (-54.29,-41.72) .. controls(-54.29,-37.66)and(-52.02,-34.21) .. (-48.1,-31.37);
  \draw[black,thick,dotted] (-48.1,-31.37) .. controls(-28.49,-17.14)and(32.49,-17.9) .. (57.99,-31.37);
  \draw[black] (57.99,-31.37) .. controls(65.04,-35.09)and(69.38,-39.79) .. (69.38,-45.41);
  \draw[black] (-93.59,48.63) -- (-93.59,-31.37);
  \draw[black] (66.41,48.63) -- (66.41,-31.37);
  \draw[red,very thick] (-93.59,-31.37) -- (66.41,-31.37);
  \draw[red,very thick] (-52.46,-43.57) to[quadratic={(-59.5,-39.5)}] (-39.5,-39.5) -- (40.5,-39.5) to[quadratic={(60.5,-39.5)}] (67.54,-43.57) to[quadratic={(74.59,-47.63)}] (54.59,-47.63) -- (-25.41,-47.63) to[quadratic={(-45.41,-47.63)}] (-52.46,-43.57);
}

\defTikzBox[baseline=-.5ex,xlen=.3pt,ylen=-.5pt]{BordFIHmtpTrd}{%
  \draw[red,very thick] (-93.59,48.63) -- (66.41,48.63);
  \draw[red,very thick] (66.41,37.02) to[quadratic={(67.01,36.74)}] (67.54,36.43) to[quadratic={(71.71,34.03)}] (66.41,33.04);
  \draw[red,very thick,dotted] (66.41,33.04) to[quadratic={(62.75,32.37)}] (54.59,32.37) -- (-25.41,32.37) to[quadratic={(-45.41,32.37)}] (-52.46,36.43) to[quadratic={(-59.5,40.5)}] (-39.5,40.5) -- (40.5,40.5) to[quadratic={(58.99,40.5)}] (66.41,37.02);
  \draw[black,thick,dotted] (-54.29,38.28) .. controls(-54.29,1.97)and(47.59,-4.42) .. (66.41,24.62);
  \draw[black] (66.41,24.62) .. controls(68.33,27.57)and(69.38,30.9) .. (69.38,34.59);
  \draw[black] (-54.29,-41.72) .. controls(-54.29,-37.93)and(-53.41,-34.48) .. (-51.8,-31.37);
  \draw[black,thick,dotted] (-51.8,-31.37) ..controls(-36.54,-1.93)and(44.38,-3.19) .. (64.73,-31.37);
  \draw[black] (64.73,-31.37) .. controls(67.7,-35.48)and(69.38,-40.17) .. (69.38,-45.41);
  \draw[black] (3.92,32.57) .. controls+(0,-10)and+(0,15) .. (50.92,4.57) .. controls+(0,-15)and+(0,10) .. (3.92,-23.43);
  \draw[black] (3.92,10.57) .. controls+(0,5)and+(0,7) .. (27.92,4.57) .. controls+(0,-7)and+(0,-5) .. (3.92,-1.43);
  \draw[black] (-93.59,48.63) -- (-93.59,-31.37);
  \draw[black] (66.41,48.63) -- (66.41,-31.37);
  \draw[red,very thick] (-93.59,-31.37) -- (66.41,-31.37);
  \draw[red,very thick] (-52.46,-43.57) to[quadratic={(-59.5,-39.5)}] (-39.5,-39.5) -- (40.5,-39.5) to[quadratic={(60.5,-39.5)}] (67.54,-43.57) to[quadratic={(74.59,-47.63)}] (54.59,-47.63) -- (-25.41,-47.63) to[quadratic={(-45.41,-47.63)}] (-52.46,-43.57);
}

\defTikzBox[baseline=-.5ex,scale=.5]{diagTwistCruxOO}{%
  \useasboundingbox (-3,-2.2) rectangle (3,1.7);
  \node[inner sep=1,circle] (CU) at (0,1.5) {};
  \node[inner sep=1,circle] (CM) at (0,.7) {};
  \node[inner sep=1,circle] (D) at (-1.5,-1.5) {};
  \node[inner sep=5,circle] (CL) at (-.5,-1.5) {};
  \node[inner sep=5,circle] (CR) at (1.5,-1.5) {};
  \draw[red,very thick] (CU.north) .. controls+(0:.2)and+(-.5,0) .. (0.7,1.67) .. controls+(3,0)and+(2,0) .. (CR.south);
  \draw[red,very thick] (CU.north) .. controls+(180:.2)and+(.5,0) .. (-0.7,1.67) .. controls+(-3,0)and+(-1.5,0) .. (-2,-2) to[out=0,in=-90] (D.west);
  \draw[blue,very thick] (CM.south) .. controls+(0:.7)and+(90:1) .. (1.81,-0.58) to[out=-90,in=0] (CR.north);
  \draw[blue,very thick] (CM.south) .. controls+(180:.7)and+(90:1) .. (-1.81,-0.58) to[out=-90,in=90] (D.west);
  \draw[red,very thick] (CU.south) to[out=0,in=0,looseness=2] (CM.north);
  \draw[red,very thick] (CM.north) to[out=180,in=180,looseness=2] (CU.south);
  \draw[red,very thick] (D.east) to[out=-90,in=180] (CL.south) -- (CR.south);
  \draw[blue,very thick] (D.east) to[out=90,in=180] (CL.north) -- (CR.north);
}

\defTikzBox[baseline=-.5ex,scale=.5]{diagTwistCruxlO}{%
  \useasboundingbox (-3,-2.2) rectangle (3,1.7);
  \node[inner sep=1,circle] (CU) at (0,1.5) {};
  \node[inner sep=1,circle] (CM) at (0,.7) {};
  \node[inner sep=1,circle] (D) at (-1.5,-1.5) {};
  \node[inner sep=5,circle] (CL) at (-.5,-1.5) {};
  \node[inner sep=5,circle] (CR) at (1.5,-1.5) {};
  \draw[red,very thick] (CU.east) to[out=90,in=180] (0.7,1.67) .. controls+(3,0)and+(2,0) .. (CR.south);
  \draw[red,very thick] (CU.west) to[out=90,in=0] (-0.7,1.67) .. controls+(-3,0)and+(-1.5,0) .. (-2,-2) to[out=0,in=-90] (D.west);
  \draw[blue,very thick] (CM.south) .. controls+(0:.7)and+(90:1) .. (1.81,-0.58) to[out=-90,in=0] (CR.north);
  \draw[blue,very thick] (CM.south) .. controls+(180:.7)and+(90:1) .. (-1.81,-0.58) to[out=-90,in=90] (D.west);
  \draw[red,very thick] (CU.east) .. controls+(0,-.1)and+(1,0) .. (CM.north);
  \draw[red,very thick] (CU.west) .. controls+(0,-.1)and+(-1,0) .. (CM.north);
  \draw[red,very thick] (D.east) to[out=-90,in=180] (CL.south) -- (CR.south);
  \draw[blue,very thick] (D.east) to[out=90,in=180] (CL.north) -- (CR.north);
}

\defTikzBox[baseline=-.5ex,scale=.5]{diagTwistCruxOl}{%
  \useasboundingbox (-3,-2.2) rectangle (3,1.7);
  \node[inner sep=1,circle] (CU) at (0,1.5) {};
  \node[inner sep=1,circle] (CM) at (0,.7) {};
  \node[inner sep=1,circle] (D) at (-1.5,-1.5) {};
  \node[inner sep=5,circle] (CL) at (-.5,-1.5) {};
  \node[inner sep=5,circle] (CR) at (1.5,-1.5) {};
  \draw[red,very thick] (CU.north) .. controls+(0:.2)and+(-.5,0) .. (0.7,1.67) .. controls+(3,0)and+(2,0) .. (CR.south);
  \draw[red,very thick] (CU.north) .. controls+(180:.2)and+(.5,0) .. (-0.7,1.67) .. controls+(-3,0)and+(-1.5,0) .. (-2,-2) to[out=0,in=-90] (D.west);
  \draw[blue,very thick] (CM.south) ++(.2,0) to[out=180,in=-90] (CM.east);
  \draw[blue,very thick] (CM.south) ++(.2,0) .. controls+(0:.7)and+(90:1) .. (1.81,-0.58) to[out=-90,in=0] (CR.north);
  \draw[blue,very thick] (CM.south) ++(-.2,0) to[out=0,in=-90] (CM.west);
  \draw[blue,very thick] (CM.south) ++(-.2,0) .. controls+(180:.7)and+(90:1) .. (-1.81,-0.58) to[out=-90,in=90] (D.west);
  \draw[blue,very thick] (CU.south) .. controls+(1,0)and+(0,.1) .. (CM.east);
  \draw[blue,very thick] (CU.south) .. controls+(-1,0)and+(0,.1) .. (CM.west);
  \draw[red,very thick] (D.east) to[out=-90,in=180] (CL.south) -- (CR.south);
  \draw[blue,very thick] (D.east) to[out=90,in=180] (CL.north) -- (CR.north);
}

\defTikzBox[baseline=-.5ex,scale=.5]{diagTwistCruxHOOO}{%
  \useasboundingbox (-3,-2.2) rectangle (3,1.7);
  \node[inner sep=1,circle] (CU) at (0,1.5) {};
  \node[inner sep=1,circle] (CM) at (0,.7) {};
  \node[inner sep=5,circle] (D) at (-1.5,-1.5) {};
  \node[inner sep=5,circle] (CL) at (-.5,-1.5) {};
  \node[inner sep=5,circle] (CR) at (1.5,-1.5) {};
  \draw[red,very thick] (CU.north) .. controls+(0:.2)and+(-.5,0) .. (0.7,1.67) .. controls+(3,0)and+(2,0) .. (CR.south);
  \draw[red,very thick] (CU.north) .. controls+(180:.2)and+(.5,0) .. (-0.7,1.67) .. controls+(-3,0)and+(-2,0) .. (D.south);
  \draw[red,very thick] (CM.south) .. controls+(0:.7)and+(90:1) .. (1.81,-0.58) to[out=-90,in=0] (CR.north);
  \draw[red,very thick] (CM.south) .. controls+(180:.7)and+(90:1) .. (-1.81,-0.58) to[out=-90,in=180] (D.north);
  \draw[red,very thick] (CU.south) to[out=0,in=0,looseness=2] (CM.north);
  \draw[red,very thick] (CM.north) to[out=180,in=180,looseness=2] (CU.south);
  \draw[red,very thick] (D.south) to[out=0,in=180] (CL.south) -- (CR.south);
  \draw[red,very thick] (D.north) to[out=0,in=180] (CL.north) -- (CR.north);
}

\defTikzBox[baseline=-.5ex,scale=.5]{diagTwistCruxHlOO}{%
  \useasboundingbox (-3,-2.2) rectangle (3,1.7);
  \node[inner sep=1,circle] (CU) at (0,1.5) {};
  \node[inner sep=1,circle] (CM) at (0,.7) {};
  \node[inner sep=5,circle] (D) at (-1.5,-1.5) {};
  \node[inner sep=5,circle] (CL) at (-.5,-1.5) {};
  \node[inner sep=5,circle] (CR) at (1.5,-1.5) {};
  \draw[red,very thick] (CU.east) to[out=90,in=180] (0.7,1.67) .. controls+(3,0)and+(2,0) .. (CR.south);
  \draw[red,very thick] (CU.west) to[out=90,in=0] (-0.7,1.67) .. controls+(-3,0)and+(-2,0) .. (D.south);
  \draw[red,very thick] (CM.south) .. controls+(0:.7)and+(90:1) .. (1.81,-0.58) to[out=-90,in=0] (CR.north);
  \draw[red,very thick] (CM.south) .. controls+(180:.7)and+(90:1) .. (-1.81,-0.58) to[out=-90,in=180] (D.north);
  \draw[red,very thick] (CU.east) .. controls+(0,-.1)and+(1,0) .. (CM.north);
  \draw[red,very thick] (CU.west) .. controls+(0,-.1)and+(-1,0) .. (CM.north);
  \draw[red,very thick] (D.south) to[out=0,in=180] (CL.south) -- (CR.south);
  \draw[red,very thick] (D.north) to[out=0,in=180] (CL.north) -- (CR.north);
}

\defTikzBox[baseline=-.5ex,scale=.5]{diagTwistCruxHOlO}{%
  \useasboundingbox (-3,-2.2) rectangle (3,1.7);
  \node[inner sep=1,circle] (CU) at (0,1.5) {};
  \node[inner sep=1,circle] (CM) at (0,.7) {};
  \node[inner sep=5,circle] (D) at (-1.5,-1.5) {};
  \node[inner sep=5,circle] (CL) at (-.5,-1.5) {};
  \node[inner sep=5,circle] (CR) at (1.5,-1.5) {};
  \draw[red,very thick] (CU.north) .. controls+(0:.2)and+(-.5,0) .. (0.7,1.67) .. controls+(3,0)and+(2,0) .. (CR.south);
  \draw[red,very thick] (CU.north) .. controls+(180:.2)and+(.5,0) .. (-0.7,1.67) .. controls+(-3,0)and+(-2,0) .. (D.south);
  \draw[red,very thick] (CM.south) ++(.2,0) to[out=180,in=-90] (CM.east);
  \draw[red,very thick] (CM.south) ++(.2,0) .. controls+(0:.7)and+(90:1) .. (1.81,-0.58) to[out=-90,in=0] (CR.north);
  \draw[red,very thick] (CM.south) ++(-.2,0) to[out=0,in=-90] (CM.west);
  \draw[red,very thick] (CM.south) ++(-.2,0) .. controls+(180:.7)and+(90:1) .. (-1.81,-0.58) to[out=-90,in=180] (D.north);
  \draw[red,very thick] (CU.south) .. controls+(1,0)and+(0,.1) .. (CM.east);
  \draw[red,very thick] (CU.south) .. controls+(-1,0)and+(0,.1) .. (CM.west);
  \draw[red,very thick] (D.south) to[out=0,in=180] (CL.south) -- (CR.south);
  \draw[red,very thick] (D.north) to[out=0,in=180] (CL.north) -- (CR.north);
}

\defTikzBox[baseline=-.5ex,scale=.5]{diagTwistCruxHllO}{%
  \useasboundingbox (-3,-2.2) rectangle (3,1.7);
  \node[inner sep=1,circle] (CU) at (0,1.5) {};
  \node[inner sep=1,circle] (CM) at (0,.7) {};
  \node[inner sep=5,circle] (D) at (-1.5,-1.5) {};
  \node[inner sep=5,circle] (CL) at (-.5,-1.5) {};
  \node[inner sep=5,circle] (CR) at (1.5,-1.5) {};
  \draw[red,very thick] (CU.east) to[out=90,in=180] (0.7,1.67) .. controls+(3,0)and+(2,0) .. (CR.south);
  \draw[red,very thick] (CU.west) to[out=90,in=0] (-0.7,1.67) .. controls+(-3,0)and+(-2,0) .. (D.south);
  \draw[red,very thick] (CM.south) ++(.2,0) to[out=180,in=-90] (CM.east);
  \draw[red,very thick] (CM.south) ++(.2,0) .. controls+(0:.7)and+(90:1) .. (1.81,-0.58) to[out=-90,in=0] (CR.north);
  \draw[red,very thick] (CM.south) ++(-.2,0) to[out=0,in=-90] (CM.west);
  \draw[red,very thick] (CM.south) ++(-.2,0) .. controls+(180:.7)and+(90:1) .. (-1.81,-0.58) to[out=-90,in=180] (D.north);
  \draw[red,very thick] (CU.east) to[out=-90,in=90,looseness=1.5] (.5,1.1) to[out=-90,in=90,looseness=1.5] (CM.east);
  \draw[red,very thick] (CU.west) to[out=-90,in=90,looseness=1.5] (-.5,1.1) to[out=-90,in=90,looseness=1.5] (CM.west);
  \draw[red,very thick] (D.south) to[out=0,in=180] (CL.south) -- (CR.south);
  \draw[red,very thick] (D.north) to[out=0,in=180] (CL.north) -- (CR.north);
}

\theoremstyle{plain}
\newtheorem{mainthm}{Main~Theorem}
\newtheorem{theorem}{Theorem}[section]
\newtheorem{proposition}[theorem]{Proposition}
\newtheorem{corollary}[theorem]{Corollary}
\newtheorem*{corollary*}{Corollary}
\newtheorem{lemma}[theorem]{Lemma}

\theoremstyle{definition}
\newtheorem{definition}[theorem]{Definition}

\theoremstyle{remark}
\newtheorem{example}[theorem]{Example}
\newtheorem{remark}[theorem]{Remark}
\newtheorem*{notation}{Notation}

\numberwithin{equation}{section}
\numberwithin{figure}{section}
\numberwithin{table}{section}

\title{Decomposition of the first Vassiliev derivative of Khovanov homology and its application}
\date{\today}
\author{Jun Yoshida}

\begin{document}
\maketitle

\begin{abstract}
Khovanov homology extends to singular links via a categorified analogue of Vassiliev skein relation.
In view of Vassiliev theory, the extended Khovanov homology can be seen as Vassiliev derivatives of Khovanov homology.
In this paper, we develop a new method to compute the first derivative.
Namely, we introduce a complex, called a crux complex, and prove that the Khovanov homologies of singular links with unique double points are homotopic to cofibers of endomorphisms on crux complexes.
Since crux complexes are actually small for some links, the result enables a direct computation of the first derivative of Khovanov homology.
Furthermore, it together with a categorified Vassiliev skein relation provides a brand-new method for the computation of Khovanov homology.
In fact, we apply the result to determine the Khovanov complexes of all twist knots in a universal way.
\end{abstract}

\tableofcontents

\section*{Introduction}
\label{sec:intro}
\addcontentsline{toc}{section}{Introduction}

In this paper, we compute Khovanov homology on singular links which was introduced by Ito and the author~\cite{ItoYoshida2020UKH}.
Namely, it is actually the homotopy cofiber of a morphism, called the \emph{genus-one morphism}, that realizes crossing change in the spirit of Vassiliev theory, so we regard Khovanov homology on singular links as \emph{Vassiliev derivatives} of Khovanov homology.
The goal is to describe the first Vassiliev derivative in terms of smaller complexes.
This result, together with a \emph{categorified Vassiliev skein relation}, will enable a direct computation of Khovanov homology.
In fact, as an application, we determine the Khovanov complexes of all twist knots in a universal way.

Crossing change is a very basic operation in the knot theory.
For example, it is a major problem to determine the minimum crossing number and the unknotting number for each knot or link.
We note that there are two aspects toward crossing change.
First, it is an operation in link cobordism theory; namely, it is realized as the following sequence of cobordisms:
\begin{equation}
\label{eq:intro:crossing-change-cob}
\diagCrossNegUp{}=
\begin{tikzpicture}[baseline=-.5ex]
\node[circle,inner sep=2] (L) at (-.2,0) {};
\node[circle,inner sep=2] (R) at (.2,0) {};
\draw[red,very thick,-stealth] (-.6,-.6) -- (L) (L) to[out=56,in=180] (R.north) to[out=0,in=236] (.6,.6);
\draw[red,very thick,-stealth] (.6,-.6) to[out=124,in=0] (R.south) to[out=180,in=-56] (L.center) -- (-.6,.6);
\end{tikzpicture}
\xrightarrow{\text{saddle}}
\begin{tikzpicture}[baseline=-.5ex]
\node[circle,inner sep=2] (L) at (-.2,0) {};
\node[circle,inner sep=1] (R) at (.2,0) {};
\draw[red,very thick,-stealth] (-.6,-.6) -- (L) (L) to[out=56,in=90,looseness=3] (R.west) to[out=-90,in=-56,looseness=3] (L.center) -- (-.6,.6);
\draw[red,very thick,-stealth] (.6,-.6) to[out=124,in=-90] (R.east) to[out=90,in=236] (.6,.6);
\end{tikzpicture}
\xrightarrow{\simeq}
\diagSmoothUp
\xrightarrow{\simeq}
\begin{tikzpicture}[baseline=-.5ex]
\node[circle,inner sep=2] (L) at (-.2,0) {};
\node[circle,inner sep=1] (R) at (.2,0) {};
\draw[red,very thick,-stealth] (-.6,-.6) -- (L.center) to[out=56,in=90,looseness=3] (R.west) to[out=-90,in=-56,looseness=3] (L) (L) -- (-.6,.6);
\draw[red,very thick,-stealth] (.6,-.6) to[out=124,in=-90] (R.east) to[out=90,in=236] (.6,.6);
\end{tikzpicture}
\xrightarrow{\text{saddle}}
\begin{tikzpicture}[baseline=-.5ex]
\node[circle,inner sep=2] (L) at (-.2,0) {};
\node[circle,inner sep=2] (R) at (.2,0) {};
\draw[red,very thick,-stealth] (-.6,-.6) -- (L.center) to[out=56,in=180] (R.north) to[out=0,in=236] (.6,.6);
\draw[red,very thick,-stealth] (.6,-.6) to[out=124,in=0] (R.south) to[out=180,in=-56] (L) (L) -- (-.6,.6);
\end{tikzpicture}
=\diagCrossPosUp
\quad.
\end{equation}
As \eqref{eq:intro:crossing-change-cob} compares three local diagrams, it is involved with skein relations of polynomial link invariants; such as the Jones, HOMFLY-PT, and Alexander-Conway polynomials.
This viewpoint is popular especially in $3$-dimensional topology.
Another aspect comes from Vassiliev's study of the space of knots $\mathcal K=\mathrm{Emb}(S^1,\mathbb R^3)$ \cite{Vassiliev1990}.
He investigated the cohomology groups of $\mathcal K$ in terms of the \emph{discriminant} $\Sigma\subset C^\infty(S^1,\mathbb R^3)$.
From his point of view, crossing change is a path in the space of smooth maps $S^1\to\mathbb R^3$ that passes through $\Sigma$.
Identifying knot invariants with $0$-th cohomology classes of $\mathcal K$, he then defined a class of knot invariants, which are nowadays called \emph{Vassiliev invariants}, using the spectral sequence associated to a stratification of $\Sigma$.
Birman and Lin \cite{Birman1993,BirmanLin1993} characterized them more explicitly in terms of crossing change.
Indeed, for an $A$-valued knot invariant $v:\pi_0\mathcal K\to A$, and for a singular knot $K$ with exactly $r$ transverse double points, define $v^{(r)}(K)\in A$ inductively by $v^{(0)}=v$ and
\begin{equation}
\label{eq:intro:Vassiliev-skein}
v^{(r+1)}\left(\diagSingUp\right)
= v^{(r)}\left(\diagCrossPosUp\right)
- v^{(r)}\left(\diagCrossNegUp\right)
\quad.
\end{equation}
We call $v^{(r)}$ the \emph{$r$-th Vassiliev derivative} of $v$.
It turns out that $v^{(r)}$ defines an invariant for singular links.
Then, $v$ is a \emph{Vassiliev invariant of order $r$} if and only if $v^{(r)}\equiv 0$ and $v^{(r+1)}\not\equiv 0$.
The equation~\eqref{eq:intro:Vassiliev-skein} is called \emph{Vassiliev skein relation}.

As for knot/link homologies, crossing change has been discussed mainly from the first aspect.
For example, according to Clark, Morrison, and Walker \cite{ClarkMorrisonWalker2009}, Khovanov homology \cite{Khovanov2000} forms a functor out of the category of links and link cobordisms in $\mathbb R^4$.
This implies that the sequence~\eqref{eq:intro:crossing-change-cob} gives rise to a map on Khovanov homology groups.
More explicitly, unwinding the definition of the map associated with the Reidemeister move of type I, one can see that it agrees with the following composition:
\begin{equation}
\label{eq:intro:Kh-crossing-change-cob}
\mathit{Kh}\left(\diagCrossNegUp\right)
\to \mathit{Kh}\left(\diagSmoothH\right)
\to \mathit{Kh}\left(\diagCrossPosUp\right)
\quad.
\end{equation}
Hedden and Watson \cite{HeddenWatson2018} compute the homotopy cofiber of \eqref{eq:intro:Kh-crossing-change-cob} to obtain a categorified version of Jones skein relation.

In contrast to their work, Ito and the author constructed another morphism \cite{ItoYoshida2020UKH}:
\[
\widehat\Phi:\mathit{Kh}\left(\diagCrossNegUp\right)
\to \mathit{Kh}\left(\diagCrossPosUp\right)
\quad.
\]
As it is made from cobordisms of genus $1$, we call this the \emph{genus-one morphisms}.
In the spirit of Vassiliev theory, we showed the following results:

\begin{theorem}[\cite{ItoYoshida2020UKH}]%
\label{theo:UKH-singular}
For every singular link diagram $D$, there is a bigraded abelian group $\{\mathit{Kh}^{i,j}(D)\}_{i,j}$ with the following properties:
\begin{enumerate}[label=\upshape(\arabic*)]
  \item if $D$ has no double point, then $\mathit{Kh}^{i,j}(D)$ agrees with the Khovanov homology of $D$;
  \item there is a long exact sequence of the following form:
\begin{equation}
\label{eq:intro:cat-Vassiliev-skein}
\begin{tikzcd}[column sep=1em]
\cdots \ar[r]
& \mathit{Kh}^{i,j}\left(\diagCrossNegUp\right) \ar[r,"\widehat\Phi"]
& \mathit{Kh}^{i,j}\left(\diagCrossPosUp\right) \ar[r] \ar[d,phantom,""{coordinate,name=C}]
& \mathit{Kh}^{i,j}\left(\diagSingUp\right) \ar[dll,rounded corners=1em,to path={%
  -- ([xshift=1.5em]\tikztostart.east)
  |- (C)
  -| ([xshift=-1.5em]\tikztotarget.west)
  -- (\tikztotarget)}] & \\
& \mathit{Kh}^{i+1,j}\left(\diagCrossNegUp\right) \ar[r,"\widehat\Phi"]
& \mathit{Kh}^{i+1,j}\left(\diagCrossPosUp\right) \ar[r]
& \mathit{Kh}^{i+1,j}\left(\diagSingUp\right) \ar[r]
& \cdots
\end{tikzcd}
\quad.
\end{equation}
  \item each group $\mathit{Kh}^{i,j}(D)$ is invariant under moves of singular links; in other words, it defines an invariant for singular links.
\end{enumerate}
\end{theorem}

Notice that the sequence~\eqref{eq:intro:cat-Vassiliev-skein} yields Vassiliev skein relation~\eqref{eq:intro:Vassiliev-skein} on Euler characteristics.
In other words, \cref{theo:UKH-singular} extends Khovanov homology to singular links with a \emph{categorified Vassiliev skein relation}~\eqref{eq:intro:cat-Vassiliev-skein}.
As Khovanov homology categorifies the Jones polynomial, using the same terminologies as in Vassiliev theory, we think of Khovanov homology $\mathit{Kh}(\blank)$ on singular links as a \emph{Vassiliev derivative of Khovanov homology}.

The main theme of the paper is especially the first derivative; that is, Khovanov homology of singular links with unique double points.
In order to compute it, we reformulate our extension of Khovanov homology in terms of \emph{multi-fold complexes}, which are recursive analogues of double complexes.
Namely, using Bar-Natan's formalism \cite{BarNatan2005}, for each singular tangle diagram $D$, we construct a multi-fold complex $\operatorname{Sm}(D)$ in Bar-Natan' ``category of pictures'' with the following rows:
\begin{itemize}
  \item for each crossing,
\begin{equation}
\label{eq:intro:saddle}
\diagSmoothH \xrightarrow{\text{saddle}} \diagSmoothV
\quad;
\end{equation}
  \item for each double point,
\begin{equation}
\label{eq:intro:dPhid}
\diagSmoothH
\xrightarrow{\text{saddle}} \diagSmoothUp
\xrightarrow{\Phi} \diagSmoothUp
\xrightarrow{\text{saddle}} \diagSmoothH
\quad,
\end{equation}
here $\Phi$ is the morphism that appears in the definition of the genus-one morphism.
\end{itemize}
It will turn out that the total complex $\operatorname{Tot}(\operatorname{Sm}(D))$ agrees with the complex $\dblBrac{D}$ defined in \cite{ItoYoshida2020UKH}.
We note that, since the morphism \eqref{eq:intro:saddle} exactly forms the cubes in Khovanov's original definition, this construction is a natural extension of it.

A key observation is that, if $D$ is a singular link diagram, then the first and the last morphisms in the sequence~\eqref{eq:intro:dPhid} admit the kernel and the cokernel.
More precisely, after a discussion on the Frobenius algebra structure on $S^1$ in the picture category, the kernel and the cokernel are isomorphic with each other and described in terms of a twisted $S^1$-module structure on the interval $I=[0,1]$.
This observation leads us to define a multi-fold complex $\operatorname{Crx}(D)$, called the \emph{crux complex} of $D$, so that whose total complex $\dblBrac{D}_{\mathsf{crx}}$ fits into the following exact sequence:
\begin{equation}
\label{eq:intro:crux-dcplx}
0
\to\dblBrac{D}_{\mathsf{crx}}
\to \dblBrac*{\diagSmoothH}
\to \dblBrac*{\diagSmoothUp}
\to \dblBrac*{\diagSmoothUp}
\to \dblBrac*{\diagSmoothH}
\to \dblBrac{D}_{\mathsf{crx}}
\to 0
\quad.
\end{equation}
By computing the spectral sequence of the double complex explicitly, we will prove the following theorem.

\begin{mainthm}[\cref{theo:Kh-crux}]
\label{main:crux}
There is a morphism $\Xi:\dblBrac{D}_{\mathsf{crx}}[2]\to\dblBrac{D}_{\mathsf{crx}}[-2]$ together with a chain homotopy equivalence
\[
\dblBrac{D}\simeq\operatorname{Cone}(\Xi)
\quad.
\]
\end{mainthm}

Applying the TQFT $Z_{h,t}$ associated with the Frobenius algebra $C_{h,t}$ in \cite{LaudaPfeiffer2009} (see also~\cite{Khovanov2006}), we further obtain the following.

\begin{corollary*}[\cref{cor:crux-homology-longex}]
There is a long exact sequence
\begin{equation}
\label{eq:intro:crux-longex}
\cdots
\to H^{i-2}Z_{h,t}\dblBrac{D}_{\mathsf{crx}}
\xrightarrow{\Xi_\ast} H^{i+2}Z_{h,t}\dblBrac{D}_{\mathsf{crx}}
\to H^iZ_{h,t}\dblBrac{D}
\to H^{i-1}Z_{h,t}\dblBrac{D}_{\mathsf{crx}}
\xrightarrow{\Xi_\ast}\cdots
\quad.
\end{equation}
\end{corollary*}

Specifically, since the morphism $\Xi$ is homogeneous with respect to \emph{Euler grading}, the sequence~\eqref{eq:intro:crux-longex} can be described degreewisely (\cref{cor:crux-longex-gr}).

A major advantage of \cref{main:crux} is that the complex $\dblBrac{D}_{\mathsf{crx}}$ is in general smaller than $\dblBrac{D}$; the rank does not exceed the quarter of the latter.
This feature enables us to compute directly the first Vassiliev derivative of Khovanov homology.
Combining it with a categorified Vassiliev skein relation~\eqref{eq:intro:cat-Vassiliev-skein}, we obtain a brand-new method for computation of Khovanov homology.
In fact, we apply it to determine the Khovanov complexes of twist knots.

\begin{mainthm}
\label{main:twistknot}
There is a chain homotopy equivalence
\[
\dblBrac*{%
  \begin{tikzpicture}[scale=.5,baseline=-2ex]
  \useasboundingbox (-3.5,-4) rectangle (3.5,2);
  \node[inner sep=2,circle] (CU) at (0,1.5) {};
  \node[inner sep=2,circle] (CM) at (0,.7) {};
  \node[inner sep=1.5,circle] (D) at (-1.5,-1.5) {};
  \node[inner sep=1.5,circle] (CL) at (-.5,-1.5) {};
  \node[inner sep=1.5,circle] (CR) at (1.5,-1.5) {};
  \draw[red,very thick] (CU) .. controls+(30:.2)and+(-.5,0) .. (0.7,1.67) .. controls+(3,0)and+(1.5,0) .. (2,-2) to[out=180,in=-60] (CR.center) -- +(120:.4);
  \draw[red,very thick,-stealth-] (CU.center) .. controls+(150:.2)and+(.5,0) .. (-0.7,1.67) .. controls+(-3,0)and+(-1.5,0) .. (-2,-2) to[out=0,in=-120] (D);
  \draw[red,very thick] (CM.center) .. controls+(-30:.5)and+(60:2) .. (CR) -- +(-120:.4);
  \draw[red,very thick,-stealth-] (D.center) .. controls+(120:2)and+(210:.5) .. (CM);
  \draw[red,very thick] (CU.center) to[out=-30,in=30] (CM);
  \draw[red,very thick] (CM.center) to[out=150,in=210] (CU);
  \draw[red,very thick] (D) to[out=60,in=120,looseness=2] (CL) -- +(-60:.4);
  \draw[red,very thick] (D.center) to[out=-60,in=240,looseness=2] (CL);
  \draw[red,very thick] (CL) -- +(60:.4);
  \draw[red,very thick,dotted] (CL) ++(60:.4) to[out=60,in=120] +(.5,0);
  \draw[red,very thick,dotted] (CL) ++(-60:.4) to[out=-60,in=240] +(.5,0);
  \draw[red,very thick,dotted] (CR) ++(120:.4) to[out=120,in=60] +(-.5,0);
  \draw[red,very thick,dotted] (CR) ++(240:.4) to[out=240,in=-60] +(-.5,0);
  \node at (.5,-1.5) {$\cdots$};
  \node[below] at (0,-2) {$\underbrace{\rule{4em}{0pt}}_{\text{$r$ negative crossings}}$};
  \end{tikzpicture}
}
\simeq \begin{dcases*}
\dblBrac{\text{unknot}}\oplus\bigoplus_{i=1}^{r/2}C(2i-1) & if $r$ is even,\\
\dblBrac{\text{trefoil}}\oplus\bigoplus_{i=1}^{(r-1)/2}C(2i) & if $r$ is odd,
\end{dcases*}
\]
here $C(k)$ is the complex of the form
\[
\cdots
\to 0
\to S^1
\xrightarrow{\mu\Delta} S^1
\to \overset{-k-(-1)^k}{\rule{0pt}{2.5ex}0}
\to S^1
\xrightarrow{\mu\Delta} S^1
\to 0 \to \cdots
\quad.
\]
\end{mainthm}

Another application of \cref{main:crux} is the next result.

\begin{mainthm}[\cref{theo:reducible}]
\label{main:reducible}
The genus-one morphism $\widehat\Phi$ is a chain homotopy equivalence on reducible crossings.
\end{mainthm}

Note that \cref{main:reducible} completely distinguishes $\widehat\Phi$ from the link-cobordism theoretic crossing change~\eqref{eq:intro:crossing-change-cob}.
Indeed, the morphism \eqref{eq:intro:crossing-change-cob} is rarely an isomorphism as it has a non-zero quantum degree.
For this reason, we expect that the morphism $\widehat\Phi$ would exhibit Khovanov homology in the context of Vassiliev theory.

The structure of the paper is as follows.
First, in \cref{sec:UKH-mcplx}, after a review on Bar-Natan's picture category, we define Khovanov complex for singular tangles in terms of multi-fold complexes; here, Khovanov's sign convention on cubes is generalized to multi-fold complexes.
Thanks to this formulation, Viro's exact sequence and the categorified Vassiliev skein relation~\eqref{eq:intro:cat-Vassiliev-skein} immediately follows.
Then, in \cref{sec:absex}, we provide basic tools to compute the spectral sequence for \eqref{eq:intro:crux-dcplx}.
The Frobenius algebra $S^1$ is investigated in \cref{sec:fstder}.
Using this result, we define the crux complexes for singular links with unique double points and prove \cref{main:crux} in a more precise form.
Finally, we apply \cref{main:crux} to some link diagrams; \cref{main:twistknot} and \cref{main:reducible} are proved here.

\subsection*{Acknowledgment}
I am grateful to Professor Noboru~Ito for exciting discussions and his kind supports.
I also thank Professor Toshitake~Kohno for his comments and encouragement.
Professor Louis H.~Kauffman gave us constructive feedbacks in MSCS Quantum Topology Seminar.
This work was supported by JSPS KAKENHI Grant Number JP20K03604.

\section{Khovanov complex via multi-fold complexes}
\label{sec:UKH-mcplx}

In this first section, we construct Khovanov complexes of singular tangle diagrams in terms of multi-fold complexes.
We use Bar-Natan's formalism \cite{BarNatan2005}; for a base ring $k$, we define a $k$-linear additive category $\Cob(Y_0,Y_1)$ of cobordisms between oriented compact $0$-manifolds $Y_0$ and $Y_1$.
Khovanov complexes are constructed as complexes in $\Cob(Y_0,Y_1)$.
The key is that, instead of constructing complexes directly, we construct multi-fold complexes in the same spirit as Khovanov's cubes of smoothings \cite{Khovanov2000}.

\subsection{The modules of weighted signs}
\label{sec:UKH-mcplx:modst}

Before discussing multi-fold complexes, we first prepare a technical notion.
Recall that Khovanov \cite{Khovanov2000} defined a module $E_S$ for every finite set $S$ to establish a functorial one-to-one correspondence between commutative cubes and skew-commutative ones.
In this subsection, we generalize the module to ``weighted'' sets to take the total complexes of multi-fold complexes in \cref{sec:UKH-mulcplx:mult-cplx}.

For a finite set $S$, we denote by $\operatorname{Ord}(S)$ the set of total orders on $S$.
We represent an element of $\operatorname{Ord}(S)$ as a word $\vec a=a_1\dots a_n$ in $S$ when it corresponds to the total order $S=\{a_1<\dots<a_n\}$.
Let us denote by $\mathbb F_2$ the field of two elements.
Then, for each map $\alpha:S\to\mathbb F_2$, we define a $k$-module $E_S(\alpha)$ as the quotient of the free $k$-module $k\cdot\operatorname{Ord}(S)$ by the relation
\[
a_1\dots a_{i-1}a_i a_{i+1} a_{i+2}\dots a_n
\sim (-1)^{\alpha(a_i)\alpha(a_{i+1})}a_1\dots a_{i-1}a_{i+1}a_ia_{i+2}\dots a_n
\quad.
\]
Hence, $E_S(\alpha)$ is a free $k$-module of rank $1$.
We call $E_S(\alpha)$ the \emph{module of $\alpha$-weighted signs} on $S$.
For $\vec a=a_1\dots a_n\in\operatorname{Ord}(S)$, we denote by $[\vec a]\in E_S(\alpha)$ the element represented by $\vec a$.

\begin{notation}
We often identify a map $\alpha:S\to\mathbb Z$ with an element of the free $\mathbb F_2$-module generated by $S$ by the map $\mathbb F_2^S\to \mathbb F_2S;\ \alpha\mapsto \sum_{a\in S}\alpha(a)\cdot a$.
Specifically, for each element $a\in S$, we write $\alpha+ a$ the map given by
\[
[\alpha+ a](a')=
\begin{cases}
\alpha(a')+1 & a'=a\ ,\\
\alpha(a') & a'\neq a\ .
\end{cases}
\]
\end{notation}

We introduce some elementary operations on modules of weight signs.
Consider the $k$-homomorphism $k\cdot\operatorname{Ord}(S)\to k\cdot\operatorname{Ord}(S)$ given by
\[
a'_1\dots a'_{i-1}a a'_{i+1}\dots a'_n
\mapsto (-1)^{\alpha(a'_1)+\dots+\alpha(a'_{i-1})}a'_1\dots a'_{i-1}a a'_{i+1}\dots a'_n
\quad.
\]
An easy calculation shows that it induces a $k$-homomorphism $\lambda_a:E_S(\alpha)\to E_S(\alpha+a)$.
Specifically, we have
\begin{equation}
\label{eq:lambda-left}
\lambda_a([a\vec a'])=[a\vec a']
\quad.
\end{equation}
Hence, we obtain $\lambda_a^2=\mathrm{id}$.
Similarly the $k$-homomorphism $k\cdot\operatorname{Ord}(S)\to k\cdot\operatorname{Ord}(S)$ given by
\[
a'_1\dots a'_{i-1}a a'_{i+1}\dots a'_n
\mapsto (-1)^{\alpha(a'_{i+1})+\dots+\alpha(a'_n)}a'_1\dots a'_{i-1}a a'_{i+1}\dots a'_n
\]
induces a $k$-homomorphism $\rho_a:E_S(\alpha)\to E_S(\alpha+a)$.
In contrast to the equations~\eqref{eq:lambda-left}, we have
\begin{equation}
\label{eq:rho-right}
\rho_a([\vec a'a])=[\vec a'a]
\end{equation}
so we also have $\rho_a^2=\mathrm{id}$.
If we write $|\alpha|\coloneqq\sum_{a\in S}\alpha(a)$, then we have
\[
\begin{gathered}
\lambda_a=(-1)^{|\alpha|-\alpha(a)}\rho_a:E_S(\alpha)\to E_S(\alpha +a)
\quad.
\end{gathered}
\]

\begin{lemma}
\label{lem:rho-lambda-com}
In the situation above, let $a,b\in S$ be two distinct elements.
Then, the following equations hold in the module of $k$-homomorphisms $E_S(\alpha)\to E_S(\alpha+a+b)$:
\[
\lambda_a\lambda_b=-\lambda_b\lambda_a
\ ,\quad
\rho_a\rho_b=-\rho_b\rho_a
\ ,\quad
\lambda_a\rho_b=\rho_b\lambda_a
\quad.
\]
\end{lemma}

In contrast to the highly technical definition of the $k$-module $E_S(\alpha)$, one can realize it in a more obvious way using the next lemma.

\begin{lemma}
\label{lem:modsign-alt}
Let $S$ be a finite set and $\alpha:S\to\mathbb F_2$ a map.
If $S=\{a_1<\dots<a_n\}$ is a total ordering on $S$, say $\vec a\coloneqq a_1\dots a_n$, then the inclusion $\{\vec a\}\hookrightarrow\operatorname{Ord}(S)$ induces an isomorphism $k\cong E_S(\alpha)$.
Moreover, with respect to this isomorphism, for each $a\in S$, the diagrams below commute:
\[
\begin{gathered}
\begin{tikzcd}[column sep=3em]
k \ar[r,"{(-1)^{\sum_{a'<a}\alpha(a')}}"] \ar[d,"\cong"'] & k \ar[d,"\cong"] \\
E_S(\alpha) \ar[r,"\lambda_a"] & E_S(\alpha+a)
\end{tikzcd}
\ ,\quad
\begin{tikzcd}[column sep=3em]
k \ar[r,"{(-1)^{\sum_{a'>a}\alpha(a')}}"] \ar[d,"\cong"'] & k \ar[d,"\cong"] \\
E_S(\alpha) \ar[r,"\rho_a"] & E_S(\alpha+a)
\end{tikzcd}
\quad.
\end{gathered}
\]
\end{lemma}

In the following sections, we consider \emph{tensor product} with $E_S(\alpha)$ in arbitrary $k$-linear category in the sense in \cite{Kel05}.
Recall that, if $A$ is a $k$-linear category, then for an object $X\in\mathcal A$ and a $k$-module $M$, $X\otimes M$ is defines as an object equipped with an isomorphism of $k$-modules
\begin{equation}
\label{eq:tensor-universal}
\mathcal A(X\otimes M,Y)
\cong \mathrm{Hom}_k(M,\mathcal A(X,Y))
\end{equation}
which is natural with respect to $Y\in\mathcal A$.
Although such an object $X\otimes M$ may or may not exist, we call $X\otimes M$ the \emph{tensor product} of $X$ with $M$ if it does.
Specifically, it always exists if $M$ is free of rank $1$.
Furthermore, if we denote by $\mathbf{Free}_k^1\subset\mathbf{Mod}_k$ the full subcategory of free $k$-modules of rank $1$, then the assignment $(X,M)\mapsto X\otimes M$ can be realized as an essentially unique $k$-bilinear functor such that \eqref{eq:tensor-universal} is also natural with respect to $X$ and $M$.

\begin{corollary}
\label{cor:tensor-ES-alt}
Let $\mathcal A$ be a $k$-linear category, and let $S$ be a finite totally-ordered set.
Then, there is a natural isomorphism $\mathrm{Id}_{\mathcal A}\cong(\blank)\otimes E_S(\alpha)$ for each $\alpha:S\to\mathbb Z$ such that the diagrams below commute:
\[
\begin{gathered}
\begin{tikzcd}
X \ar[r,"{(-1)^{\sum_{a'<a}\alpha(a')}\mathrm{id}_X}"] \ar[d,"\cong"'] & X \ar[d,"\cong"] \\
X\otimes E_S(\alpha) \ar[r,"X\otimes \lambda_a"] & X\otimes E_S(\alpha+a)
\end{tikzcd}
\ ,\quad
\begin{tikzcd}
X \ar[r,"{(-1)^{\sum_{a'>a}\alpha(a')}\mathrm{id}_X}"] \ar[d,"\cong"'] & X \ar[d,"\cong"] \\
X\otimes E_S(\alpha) \ar[r,"X\otimes \rho_a"] & X\otimes E_S(\alpha+a)
\end{tikzcd}
\quad.
\end{gathered}
\]
\end{corollary}

\begin{remark}
In the following sections, we mainly consider maps $\alpha:S\to\mathbb Z$ instead of $S\to\mathbb F_2$.
In this case, by abuse of notation, we also denote by $\alpha$ the composition of $\alpha$ with the canonical projection $\mathbb Z\to\mathbb F_2$, which is a ring homomorphism.
In this point of view, the operations $\lambda_a$ and $\rho_a$ define both maps $E_S(\alpha)\to E_S(\alpha+a)$ and $E_S(\alpha)\to E_S(\alpha-a)$ since $E_S(\alpha+a)=E_S(\alpha-a)$ in the convention.
\end{remark}

\subsection{Multi-fold complexes}
\label{sec:UKH-mulcplx:mult-cplx}

Throughout the section, we fix a $k$-linear category $\mathcal A$.

\begin{definition}
Let $S$ be a finite set.
Then, an \emph{$S$-fold complex} in $\mathcal A$ is a family $X^\bullet=\{X^\alpha\}_\alpha$ of objects of $\mathcal A$ indexed by maps $\alpha:S\to\mathbb Z$ equipped with a morphism $d_a^\alpha:X^\alpha\to X^{\alpha+a}$ for each $a\in S$ and $\alpha:S\to\mathbb Z$ such that
\begin{equation}
\label{eq:multcplx-diff}
d_a^{\alpha+a}\circ d_a^\alpha=0
\ ,\quad
d_a^{\alpha+b}\circ d_b^\alpha=d_b^{\alpha+a}\circ d_a^\alpha\quad \text{($a\neq b$)}
\quad.
\end{equation}
We call the morphisms $d_a^\alpha$ the \emph{differential} on $X^\bullet$.
\end{definition}

If $X^\bullet$ is an $S$-fold complex, then we often omit the superscript of its differentials and just write $d_a=d_a^\alpha$.
Hence, the equations~\eqref{eq:multcplx-diff} are written as $d_a d_a=0$ and $d_a d_b=d_b d_a$ respectively.

\begin{definition}
Let $S$ be a finite set, and let $X^\bullet$ and $Y^\bullet$ be $S$-fold complexes in $\mathcal A$.
Then, a \emph{morphism of $S$-fold complexes} is a family of morphisms $f^\alpha:X^\alpha\to Y^\alpha\in\mathcal A$ such that
\[
d^\alpha_a f^\alpha=f^{\alpha+a}d_a^\alpha:X^\alpha\to Y^{\alpha+a}
\]
for each $a\in S$ and $\alpha:S\to\mathbb Z$.
\end{definition}

As with the case of differentials, we also omit the superscript of morphisms of $S$-fold complexes unless there is danger of confusion.

Since morphisms of $S$-fold complexes compose in an obvious way, they form a category, which we denote by $\mathbf{MCh}_S(\mathcal A)$.
The category $\mathbf{MCh}_S(\mathcal A)$ is canonically $k$-linear, and it is additive (resp.~abelian) if $\mathcal A$ is so.

\begin{example}
\label{ex:0fold-cplx}
There is a canonical identification $\mathbf{MCh}_\varnothing(\mathcal A)\cong\mathcal A$.
\end{example}

\begin{example}
\label{ex:1fold-cplx}
If $S=\{a\}$ is a singleton, we have $\mathbf{MCh}_{\{a\}}(\mathcal A)\simeq \mathbf{Ch}(\mathcal A)$ the category of chain complexes and morphisms of them.
\end{example}

\begin{example}
\label{ex:mulfold-cplx-disj}
For two finite sets $S$ and $T$, there is an equivalence of categories
\[
\mathbf{MCh}_{S\amalg T}(\mathcal A)
\simeq \mathbf{MCh}_S(\mathbf{MCh}_T(\mathcal A))
\quad.
\]
In particular, in view of \cref{ex:1fold-cplx}, we have $\mathbf{MCh}_{S\amalg\{a\}}(\mathcal A)\cong\mathbf{Ch}(\mathbf{MCh}_S(\mathcal A))$, which is the reason of the terminology.
\end{example}

An $S$-fold complex $X^\bullet$ is said to be \emph{bounded} if $X^\alpha=0$ for all but finitely many indices $\alpha$.
In case $\mathcal A$ is additive as a $k$-linear category, we assign to a bounded $S$-fold complex $X^\bullet$ a chain complex $\operatorname{Tot}(X^\bullet)$ as follows: as a graded object, it is given by
\[
\operatorname{Tot}(X^\bullet)^i
\coloneqq \bigoplus_{|\alpha|=i} X^\alpha\otimes E_S(\alpha)
\quad,
\]
here we use the tensor product defined in \eqref{eq:tensor-universal} (see also \cref{cor:tensor-ES-alt}).
The differential $d^i:\operatorname{Tot}(X^\bullet)^i\to \operatorname{Tot}(X^\bullet)^{i+1}$ is then the sum of morphisms of the form
\[
\sum_{a\in S}d_a^\alpha\otimes\rho_a:X^\alpha\otimes E_S(\alpha)\to\bigoplus_{a\in S}X^{\alpha+a}\otimes E_S(\alpha+a)
\]
for $\alpha:S\to\mathbb Z$.
\Cref{lem:rho-lambda-com} implies $d^{i+1}\circ d^i=0$, so it makes $\operatorname{Tot}(X^\bullet)$ into a chain complex in $\mathcal A$.
We call $\operatorname{Tot}(X^\bullet)$ the \emph{total complex} of $X^\bullet$.

We denote by $\mathbf{MCh}_S^{\mathsf b}(\mathcal A)\subset\mathbf{MCh}_S(\mathcal A)$ the full subcategory spanned by bounded $S$-fold complexes.
The construction above actually gives rise to a $k$-linear functor
\begin{equation}
\label{eq:MCh-Tot}
\operatorname{Tot}:\mathbf{MCh}_S^{\mathsf b}(\mathcal A)\to\mathbf{Ch}^{\mathsf b}(\mathcal A)
\end{equation}
in an obvious way.

\begin{remark}
The equivalence in \cref{ex:1fold-cplx} is realized as the functor \eqref{eq:MCh-Tot}.
\end{remark}

\subsection{The category of cobordisms and local relations}
\label{sec:UKH-mulcplx:cob}

We quickly review cobordisms with boundaries and Bar-Natan's local relations.
In this paper, we rather use the formulation in terms of cobordisms; we write $\mathbb R_+\coloneqq[0,\infty)\subset\mathbb R$ the half open interval.

\begin{definition}[\cite{Janich1968}, see also {\cite[Proposition~2.1.7]{Laures2000}}]
A \emph{$\langle2\rangle$-manifold} is a manifold with corners $M$ together with a pair of submanifolds $\partial_0M$ and $\partial_1M$ such that there is a neat embedding $M\hookrightarrow\mathbb R_+^2\times\mathbb R^n$ for some positive integer $n$ with
\[
\partial_0M = M\cap (\{0\}\times\mathbb R_+\times\mathbb R^n)
\ ,\quad\partial_1M = M\cap(\mathbb R_+\times\{0\}\times\mathbb R^n)
\quad.
\]
\end{definition}

\begin{definition}[\cite{SchommerPries2009}]
Let $Y_0$ and $Y_1$ be two compact oriented $0$-manifolds; i.e.~finite sets with labels in $\{\mathord-,\mathord+\}$.
We define a category $\mathbf{Cob}_2(Y_0,Y_1)$ as follows:
\begin{itemize}
  \item objects are $1$-dimensional cobordisms $W:Y_0\to Y_1$;
  \item morphisms $W_0\to W_1$ are diffeomorphism classes of $2$-bordisms; i.e.~compact oriented $2$-dimensional $\langle2\rangle$-manifolds $S$ equipped with a diffeomorphism
\[
\partial_0 S \cong\overbar{W_0}\cup W_1
\ ,\quad
\partial_1 S\cong (\overbar{Y_0}\amalg Y_1)\times[0,1]
\]
which agree with each other on the boundaries, where $\overbar{W_0}$ and $\overbar{Y_0}$ are respectively the manifolds $W_0$ and $Y_0$ with the reversed orientations;
  \item the composition is defined in terms of the gluing of $\langle2\rangle$-manifolds along boundaries (see \cite{SchommerPries2009}).
\end{itemize}
\end{definition}

\begin{example}
The category $\mathbf{Cob}_2(\varnothing,\varnothing)$ is exactly the category of oriented $1$-manifolds and cobordisms between them.
\end{example}

\begin{notation}
In this paper, we always use the ``bottom-to-top'' convention for cobordisms and the ``left-to-right'' one for $2$-bordisms as in \cref{fig:example-2bord}.
\begin{figure}[htbp]
\centering
\begin{tikzpicture}[xlen=.5pt,ylen=-.5pt,scale=.8]
\node at (0,0) {$S$};
\node at (-100,0) {$\partial_1^- S$};
\node at (100,0) {$\partial_1^+ S$};
\node[left] at (0,60) {$\xrightarrow{\simeq}\;\partial_0^-S$};
\node[right] at (0,-60) {$\partial_0^+S\;\xleftarrow{\simeq}$};
\begin{scope}[transform canvas={xshift=-120}]
\node at (-100,0) {$Y_0$};
\node at (100,0) {$Y_1$};
\node at (0,60) {$W_0$};
\draw[red,very thick] (-59.5,40.5) -- (-19.5,40.5) to[quadratic={(0.5,40.5)}] (0.5,20.5) to[quadratic={(0.5,0.5)}] (-19.5,0.5) to[quadratic={(-39.5,0.5)}] (-39.5,-19.5) to[quadratic={(-39.5,-39.5)}] (-19.5,-39.5) -- (60.5,-39.5);
\draw[red,very thick] (60.5,40.5) to[quadratic={(40.5,40.5)}] (40.5,20.5) to[quadratic={(40.5,0.5)}] (60.5,0.5);
\fill[blue] (-59.5,40.5) circle (.5ex);
\fill[blue] (60.5,-39.5) circle(.5ex);
\fill[blue] (60.5,40.5) circle(.5ex);
\fill[blue] (60.5,0.5) circle(.5ex);
\end{scope}
\draw[red,very thick] (-73.59,48.63) -- (-33.59,48.63) to[quadratic={(-13.59,48.63)}] (-6.545,44.57) to[quadratic={(-4.71,43.51)}] (-4.71,42.73);
\draw[red,very thick,dotted] (-4.71,42.73) to[quadratic={(-4.71,40.5)}] (-19.5,40.5) to[quadratic={(-39.5,40.5)}] (-32.46,36.43) to[quadratic={(-25.41,32.37)}] (-5.41,32.37) -- (-3.76,32.37);
\draw[red,very thick] (-3.76,32.37) -- (29.86,32.37);
\draw[red,very thick,dotted] (29.86,32.37) -- (60.5,32.37);
\draw[red,very thick] (60.5,32.37) -- (74.59,32.37);
\draw[red,very thick] (46.41,48.63) to[quadratic={(31.62,48.63)}] (31.62,46.41);
\draw[red,very thick,dotted] (31.62,46.41) to[quadratic={(31.62,45.63)}] (33.46,44.57) to[quadratic={(37.64,42.15)}] (46.41,41.17);
\draw[red,very thick] (46.41,41.17) to[quadratic={(52.39,40.5)}] (60.5,40.5);
\draw[black] (-4.71,42.72) .. controls+(0,-40)and+(0,-40) .. (31.62,46.41);
\draw[black,thick,dotted] (-34.29,38.28) -- (-34.29,-31.37);
\draw[black] (-34.29,-31.37) -- (-34.29,-41.72);
\draw[black] (9.38,-45.41) .. controls(9.38,-40.08)and(10.03,-35.4) .. (11.14,-31.37);
\draw[black,thick,dotted] (11.14,-31.37) .. controls(17.61,-8.11)and(39.9,-7) .. (44.76,-31.37);
\draw[black] (44.76,-31.37) .. controls(45.37,-34.42)and(45.71,-37.87) .. (45.71,-41.72);
\draw[blue,very thick] (-73.59,48.63) -- (-73.59,-31.37);
\draw[blue,very thick] (46.41,48.63) -- (46.41,-31.37);
\draw[blue,very thick] (60.5,40.5) -- (60.5,-39.5);
\draw[blue,very thick] (74.59,32.37) -- (74.59,-47.63);
\draw[red,very thick] (-73.59,-31.37) -- (46.41,-31.37);
\draw[red,very thick] (-32.46,-43.57) to[quadratic={(-39.5,-39.5)}] (-19.5,-39.5) to[quadratic={(0.5,-39.5)}] (7.545,-43.57) to[quadratic={(14.59,-47.63)}] (-5.41,-47.63) to[quadratic={(-25.41,-47.63)}] (-32.46,-43.57);
\draw[red,very thick] (60.5,-39.5) to[quadratic={(40.5,-39.5)}] (47.54,-43.57) to[quadratic={(54.59,-47.63)}] (74.59,-47.63);
\begin{scope}[transform canvas={xshift=120}]
\node at (-100,0) {$Y_0$};
\node at (100,0) {$Y_1$};
\node at (0,-60){$W_1$};
\draw[red,very thick] (-59.5,40.5) -- (60.5,40.5);
\draw[red,very thick] (-39.5,-19.5) to[quadratic={(-39.5,0.5)}] (-19.5,0.5) to[quadratic={(0.5,0.5)}] (0.5,-19.5) to[quadratic={(0.5,-39.5)}] (-19.5,-39.5) to[quadratic={(-39.5,-39.5)}] (-39.5,-19.5);
\draw[red,very thick] (60.5,0.5) to[quadratic={(40.5,0.5)}] (40.5,-19.5) to[quadratic={(40.5,-39.5)}] (60.5,-39.5);
\fill[blue] (-59.5,40.5) circle (.5ex);
\fill[blue] (60.5,-39.5) circle(.5ex);
\fill[blue] (60.5,40.5) circle(.5ex);
\fill[blue] (60.5,0.5) circle(.5ex);
\end{scope}
\end{tikzpicture}
\caption{Example of a $2$-bordism (orientation omitted).}
\label{fig:example-2bord}
\end{figure}
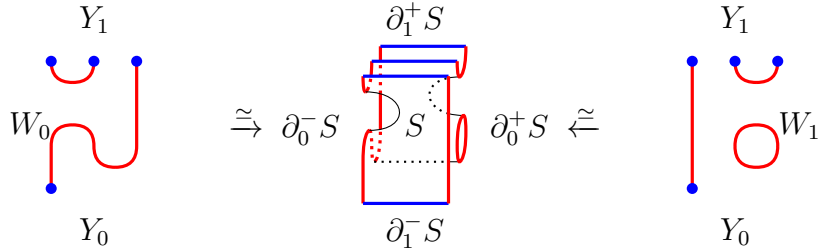
\end{notation}

For a fixed base ring $k$, we denote by $k\mathbf{Cob}_2(Y_0,Y_1)$ the $k$-linear category freely generated by $\mathbf{Cob}_2(Y_0,Y_1)$; i.e.~it has the same objects as $\mathbf{Cob}_2(Y_0,Y_1)$ and $k\mathbf{Cob}_2(Y_0,Y_1)(W_0,W_1)$ is the free $k$-module generated by the set $\mathbf{Cob}_2(Y_0,Y_1)(W_0,W_1)$.
Bar-Natan \cite{BarNatan2005} introduced the following three relations on the morphisms of $k\mathbf{Cob}_2(Y_0,Y_1)$:
\begin{enumerate}[label=\upshape(\unexpanded{\labelseq{$S$\\$T$\\$4Tu$}}{\value*})]
  \item\label{relBN:S} $S\amalg S^2\sim0$, here $S^2$ is the $2$-sphere.
  \item\label{relBN:T} $S\amalg T^2\sim 2\cdot W$, here $T^2$ is the torus $S^1\times S^2$;
  \item\label{relBN:4Tu} $S_0+S_1-S_2-S_3\sim 0$ if $S_0$, $S_1$, $S_2$, and $S_3$ are identical outside disks and tubes that are depicted as in \cref{tab:4Tu-terms}.
\begin{table}[t]
\centering
\begin{tabular}{c|c|c|c}
  $S_0$ & $S_1$ & $S_2$ & $S_3$ \\\hline\rule{0pt}{11ex}
  \BordFourTuL & \BordFourTuR & \BordFourTuU & \BordFourTuD
\end{tabular}
\caption{Terms in the relation \ref{relBN:4Tu}}
\label{tab:4Tu-terms}
\end{table}
\end{enumerate}
We define a $k$-linear category $\Cob(Y_0,Y_1)$ as the additive closure of the quotient category of $k\mathbf{Cob}_2(Y_0,Y_1)$ by the three relations above; more explicitly
\begin{itemize}
  \item objects of $\Cob(Y_0,Y_1)$ are finite sequences $(W_1,\dots,W_n)$ of oriented cobordisms $W_i:Y_0\to Y_1$;
  \item morphisms are matrices of formal sums of (diffeomorphism classes of) $2$-bordisms modulo the relations \ref{relBN:S}, \ref{relBN:T}, and \ref{relBN:4Tu}.
\end{itemize}
For more detailed description, we refer the reader to \cite[Section~3]{BarNatan2005}.
The following is straightforward from the construction, where we denote by $\mathbf{Mod}_k$ the category of $k$-modules and $k$-homomorphisms.

\begin{lemma}
\label{lem:CobB-additive}
For every compact oriented $0$-manifolds $Y_0$ and $Y_1$, the category $\Cob(Y_0,Y_1)$ is an additive category with an isomorphism
\[
(W_1,\dots,W_p)\oplus(W'_1,\dots,W'_q)
\cong (W_1,\dots,W_p,W'_1,\dots,W'_q)
\quad.
\]
Furthermore, $k$-linear functors $\Cob(Y_0,Y_1)\to\mathbf{Mod}_k$ correspond essentially in one-to-one to $k$-linear functors $k\mathbf{Cob}(Y_0,Y_1)\to\mathbf{Mod}_k$ which identifies morphisms connected by the relations \ref{relBN:S}, \ref{relBN:T}, and \ref{relBN:4Tu}.
\end{lemma}
\begin{proof}
The first half is easily verified.
As a result, if we write $W\coloneqq(W)\in\Cob(Y_0,Y_1)$ for each cobordism $W:Y_0\to Y_1$, then every object of $\Cob(Y_0,Y_1)$ can be written in the form $\bigoplus_iW_i$ with $W_i\in\mathbf{Cob}_2(Y_0,Y_1)$.
Since every $k$-linear functor preserves direct products, it follows that a functor $F:\Cob(Y_0,Y_1)\to\mathbf{Mod}_k$ is determined by its restriction to $k\mathbf{Cob}(Y_0,Y_1)\to\mathbf{Mod}_k$.
Hence, the result follows.
\end{proof}

In addition to the direct sums, we use the disjoint union of manifolds as a kind of tensor product.
Namely, it defines a functor
\[
\mathbf{Cob}_2(Y_0,Y_1)\times\mathbf{Cob}_2(Y_0',Y_1')
\to \mathbf{Cob}_2(Y_0\amalg Y_0',Y_1\amalg Y_1')
\quad.
\]
It then turns out that it further induces a $k$-bilinear functor
\begin{equation}
\label{eq:disj-tensor}
\Cob(Y_0,Y_1)\times\Cob(Y_0',Y_1')
\to \Cob(Y_0\amalg Y_0',Y_1\amalg Y_1')
\quad.
\end{equation}
For cobordisms $W,W'$, and for $2$-bordisms $S,S'$, we denote by $W\otimes W'=W\amalg W'$ and $S\otimes S'=S\amalg S'$ the images under the functor.
Specifically, with this product, the category $\Cob(\varnothing,\varnothing)$ becomes a symmetric monoidal category.

As examples, we construct a family of symmetric monoidal $k$-linear functors $\Cob(\varnothing,\varnothing)\to\mathbf{Mod}_k$.

\begin{definition}[{\cite{Khovanov2006}}, {\cite[Definition~1.1]{LaudaPfeiffer2009}}]
\label{def:FrobCht}
For elements $h,t\in k$, we define a Frobenius algebra $C_{h,t}\coloneqq k[x]/(x^2-hx-t)$ with
\[
\begin{gathered}
\Delta(1)=1\otimes x+x\otimes 1 - h 1\otimes 1
\quad,\\
\Delta(x)=x\otimes x + t1\otimes 1
\quad,\\
\varepsilon(1)=0\ ,\quad \varepsilon(x)=1
\quad.
\end{gathered}
\]
\end{definition}

Recall that every Frobenius algebra $A$ over $k$ gives rise to a $k$-linear functor $Z_A:k\mathbf{Cob}_2(\varnothing,\varnothing)\to\mathbf{Mod}_k$ with $Z_A(S^1)= A$.

\begin{lemma}
\label{lem:Cht-induce}
For every pair of elements $h,t\in k$, the associated $k$-linear functor $Z_{C_{h,t}}:k\mathbf{Cob}_2(\varnothing,\varnothing)\to\mathbf{Mod}_k$ induces a symmetric monoidal $k$-linear functor $\Cob(\varnothing,\varnothing)\to\mathbf{Mod}_k$.
\end{lemma}
\begin{proof}[Sketch of the proof]
For every Frobenius algebra $A=(A,\mu,\eta,\Delta,\varepsilon)$ over $k$, unwinding the construction of the functor $Z_A$, one can see that the relations \ref{relBN:S}, \ref{relBN:T}, and \ref{relBN:4Tu} are respectively equivalent to the following equations:
\[
\begin{gathered}
\varepsilon\eta=0\,,\  \varepsilon\mu\Delta\eta=2\ \in\mathrm{Hom}_k(k,k)
\quad,\\
\mathrm{id}\otimes\eta\varepsilon + \eta\varepsilon\otimes\mathrm{id}-(\eta\otimes\eta)\circ(\varepsilon\mu)-(\Delta\eta)\circ(\varepsilon\otimes\varepsilon)= 0\ \in\mathrm{Hom}_k(A\otimes A,A\otimes A)
\quad.
\end{gathered}
\]
As for the Frobenius algebra $A=C_{h,t}$, these equations are easily verified by the direct computations.
\end{proof}

\begin{definition}
\label{def:TQFT-ht}
For elements $h,t\in k$, we write
\[
Z_{h,t}:\Cob(\varnothing,\varnothing)\to\mathbf{Mod}_k
\]
the symmetric monoidal $k$-linear functor induced by the Frobenius algebra $C_{h,t}$ in \cref{def:FrobCht}.
\end{definition}

\begin{remark}
\label{rem:Euler-gr}
A morphism $S:W_0\to W_1\in\mathbf{Cob}_2(Y_0,Y_1)$ has a canonical grading given by
\begin{equation}
\label{eq:Euler-gr-def}
\deg S
\coloneqq \chi(S)-\frac{\chi(W_0)+\chi(W_1)}{2}
\quad,
\end{equation}
where $\chi(\blank)$ is the Euler characteristic.
We call it the \emph{Euler grading}.
By virtue of the Mayer-Vietoris sequence, we have $\deg(S'\circ S)=\deg S'+\deg S$.
Furthermore, since the relations~\ref{relBN:S}, \ref{relBN:T}, and \ref{relBN:4Tu} are homogeneous with respect to the grading, it also induces a grading on the morphisms in $\Cob(Y_0,Y_1)$.
In other words, $\Cob(Y_0,Y_1)$ can be regarded to be enriched over graded $k$-modules.
In this point of view, the functor $Z_{h,t}$ in \cref{def:TQFT-ht} respects gradings if $k$ is a graded ring with $\deg h=-2$ and $\deg t=-4$.
Indeed, we have a grading on $C_{h,t}$ with $\deg 1=1$ and $\deg x=-1$ so that the operations in the Frobenius algebra structure reflects Euler characteristics.
In this case, the TQFT $Z_{h,t}$ is said to be \emph{Euler-graded} \cite{LaudaPfeiffer2009}.
\end{remark}

\subsection{The multi-fold complex of smoothings}
\label{sec:UKH-mulcplx:UKH-mcplx}

Bar-Natan \cite{BarNatan2005} defined Khovanov complex for tangle diagrams in the category $\Cob(Y_0,Y_1)$.
As it yields variants of Khovanov homology through TQFTs, we sometimes refer to it as \emph{universal Khovanov complex} \cite{ItoYoshida2020UKH} to distinguish it from Khovanov's original complex.
For a technical reason, we extend it to a bit more general graphs.

\begin{notation}
For a graph $G=(V,E)$, we denote by $V^n(E)\subset V(E)$ the set of $n$-valent vertices.
\end{notation}

\begin{definition}[cf.~\cite{KhovanovRozansky2008}]
A \emph{singular tangle-like graph} is a planar oriented finite graph $G=(V(G),E(G))$ neatly embedded in the strip $\mathbb R\times[0,1]$ equipped with
\begin{itemize}
  \item a subset $E^{\mathsf{wide}}(G)\subset E(G)$, whose elements are called \emph{wide edges};
  \item a subset $c^\#(G)\subset V^4(G)$, whose elements are called \emph{double points}; and
  \item a map $u:V^4(G)\setminus c^\sharp(G)\to E(G)$
\end{itemize}
such that
\begin{enumerate}[label=\upshape(\roman*)]
  \item every vertex is of the form in \cref{fig:smtgr-wide}, where wide edges are drawn with gray lines;
  \item for each $v\in V^4(G)$, $v$ is the source of $u(v)$.
\end{enumerate}
\begin{figure}[t]
\centering
\begin{tikzpicture}[baseline=(current bounding box.center)]
\draw (-.5,.5) node[left]{$\mathbb R^2\times\{0,1\}$} -- (.5,.5);
\fill[blue] (0,.5) circle(.1);
\draw[red,very thick,-stealth] (0,-.6) -- (0,.5);
\end{tikzpicture}
\;,\quad
\begin{tikzpicture}[baseline=(current bounding box.center)]
\draw (-.5,.5) node[left]{$\mathbb R^2\times\{0,1\}$} -- (.5,.5);
\fill[blue] (0,.5) circle(.1);
\draw[red,very thick,stealth-] (0,-.6) -- (0,.5);
\end{tikzpicture}
\;,\quad
\begin{tikzpicture}[baseline=(current bounding box.center)]
\fill[blue] (0,0) circle(.1);
\draw[red,very thick,-stealth] (-.5,-.6) to[out=60,in=180] (0,0);
\draw[red,very thick,-stealth] (.5,-.6) to[out=120,in=0] (0,0);
\draw[gray,line width=3,-stealth] (0,0) -- (0,.6);
\end{tikzpicture}
\;,\quad
\begin{tikzpicture}[baseline=(current bounding box.center)]
\fill[blue] (0,0) circle(.1);
\draw[gray,line width=3,-stealth] (0,-.6) -- (0,0);
\draw[red,very thick,-stealth] (0,0) to[out=180,in=-60] (-.5,.6);
\draw[red,very thick,-stealth] (0,0) to[out=0,in=-120] (.5,.6);
\end{tikzpicture}
\;, or\quad
\begin{tikzpicture}[baseline=(current bounding box.center)]
\fill[blue] (0,0) circle(.1);
\draw[red,very thick,-stealth] (-.5,-.6) -- (0,0);
\draw[red,very thick,-stealth] (.5,-.6) -- (0,0);
\draw[red,very thick,-stealth] (0,0) -- (-.5,.6);
\draw[red,very thick,-stealth] (0,0) -- (.5,.6);
\end{tikzpicture}
\caption{Vertices in singular tangle-like graphs}
\label{fig:smtgr-wide}
\end{figure}
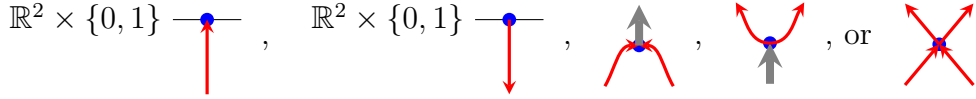
\end{definition}

If $G$ is a singular tangle-like graph, then we depict its quadrivalent vertices as in \cref{tab:quad-vert-pics}.
\begin{table}[t]
\centering
\begin{tabular}{c||c|c|c}
  $G$ & $v\in c^\sharp(G)$ & $u(v)=e_1$ & $u(v)=e_2$ \\\hline
  \begin{tikzpicture}[baseline=-.5ex]
  \fill[blue] (0,0) node[left=.5em]{$v$} circle(.1);
  \draw[red,very thick,-stealth] (-.5,-.6) -- (0,0);
  \draw[red,very thick,-stealth] (.5,-.6) -- (0,0);
  \draw[red,very thick,-stealth] (0,0) -- (-.5,.6) node[left]{$e_1$};
  \draw[red,very thick,-stealth] (0,0) -- (.5,.6) node[right]{$e_2$};
  \end{tikzpicture}
    & \diagSingUp & \diagCrossNegUp & \diagCrossPosUp
\end{tabular}
\caption{Pictures around quadrivalent vertices}
\label{tab:quad-vert-pics}
\end{table}
Quadrivalent vertices of the forms in \cref{tab:quad-vert-pics} are respectively called \emph{double points}, \emph{negative crossings}, and \emph{positive crossings}.
Hence, forgetting the wide edges and the orientations, one obtains an unoriented tangle diagram.
Conversely any (oriented) singular tangle diagram is nothing but a singular tangle-like graph with no wide edges.

Let $G$ be a singular tangle-like graph.
For every map $\alpha:V^4(G)\to\mathbb Z$, we say $\alpha$ lies in the effective range if
\begin{itemize}
  \item $-2\le\alpha(v)\le 1$ for each double point $v$;
  \item $-1\le\alpha(v)\le 0$ for each negative crossing $v$; and
  \item $0\le \alpha(v)\le 1$ for each positive crossing $v$.
\end{itemize}
In this case, we define a singular tangle-like graph $G_\alpha$ by replacing all the quadrivalent vertices of $G$ as in \cref{tab:quad-repl}.
\begin{table}[t]
\centering
\begin{tabular}{c||c|c|c|c}
  $G$ & $\alpha(v)=-2$ & $\alpha(v)=-1$ & $\alpha(v)=0$ & $\alpha(v)=1$ \\\hline\rule[-3.5ex]{0pt}{8ex}
  \diagSingUp & \diagSmoothW & \diagSmoothUp & \diagSmoothUp & \diagSmoothW \\\hline\rule[-3.5ex]{0pt}{8ex}
  \diagCrossNegUp && \diagSmoothW & \diagSmoothUp &  \\\hline\rule[-3.5ex]{0pt}{8ex}
  \diagCrossPosUp &&& \diagSmoothUp & \diagSmoothW
\end{tabular}
\caption{Replacing quadrivalent vertices}
\label{tab:quad-repl}
\end{table}
We call $G_\alpha$ the \emph{$\alpha$-smoothing} of $G$.
Note that since $G_\alpha$ has no quadrivalent vertex, the union of all non-wide edges of $G_\alpha$ forms a neat submanifold of $\mathbb R^2\times[0,1]$, which we denote by $|G_\alpha|$.
To specify an orientation on $|G_\alpha|$, we introduce the following notion.

\begin{definition}
For a singular tangle-like graph $G$, a \emph{checkerboard coloring} on the complement of $G$ is a mapping
\[
\chi:\pi_0((\mathbb R^2\times[0,1])\setminus G)\to\{\mathsf{black},\mathsf{white}\}
\quad,
\]
here $\pi_0(\blank)$ is the set of connected components, which satisfies the following conditions:
\begin{enumerate}[label=\upshape(\roman*)]
  \item if $C,C'\in\pi_0((\mathbb R^2\times[0,1])\setminus G)$ are adjacent across a non-wide edge, then $\chi(C)\neq\chi(C')$;
  \item if $C,C'\in\pi_0((\mathbb R^2\times[0,1])\setminus G)$ are adjacent across a wide edge, then $\chi(C)=\chi(C')$.
\end{enumerate}
In this case, a component $C\in\pi_0((\mathbb R^2\times[0,1])\setminus G)$ is said to be \emph{black} (resp.~\emph{white}) if $\chi(C)=\mathsf{black}$ (resp.~$\chi(C)=\mathsf{white}$).
\end{definition}

Let $G$ be a singular tangle-like graph.
If a map $\alpha:V^4(G)\to\mathbb Z$ lies in the effective range, then there is a canonical surjection $\pi_0((\mathbb R^2\times[0,1])\setminus G_\alpha)\to \pi_0((\mathbb R^2\times[0,1])\setminus G)$.
For a checkerboard coloring $\chi$ on $(\mathbb R^2\times[0,1])\setminus G$, it turns out that the composition
\[
\pi_0((\mathbb R^2\times[0,1])\setminus G_\alpha)
\twoheadrightarrow \pi_0((\mathbb R^2\times[0,1])\setminus G)
\xrightarrow\chi
\{\mathsf{black},\mathsf{white}\}
\]
is a checkerboard coloring, which we again denote by $\chi$ by abuse of notation.
In this case, the submanifold $|G_\alpha|\subset \mathbb R^2\times[0,1]$ coincides with the boundary of the union of the black components.
We hence endow $|G_\alpha|$ with the orientation induced from this identification and denote by $|G_\alpha|^\chi$ the resulting oriented manifold.
Specifically, since the oriented $0$-dimensional submanifold $\partial|G_\alpha|^\chi\subset \mathbb R^2\times\{0,1\}$ does not depend on $\alpha$; we write
\[
\partial^-|G|^\chi\coloneqq \overbar{|G_0|^\chi\cap(\mathbb R^2\times\{0\})}
\ ,\quad
\partial^+|G|^\chi\coloneqq |G_0|^\chi\cap(\mathbb R^2\times\{1\})
\quad.
\]
Then, we think of $|G_\alpha|^\chi$ as an object of the category $k\mathbf{Cob}_2(\partial^-|G|^\chi,\partial^+|G|^\chi)$ and hence of $\Cob(\partial^-|G|^\chi,\partial^+|G|^\chi)$.
From this point of view, we also write $|G_\alpha|^\chi=0\in \Cob(\partial^-|G|^\chi,\partial^+|G|^\chi)$ for the maps $\alpha:V^4(G)\to\mathbb Z$ that lie outside the effective range.

\begin{remark}
Every singular tangle-like graph has exactly two checkerboard colorings.
Namely, if $\chi$ is one, then the other is obtained by swapping all the values of $\chi$, which we denote by $-\chi$.
The oriented manifold $|G_\alpha|^{-\chi}$ is identified with $|G_\alpha|^\chi$ with the reversed orientation.
\end{remark}

We introduce three morphisms in $\Cob(\partial^-|G|^\chi,\partial^+|G|^\chi)$ as the last ingredients.

\begin{definition}
\label{def:delta-Phi}
Let $H$ and $H'$ be singular tangle-like graphs without quadrivalent vertices which differ only on a small region in $\mathbb R^2\times[0,1]$ where we have
\[
H=\;\diagSmoothW
\ ,\quad H'=\;\diagSmoothUp
\quad.
\]
For any checkerboard coloring $\chi$ on $(\mathbb R^2\times[0,1])\setminus H$, we may think of it also as a checkerboard coloring on $(\mathbb R^2\times[0,1])\setminus H'$.
In this case, we define morphisms $\delta_-:|H|^\chi\to |H'|^\chi$, $\delta_+:|H'|^\chi\to |H|^\chi$, and $\Phi:|H'|^\chi\to |H'|^\chi$ in $\Cob(\partial^-|H|^\chi,\partial^+|H|^\chi)$ as follows:
\begin{gather}
\label{eq:delta-}
\delta_-\coloneqq\BordDeltaN\::\:\left|\diagSmoothW\right|^\chi\to\left|\diagSmoothUp\right|^\chi
\quad,\\
\label{eq:delta+}
\delta_+\coloneqq\BordDeltaP\::\:\left|\diagSmoothUp\right|^\chi\to\left|\diagSmoothW\right|^\chi
\quad,\\
\label{eq:Phi}
\Phi\coloneqq\BordPhiFst\;-\;\BordPhiSnd\::\:\left|\diagSmoothUp\right|^\chi\to\left|\diagSmoothUp\right|^\chi
\quad.
\end{gather}
\end{definition}

\begin{lemma}
\label{lem:deltaPhi=0}
In the situation in \cref{def:delta-Phi}, the compositions $\delta_+\Phi$ and $\Phi\delta_-$ are zero.
\end{lemma}
\begin{proof}
The result immediately follows from the diffeomorphisms
\[
\BordPhiFst\BordDeltaP\;\cong\;\BordPhiSnd\BordDeltaP
\ ,\quad
\BordDeltaN\BordPhiFst\;\cong\;\BordDeltaN\BordPhiSnd
\quad.
\]
\end{proof}

\begin{remark}
\label{rem:Phi-reverse}
Although the oriented manifold $|H'|^\chi$ does not depend on the orientations on the edges, the morphism $\Phi$ does.
Indeed, if $\overbar{H'}$ is the singular tangle-like graph obtained by reversing the edges of $H'$, then the morphism $\Phi=\Phi_{\overbar{H'}}$ is given by
\[
\Phi_{\overbar{H'}}=\BordPhiFst\;-\;\BordPhiSndVar\::\:\left|\diagSmoothDown\right|^\chi\to\left|\diagSmoothDown\right|^\chi
\quad.
\]
One can however see $\Phi_{\overbar{H'}}=-\Phi_{H'}$ by virtue of the relation \ref{relBN:4Tu}.
\end{remark}

For a singular tangle-like graph $G$ and a checkerboard coloring on the complement, we now define a $V^4(G)$-fold complex $\operatorname{Sm}(G^\chi)^\bullet$ in $\Cob(\partial^-|G|^\chi,\partial^+|G|^\chi)$ by $\operatorname{Sm}(G^\chi)^\alpha\coloneqq |G_\alpha|^\chi$ for each $\alpha:V^4(G)\to\mathbb Z$ and the differentials $\{d_v:|G_\alpha|^\chi\to|G_{\alpha+v}|^\chi\}$ defined so that, for each $v\in V^4(G)$,
\begin{itemize}
  \item if $v$ is a double point, $d_v$ is given by
\[
\cdots\to 0
\to \overset{\alpha(v)=-2}{\left|\diagSmoothW\right|^\chi}
\xrightarrow{-\delta_-} \overset{\alpha(v)=-1}{\left|\diagSmoothUp\right|^\chi}
\xrightarrow{\Phi} \overset{\alpha(v)=0}{\left|\diagSmoothUp\right|^\chi}
\xrightarrow{-\delta_+} \overset{\alpha(v)=1}{\left|\diagSmoothW\right|^\chi}
\to 0\to\cdots
\quad;
\]
  \item if $v$ is a negative crossing, $d_v$ is given by
\[
\cdots\to 0
\to \overset{\alpha(v)=-1}{\left|\diagSmoothW\right|^\chi}
\xrightarrow{\delta_-}\overset{\alpha(v)=0}{\left|\diagSmoothUp\right|^\chi}
\to 0\to\cdots
\quad;
\]
  \item if $v$ is a positive crossing, $d_v$ is given by
\[
\cdots\to 0
\to \overset{\alpha(v)=0}{\left|\diagSmoothUp\right|^\chi}
\xrightarrow{-\delta_+}\overset{\alpha(v)=1}{\left|\diagSmoothW\right|^\chi}
\to 0\to\cdots
\quad.
\]
\end{itemize}
In view of \cref{lem:deltaPhi=0}, the family $\{\operatorname{Sm}(G^\chi)^\alpha,d_v\}$ actually forms a $V^4(G)$-fold complex.

\begin{definition}
\label{def:UKH-as-mcplx}
For a singular tangle-like graph $G$ with a checkerboard coloring $\chi$ on the complement, we set
\[
\dblBrac{G^\chi}\coloneqq \operatorname{Tot}(\operatorname{Sm}(G^\chi)^\bullet)
\]
and call it the \emph{universal Khovanov complex} of $G$ with respect to $\chi$.
\end{definition}

\begin{remark}
The checkerboard coloring $\chi$ is often omitted from the notation of the universal Khovanov complex when its choice is not essential in discussion.
\end{remark}

In \cite{ItoYoshida2020UKH}, the author and Ito discussed a modified version of Bar-Natan's complex of cobordisms for tangle diagrams, which was called the \emph{universal Khovanov complex}.
Namely, if $D$ is an ordinary tangle diagram (i.e.~with no double point) seen as a (singular) tangle-like graph, then the complex $\dblBrac{D^\chi}$ is exactly the one.
Furthermore, they showed that the morphism $\Phi$ defined in \eqref{eq:Phi} induces a morphism of chain complexes
\begin{equation}
\label{eq:PhiHat}
\widehat\Phi:\dblBrac*{\diagCrossNegUp}\to\dblBrac*{\diagCrossPosUp}
\end{equation}
which is invariant under the moves of singular tangle diagrams \cite[Section~4]{ItoYoshida2020UKH}.
Using this fact, they extend the universal Khovanov complex to singular tangle diagrams by taking iterated mapping cones for double points.

\begin{notation}
For a singular tangle-like graph $G$, we say a connected component $C$ of the complement $(\mathbb R\times[0,1])\setminus G$ is \emph{negatively unbounded} if $C\cap((-\infty,0]\times[0,1])$ is unbounded.
It turns out that the negatively unbounded component is always unique and that  a checkerboard coloring on the complement $(\mathbb R\times[0,1])\setminus G$ is determined by its value at it.
\end{notation}

\begin{theorem}
\label{theo:UKH-compare}
Let $D$ be a singular tangle diagram, and let $\chi_{\mathsf w}$ be the checkerboard coloring on the complement which has negatively unbounded white component.
Then the complex $\dblBrac{D^{\chi_{\mathsf w}}}$ given in \cref{def:UKH-as-mcplx} is isomorphic to the one given in \cite[Definition~5.2]{ItoYoshida2020UKH}.
\end{theorem}

\begin{corollary}
\label{theo:UKH-inv}
The assignment $D\mapsto\dblBrac{D^{\mathsf w}}$ defines an invariant of singular tangles up to chain homotopies.
\end{corollary}

We prove \cref{theo:UKH-compare} in the next section.

\begin{remark}
\label{rem:Kh-grading}
The morphisms $\delta_\pm$ and $\Phi$ are homogeneous of degree $-1$ and $-2$ respectively with respect to the Euler grading (see \cref{rem:Euler-gr}).
Hence, if $Z_{h,t}:\Cob(\varnothing,\varnothing)\to\mathbf{Mod}_k$ is an Euler-graded TQFT in the sense of \cref{rem:Euler-gr}, the chain complex $Z_{h,t}\dblBrac{G^\chi}$ has an additional grading, sometimes referred to as the \emph{quantum degree}, given by
\[
Z_{h,t}(\dblBrac{G^\chi})^{i,j}
\coloneqq \bigoplus_{|\alpha|=i}Z_{h,t}(|G_\alpha|^\chi)^{j-q(\alpha)-\widetilde w(G)}
\subset Z_{h,t}(\dblBrac{G^\chi}^i)
\quad,
\]
where we set
\begin{itemize}
  \item $q(\alpha)\coloneqq\sum_v\alpha^q(v)$ to be the sum with
\[
\alpha^q(v)\coloneqq
\begin{cases*}
\alpha(v)-1 & if $v$ is a double point and $\alpha(v)<0$, \\
\alpha(v) & otherwise;
\end{cases*}
\]
  \item $\widetilde w(G)\coloneqq n_\times+n_+-n_-$ with $n_\times$, $n_+$, and $n_-$ being the numbers of double points, positive crossings, and negative crossings respectively.
\end{itemize}
Specifically, if $G$ is an ordinary link diagram, then $q(\alpha)=|\alpha|$ and $\widetilde w(G)$ equals the writhe of $G$.
Thanks to \cref{theo:UKH-compare}, it turns out that, for each $i,j\in\mathbb Z$, the $k$-module
\[
H^i\left(Z_{h,t}\dblBrac{G^\chi}^{\bullet,j}\right)
\]
defines a singular link invariant.
For instance, $\mathit{Kh}^{i,j}(G^\chi;k)\coloneqq H^i\left(Z_{0,0}\dblBrac{G^\chi}^{\bullet,j}\right)$ is exactly the \emph{Khovanov homology} of $G$ \cite{ItoYoshida2020}.
\end{remark}

\subsection{Mapping cones}
\label{sec:UKH-mulcplx:mcone}

Multi-fold complexes have intrinsic relationship to mapping cones.
To discuss it, we introduce several operations.

\begin{definition}
Let $\mathcal A$ be an additive category, and let $X^\bullet=\{X^\alpha,d_a^\alpha\}_{\alpha,a}$ be an $S$-fold complex in $\mathcal A$ for a finite set $S$.
For a map $\alpha_0:S\to\mathbb Z$, we define an $S$-fold complex $X[\alpha_0]^\bullet=\{X[\alpha_0]^\alpha,(d_{[\alpha_0]})_a^\alpha\}_{\alpha,a}$ by $X[\alpha_0]^\alpha\coloneqq X^{\alpha-\alpha_0}$ with the differential
\[
(d_{[\alpha_0]})^\alpha_a\coloneqq(-1)^{\alpha_0(a)}d_a^{\alpha-\alpha_0}
\]
and call it the \emph{shift} of $X^\bullet$ by $\alpha_0$.
\end{definition}

\begin{notation}
Following the convention in \cref{sec:UKH-mcplx:modst}, we regard each element of $S$ as a map $S\to\mathbb Z$; hence, for each $a_0\in S$ and $r\in\mathbb Z$, we may write $X[r\cdot a_0]^\bullet$ the shift by $r\cdot a_0$.
In particular, if $S$ is a singleton, we simply write $X[r]$ instead, which is consistent with the shift of ordinary chain complexes.
\end{notation}

Note that there is an isomorphism
\[
\operatorname{Tot}(X[\alpha_0]^\bullet)
\cong\operatorname{Tot}(X^\bullet)[|\alpha_0|]
\quad,
\]
which is, however, not the identity because of the sign convention in the shift.

\begin{definition}
Let $\mathcal A$ be an additive category, and let $X^\bullet=\{X^\alpha,d_a^\alpha\}_{\alpha,a}$ be an $S$-fold complex in $\mathcal A$ for a finite set $S$.
For each $a_0\in S$ and $r\in\mathbb Z$, we define an $S$-fold complex $\sigma_{a_0}^{\ge r}X^\bullet=\{\sigma_{a_0}^{\ge r}X^\alpha,(d^{\ge r})^\alpha_a\}_{\alpha,a}$ by
\[
\sigma_{a_0}^{\ge r}X^\alpha\coloneqq
\begin{cases}
0 & \alpha(a_0) < r\ ,\\
X^\alpha & \alpha(a_0) \ge r\ ,
\end{cases}
\qquad
(d^{\ge r})^\alpha_a\coloneqq
\begin{cases}
0 & \alpha(a_0) < r\ ,\\
d^\alpha_a & \alpha(a_0) \ge r\ .
\end{cases}
\]
Similarly, we define $\sigma_{a_0}^{\le r}X^\bullet=\{\sigma_{a_0}^{\le r}X^\alpha,(d^{\le r})^\alpha_a\}_{\alpha,a}$ by
\[
\sigma_{a_0}^{\le r}X^\alpha\coloneqq
\begin{cases}
X^\alpha & \alpha(a_0)\le r\ ,\\
0 & \alpha(a_0)\le r\ ,
\end{cases}
\qquad
(d^{\le r})^\alpha_a\coloneqq
\begin{cases}
d^\alpha_{a_0} & \alpha(a_0)< r\ ,\\
0 & \alpha(a_0)\ge r\ .
\end{cases}
\]
These are sometimes called \emph{stupid truncations} of $X^\bullet$.
\end{definition}

Recall that, if $f:X\to Y$ be a morphism of chain complexes in an additive category $\mathcal A$, then the \emph{mapping cone} of $f$ is the chain complex $\operatorname{Cone}(f)$ with $\operatorname{Cone}(f)^n\coloneqq Y^n\oplus X^{n+1}$ and the differential
\[
d_{\operatorname{Cone}(f)}^n\coloneqq
\begin{pmatrix}
d_Y^n & f^{n+1} \\ 0 & -d_X^{n+1}
\end{pmatrix}
Y^n\oplus X^{n+1}\to Y^{n+1}\oplus X^{n+2}
\quad.
\]
The following result is fundamental.

\begin{lemma}
\label{lem:mcplx-cone}
Let $\mathcal A$ be an additive category, and let $S$ be a finite set with an element $a_0\in S$.
For an $S$-fold complex $X^\bullet=\{X^\alpha,d_a^\alpha\}$, and for an integer $r\in\mathbb Z$,
the following hold.
\begin{enumerate}[label=\upshape(\arabic*)]
  \item\label{sub:mcplx-cone:mor} Define $\varphi^\alpha:(\sigma^{\le r-1}_{a_0}X)[a_0]^\alpha\to \sigma^{\ge r}_{a_0}X^\alpha$ by
\[
\varphi^\alpha\coloneqq
\begin{cases}
d^{\alpha-a_0}_{a_0} & \alpha(a_0)= r, \\
0 & \alpha(a_0)\neq r.
\end{cases}
\]
Then the family $\varphi=\{\varphi_r^\alpha\}_\alpha$ forms a morphism $(\sigma^{\le r-1}_{a_0}X)[a_0]^\bullet\to\sigma^{\ge r}_{a_0}X^\bullet$ of $S$-fold complexes.
  \item\label{sub:mcplx-cone:cone} We denote by $\widehat\varphi:\operatorname{Tot}((\sigma^{\le r-1}_{a_0}X)[a_0]^\bullet)\to\operatorname{Tot}(\sigma^{\ge r}_{a_0}X^\bullet)$ the morphism of chain complexes induced by the morphism $\varphi$ in \ref{sub:mcplx-cone:mor}.
Then, there is an isomorphism
\[
\operatorname{Tot}(X^\bullet)
\cong \operatorname{Cone}(\widehat\varphi)
\]
of chain complexes in $\mathcal A$.
\end{enumerate}
\end{lemma}
\begin{proof}
For each $\alpha:S\to\mathbb Z$ with $\alpha(a_0)=0$, we have the following commutative diagram:
\[
\begin{tikzcd}
\cdots \ar[r] & X^{\alpha+(r-2)a_0} \ar[r,"d_{a_0}"] \ar[d] & X^{\alpha+(r-1)a_0} \ar[r] \ar[d,"d_{a_0}"] & 0 \ar[r] \ar[d] & \cdots \\
\cdots \ar[r] & 0 \ar[r] & X^{\alpha+ra_0} \ar[r,"d_{a_0}"] & X^{\alpha+(r+1)a_0} \ar[r,"d_{a_0}"] & \cdots
\end{tikzcd}
\quad.
\]
It clearly defines the morphism $\varphi:(\sigma^{\le r-1}_{a_0}X)[a_0]^\bullet\to\sigma^{\ge r}_{a_0}X^\bullet$ of $S$-fold complexes as in \ref{sub:mcplx-cone:mor}.

To verify \ref{sub:mcplx-cone:cone}, note that we have an isomorphism
\begin{equation}
\label{eq:prf:mcplx-cone:mapdef}
\begin{split}
\operatorname{Tot}(X^\bullet)^n
&= \biggl(\bigoplus_{\substack{|\alpha|=n\\\alpha(a_0)\ge r}}X^\alpha\otimes E_\alpha\biggr) \oplus \biggl(\bigoplus_{\substack{|\alpha|=n\\\alpha(a_0)\le r-1}}X^\alpha\otimes E_\alpha\biggr) \\[1ex]
&\xrightarrow{\begin{psmallmatrix}\mathrm{id} & 0 \\ 0 & \mathrm{id}\otimes\rho_{a_0}
  \end{psmallmatrix}} \biggl(\bigoplus_{\substack{|\alpha|=n\\\alpha(a_0)\ge r}}X^\alpha\otimes E_\alpha\biggr) \oplus \biggl(\bigoplus_{\substack{|\alpha'|=n+1\\\alpha(a_0)\le r}}X^{\alpha'-a_0}\otimes E_{\alpha'}\biggr) \\[1ex]
&= \operatorname{Tot}(\sigma^{\ge r}_{a_0}X^\bullet)^n\oplus\operatorname{Tot}((\sigma^{\le r-1}_{a_0}X)[a_0]^\bullet)^{n+1} \\
&= \operatorname{Cone}\left(\operatorname{Tot}((\sigma^{\le r-1}_{a_0}X)[a_0]^\bullet)\xrightarrow{\varphi_\ast}\operatorname{Tot}(\sigma^{\ge r}_{a_0}X^\bullet)\right)^n
\quad.
\end{split}
\end{equation}
With respect to the first decomposition in \eqref{eq:prf:mcplx-cone:mapdef}, the differential on $\operatorname{Tot}(X^\bullet)^n$ is written as
\[
d_{\operatorname{Tot}(X^\bullet)}^n
=
\begin{multlined}[t]
\sum_{\substack{|\alpha|=n\\\alpha(a_0)\le r-1}}\sum_{a\in S\setminus\{a_0\}} d^\alpha_a\otimes\rho_a
+ \sum_{\substack{|\alpha|=n\\\alpha(a_0)< r-1}} d^\alpha_{a_0}\otimes\rho_{a_0} \\
+ \sum_{\substack{|\alpha|=n\\\alpha(a_0)=r-1}}d^\alpha_{a_0}\otimes\rho_{a_0}
+ \sum_{\substack{|\alpha|=n\\\alpha(a_0)\ge r}}\sum_{a\in S} d^\alpha_a\otimes\rho_a
\quad.
\end{multlined}
\]
By \cref{lem:rho-lambda-com}, we have
\begin{equation}
\label{eq:prf:mcplx-cone:term-fst}
\begin{split}
&(\mathrm{id}\otimes\rho_{a_0})
\circ\biggl(\sum_{\substack{|\alpha|=n\\\alpha(a_0)\le r-1}}\sum_{a\in S\setminus\{a_0\}} d^\alpha_a\otimes\rho_a+ \sum_{\substack{|\alpha|=n\\\alpha(a_0)<r-1}} d^\alpha_{a_0}\otimes\rho_{a_0}\biggr) \\
&= \biggl(-\sum_{\substack{|\alpha|=n\\\alpha(a_0)\le r-1}}\sum_{a\in S\setminus\{a_0\}} d^\alpha_a\otimes\rho_a+ \sum_{\substack{|\alpha|=n\\\alpha(a_0)<r-1}} d^\alpha_{a_0}\otimes\rho_{a_0}\biggr)\circ(\mathrm{id}\otimes\rho_{a_0}) \\
&= \biggl(-\sum_{|\alpha'|=n+1}\sum_{a\in S}(d^{\le r-1}_{[a_0]})^\alpha_a\otimes\rho_a\biggr)\circ(\mathrm{id}\otimes\rho_{a_0})
\end{split}
\end{equation}
and
\begin{gather}
\label{eq:prf:mcplx-cone:term-snd}
\sum_{\substack{|\alpha|=n\\\alpha(a_0)=r-1}}d^\alpha_{a_0}\otimes\rho_{a_0}
= \biggl(\sum_{|\alpha|=n+1}\varphi^\alpha\otimes\mathrm{id}\biggr)\circ(\mathrm{id}\otimes\rho_{a_0})
\quad,\\
\label{eq:prf:mcplx-cone:term-trd}
\sum_{\substack{|\alpha|=n\\\alpha(a_0)\ge r}}\sum_{a\in S} d^\alpha_a\otimes\rho_a
= \sum_{|\alpha|=n}\sum_{a\in S}(d^{\ge r})^\alpha_a\otimes\rho_a
\quad.
\end{gather}
Adding \eqref{eq:prf:mcplx-cone:term-fst}, \eqref{eq:prf:mcplx-cone:term-snd}, and \eqref{eq:prf:mcplx-cone:term-trd}, one can see that \eqref{eq:prf:mcplx-cone:mapdef} is actually an isomorphism of chain complexes.
This completes the proof.
\end{proof}

Using \cref{lem:mcplx-cone}, we obtain several important isomorphisms.

\begin{proposition}
\label{prop:skein-mcone}
Let $\Dresol{G}$, $\Nresol{G}$, $\Presol{G}$, $\Hsmooth{G}$, and $\Vsmooth{G}$ be singular tangle-like graphs which are identical except on a local part where they are of the form as in \cref{tab:skein-graphs}.
\begin{table}[t]
\centering
\begin{tabular}{c|c|c|c|c}
  $\Dresol{G}$ & $\Nresol{G}$ & $\Presol{G}$ & $\Hsmooth{G}$ & $\Vsmooth{G}$ \\\hline\rule[-3.5ex]{0pt}{8ex}
  \diagSingUp & \diagCrossNegUp & \diagCrossPosUp & \diagSmoothW & \diagSmoothUp
\end{tabular}
\caption{Local pictures in skein relations}
\label{tab:skein-graphs}
\end{table}
Then, for a checkerboard coloring $\chi$, the morphisms \eqref{eq:delta-}, \eqref{eq:delta+}, and \eqref{eq:Phi} respectively induces morphisms of chain complexes
\[
\widehat\delta_-:\dblBrac{\Hsmooth{G}^\chi}\to\dblBrac{\Vsmooth{G}^\chi}
\ ,\quad \widehat\delta_+:\dblBrac{\Vsmooth{G}^\chi}\to\dblBrac{\Hsmooth{G}^\chi}
\ ,\quad \widehat\Phi:\dblBrac{\Nresol{G}^\chi}\to\dblBrac{\Presol{G}^\chi}
\]
such that there are isomorphisms of chain complexes below:
\[
\dblBrac{\Nresol{G}^\chi}\cong\operatorname{Cone}(\widehat\delta_-)
\ ,\quad\dblBrac{\Presol{G}^\chi}\cong\operatorname{Cone}(\widehat\delta_+)[1]
\ ,\quad\dblBrac{\Dresol{G}^\chi}\cong\operatorname{Cone}(\widehat\Phi)
\quad.
\]
\end{proposition}
\begin{proof}
Let us denote by $v_\times\in V^4(\Dresol{G})$, $v_-\in V^4(\Nresol{G})$, and $v_+\in V^4(\Presol{G})$ the quadrivalent vertices in \cref{tab:skein-graphs}.
By direct computations, one obtains isomorphisms
\[
\begin{gathered}
\operatorname{Tot}(\sigma^{\le -1}_{v_-}\operatorname{Sm}(\Nresol{G}^\chi)[v_-]^\bullet)
\cong \dblBrac{\Hsmooth{G}^\chi}
\ ,\quad
\operatorname{Tot}(\sigma^{\ge 0}_{v_-}\operatorname{Sm}(\Nresol{G}^\chi)^\bullet)
\cong \dblBrac{\Vsmooth{G}^\chi}
\quad,\\
\operatorname{Tot}((\sigma^{\le 0}_{v_+}\operatorname{Sm}(\Presol{G}^\chi))[v_+]^\bullet)
\cong \dblBrac{\Vsmooth{G}^\chi}[1]
\ ,\quad
\operatorname{Tot}(\sigma^{\ge 1}_{v_+}\operatorname{Sm}(\Presol{G}^\chi)^\bullet)
\cong \dblBrac{\Hsmooth{G}^\chi}[1]
\quad,\\
\operatorname{Tot}((\sigma^{\le -1}_{v_\times}\operatorname{Sm}(\Dresol{G}^\chi))[v_\times]^\bullet)
\cong \dblBrac{\Nresol{G}^\chi}
\ ,\quad
\operatorname{Tot}(\sigma^{\ge 0}_{v_\times}\operatorname{Sm}(\Dresol{G}^\chi)^\bullet)
\cong \dblBrac{\Presol{G}^\chi}
\quad.
\end{gathered}
\]
The result hence follows from \cref{lem:mcplx-cone}.
\end{proof}

The isomorphisms in \cref{prop:skein-mcone} are categorified analogues of the Kauffman skein relation and the Vassiliev skein relation.

\begin{proof}[Proof of \cref{theo:UKH-compare}]
We prove the statement by induction on the number of double points.
As for ordinary tangle diagrams, there is nothing to prove.
Suppose the statement is true for diagrams with double point less than $n$, and let $D$ be a singular tangle diagram with $n$ double point.
For a double point $v\in c^\sharp(D)$, let us denote by $\Nresol{D}$ and $\Presol{D}$ the diagrams obtained by resolving $v$ into negative and positive crossings respectively.
Hence, \cref{prop:skein-mcone} yields an isomorphism
\[
\dblBrac*{D^{\chi_{\mathsf w}}}
\cong \operatorname{Cone}\left(
\dblBrac*{\Nresol{D}^{\chi_{\mathsf w}}}
\xrightarrow{\widehat\Phi} \dblBrac*{\Presol{D}^{\chi_{\mathsf w}}}
\right)
\quad.
\]
By induction hypothesis, the complexes $\dblBrac{\Nresol{D}^{\chi_{\mathsf w}}}$ and $\dblBrac{\Presol{D}^{\chi_{\mathsf W}}}$ are isomorphic to the ones defined in \cite[Definition~5.2]{ItoYoshida2020UKH}.
In addition, unwinding the construction of the morphism $\widehat\Phi$, one can verify that the morphism $\widehat\Phi$ defined in \cref{prop:skein-mcone} agrees with the morphism \eqref{eq:PhiHat} (see \cite[Section~3]{ItoYoshida2020UKH}).
Thus, in view of \cite[Corollary~5.3]{ItoYoshida2020UKH} and \cref{prop:skein-mcone}, we conclude that $\dblBrac{D^{\chi_{\mathsf w}}}$ is also isomorphic to the complex constructed in \cite[Definition~5.2]{ItoYoshida2020UKH}.
This completes the induction and hence the proof.
\end{proof}

\section{Absolute exact sequences}
\label{sec:absex}

Throughout the section, we fix a base ring $k$ and a $k$-linear category $\mathcal A$.

\subsection{Exactness and chain contractions}
\label{sec:absex:excontr}

In the discussion of chain complexes in an arbitrary $k$-linear category, one major problem is that images and kernels may not exist.
In particular, the notion of exactness no longer make sense.
We hence consider the absolute version instead.

\begin{definition}
A sequence
\begin{equation}
\label{eq:absex-seq}
\cdots\xrightarrow{f^{n-2}}X^{n-1}\xrightarrow{f^{n-1}}X^n\xrightarrow{f^n} X^{n+1}\xrightarrow{f^{n+1}}\cdots
\end{equation}
in $\mathcal A$ is said to be \emph{absolutely exact} if every $k$-linear functor $F:\mathcal A\to\mathcal B$ with $\mathcal B$ abelian carries \eqref{eq:absex-seq} into an exact sequence.
\end{definition}

\begin{proposition}
\label{prop:absex-contractible}
A sequence~\eqref{eq:absex-seq} is absolutely exact if and only if it satisfies the following conditions:
\begin{enumerate}[label=\upshape(\roman*)]
  \item\label{sub:absex-contractible:diff} for each $n\in\mathbb Z$, $f^{n+1}f^n=0$;
  \item\label{sub:absex-contractible:contr} the sequence is contractible as a chain complex; more precisely, there is a family of morphisms $\theta^n:X^n\to X^{n-1}$ such that
\begin{equation}
\label{eq:absex-contractible:theta}
f^{n-1}\theta^n + \theta^{n+1}f^n = \mathrm{id}_{X^n}
\end{equation}
for each $n\in\mathbb Z$.
\end{enumerate}
\end{proposition}

The proof is essentially a recursive use of the following lemma.

\begin{lemma}
\label{lem:absex-indstep}
Suppose a sequence~\eqref{eq:absex-seq} is absolutely exact.
Given a morphism $\theta^n:X^n\to X^{n-1}$ satisfying the equation $f^{n-1}\theta^nf^{n-1}=f^{n-1}$, there are morphisms $\theta^{n+1}:X^{n+1}\to X^n$ and $\theta^{n-1}:X^{n-1}\to X^{n-2}$ such that
\[
\theta^{n+1}f^n + f^{n-1}\theta^n = \mathrm{id}_{X^n}
\ ,\quad\theta^nf^{n-1} + f^{n-2}\theta^{n-1} = \mathrm{id}_{X^{n-1}}
\quad.
\]
\end{lemma}
\begin{proof}
Since \eqref{eq:absex-seq} is absolutely exact, we obtain the following exact sequence of $k$-modules:
\[
\mathcal A(X^{n+1},X^n)
\xrightarrow{(f^n)^\ast}\mathcal A(X^n,X^n)
\xrightarrow{(f^{n-1})^\ast}\mathcal A(X^{n-1},X^n)
\quad.
\]
The assumption on $\theta^n$ implies that the element $\mathrm{id}_{X^n}-f^{n-1}\theta^n\in\mathcal A(X^n,X^n)$ belongs to the kernel of the map $(f^{n-1})^\ast$.
It follows that there is an element $\theta^{n+1}\in\mathcal A(X^{n+1},X^n)$ with $(f^n)^\ast(\theta^{n+1})=\mathrm{id}_{X^n}-f^{n-1}\theta^n$; in other words,
\[
\theta^{n+1}f^n + f^{n-1}\theta^n = \mathrm{id}_{X^n}
\quad.
\]
Thus, $\theta^{n+1}$ is one of the required morphisms in the statement.
By the dual argument, one can also construct a morphism $\theta^{n-1}:X^{n-1}\to X^{n-2}$, which completes the proof.
\end{proof}

\begin{proof}[Proof of \cref{prop:absex-contractible}]
The ``if'' part follows from the fact that, in an abelian category, a chain complex is exact if and only if all the homology groups vanish.
Conversely, suppose \eqref{eq:absex-seq} is absolutely exact.
First, for each $n\in\mathbb Z$, we have the following exact sequence of $k$-modules:
\[
\mathcal A(X^n,X^n)
\xrightarrow{f^n_\ast} \mathcal A(X^n,X^{n+1})
\xrightarrow{f^{n+1}_\ast} \mathcal A(X^n,X^{n+2})
\quad.
\]
By chasing the images of the identity, we obtain $f^{n+1}f^n=0$, so the property~\ref{sub:absex-contractible:diff} follows.
On the other hand, in order to prove the property~\ref{sub:absex-contractible:contr}, it suffices to construct a morphism $\theta^1:X^1\to X^0$ with $f^0\theta^1f^0=f^0$.
Indeed, if such a morphism is given, one can inductively construct $\theta^n$ for every $n\in\mathbb Z$ by using \cref{lem:absex-indstep}.

To construct $\theta^1$, we define a functor $I_{f^0}:\mathcal A^\opposite\to\mathbf{Mod}_k$ so that, for each $W\in\mathcal A$, $I_{f^0}(W)$ is the image of the $k$-homomorphism
\[
f^0_\ast:\mathcal A(W,X^0)\to \mathcal A(W,X^1)
\]
and that the inclusion $I_{f^0}(W)\hookrightarrow\mathcal A(W,X^1)$ defines a natural transformation.
Since \eqref{eq:absex-seq} is absolutely exact, it induces the following exact sequence of $k$-modules:
\[
I_{f^0}(X^1)
\xrightarrow{(f^0)^\ast} I_{f^0}(X^0)
\xrightarrow{(f^{-1})^\ast}I_{f^0}(X^{-1})
\quad.
\]
Note that $f^0\in\mathcal A(X^0,X^1)$ lies in the submodule $I_{f^0}(X^0)$ and that the property~\ref{sub:absex-contractible:diff} implies $(f^{-1})^\ast(f^0)=0$.
It follows that there is an element $\beta\in I_{f^0}(X^1)$ with $(f^0)^\ast(\beta)=f^0$.
Furthermore, by definition of $I_{f^0}$, the element $\beta$ can be written in the form $\beta=f^0\theta^1$ for an element $\theta^1\in\mathcal A(X^1,X^0)$.
Finally, we obtain the equation $f^0\theta^1f^0 = f^0$, so $\theta^1$ is a required morphism.
As discussed above, the property~\ref{sub:absex-contractible:contr} hence follows, which completes the proof.
\end{proof}

We mainly use absolute exactness in the form of \cref{prop:absex-contractible}.
The following lemma is convenient.

\begin{lemma}
\label{lem:theta-sqzero}
Suppose we are given an absolutely exact sequence \eqref{eq:absex-seq} in an additive category $\mathcal A$.
Then, there is a family of morphisms $\theta^n:X^n\to X^{n-1}$ such that
\[
f^{n-1}\theta^n+\theta^{n+1}f^n = \mathrm{id}_{X^n}
\ ,\quad \theta^{n-1}\theta^n = 0
\]
for each $n\in\mathbb Z$.
\end{lemma}
\begin{proof}
In view of \cref{prop:absex-contractible}, there is a family of morphisms $\zeta^n:X^n\to X^{n-1}$ such that $f^{n-1}\zeta^n+\zeta^{n+1}f^n=\mathrm{id}$.
Notice that, in this case, the endomorphisms $f^{n-1}\zeta^n,\zeta^{n+1}f^n:X^n\to X^n$ are idempotent; indeed, we have
\begin{equation}
\label{eq:prf:theta-sqzero:idemp}
\begin{gathered}
f^{n-1}\zeta^nf^{n-1}\zeta^n
= f^{n-1}(\mathrm{id}-f^{n-2}\zeta^{n-1})\zeta^n
= f^{n-1}\zeta^n
\quad,\\
\zeta^{n+1}f^n\zeta^{n+1}f^n
= \zeta^{n+1}(\mathrm{id}-\zeta^{n+2}f^{n+1})f^n
= \zeta^{n+1}f^n
\quad.
\end{gathered}
\end{equation}
Now, we set $\theta^n\coloneqq\zeta^nf^{n-1}\zeta^n$.
Then, the equation~\eqref{eq:prf:theta-sqzero:idemp} implies
\[
f^{n-1}\theta^n+\theta^{n+1}f^n
= f^{n-1}\zeta^n + \zeta^{n+1}f^n = \mathrm{id}
\quad.
\]
We in addition have
\[
\begin{split}
\theta^{n-1}\theta^n
&= \zeta^{n-1}f^{n-2}\zeta^{n-1}\zeta^nf^{n-1}\zeta^n \\
&= \zeta^{n-1}(\mathrm{id}-\zeta^nf^{n-1})(\mathrm{id}-f^{n-2}\zeta^{n-1})\zeta^n \\
&= \zeta^{n-1}(\mathrm{id}-(f^{n-2}\zeta^{n-1}+\zeta^nf^{n-1})+\zeta^n f^{n-1}f^{n-2}\zeta^{n-1})\zeta^n \\
&= 0
\quad.
\end{split}
\]
Hence, the result follows.
\end{proof}

\subsection{Some exact sequences of cobordisms}
\label{sec:absex:exseq-cob}

We give several examples of absolutely exact sequences.
Namely, we discuss sequences in the $k$-linear category $\Cob(\mathord-,\mathord-)$ involved with the Reidemeister move of type I and the FI relation.

We first fix notations.
We denote by $I\in\Cob(\mathord-,\mathord-)$ and $S^1\in\Cob(\varnothing,\varnothing)$ the $1$-dimensional cobordisms diffeomorphic to the unit interval and the circle respectively.
Hence, we have
\[
I=\diagFiNil=\diagFiH
\ ,\quad I\otimes S^1=\diagFiV
\quad.
\]
The circle $S^1$ has a canonical structure of a Frobenius monoid object in $\Cob(\varnothing,\varnothing)$; namely, the multiplication $\mu:S^1\otimes S^1\to S^1$ and the comultiplication $\Delta:S^1\to S^1\otimes S^1$ consist of the saddles, the unit $\eta:\varnothing\to S^1$ given by the ``cup'', and the counit $\varepsilon:S^1\to\varnothing$ given by the ``cap.''
On the other hand, we consider the following morphisms in $\Cob(\mathord-,\mathord-)$ defined by the saddles:
\begin{equation}
\label{eq:interval-S1mod}
\mu\coloneqq\BordROneDeltaP:\diagFiV\to\diagFiH
\ ,\quad
\Delta\coloneqq\BordROneDeltaN:\diagFiH\to\diagFiV
\quad.
\end{equation}
Writing $I\in\Cob(\mathord-,\mathord-)$ the interval, one can regard these morphisms as $\mu:I\otimes S^1\to I$ and $\Delta:I\to I\otimes S^1$.

The following is essentially due to Bar-Natan \cite{BarNatan2005}.

\begin{proposition}
\label{prop:ROnePos-absex}
The sequence below is absolutely exact in $\Cob(\mathord-,\mathord-)$:
\begin{equation}
\label{eq:ROnePos-absex:seq}
0\to\diagFiH\xrightarrow{\Delta-\mu\Delta\otimes\eta}\diagFiV\xrightarrow{\mu}\diagFiH\to 0
\quad.
\end{equation}
More precisely, the following defines a chain contraction:
\begin{equation}
\label{eq:ROnePos-absex:contr}
\mathrm{id}_I\otimes\eta:\diagFiH\to\diagFiV
\ ,\quad \mathrm{id}_I\otimes\varepsilon:\diagFiV\to\diagFiH
\quad.
\end{equation}
\end{proposition}
\begin{proof}
We clearly have $\mu(\mathrm{id}_I\otimes\eta)=\mathrm{id}_I$.
Also, the $S$-relation implies $(\mathrm{id}_I\otimes\varepsilon)(\Delta-\mu\Delta\otimes\eta)=\mathrm{id}_I$.
Furthermore, by $4Tu$ relation, we obtain
\[
\begin{split}
&(\Delta-\mu\Delta\otimes\eta)(\mathrm{id}_I\otimes\varepsilon)+(\mathrm{id}_I\otimes\eta)\mu \\
&=\BordFIHmtpFst\;-\;\BordFIHmtpTrd\;+\;\BordFIHmtpSnd\\
&=\; \BordFIVId\; = \mathrm{id}_{I\otimes S^1}
\quad.
\end{split}
\]
This completes the proof.
\end{proof}

The dual argument shows the following.

\begin{proposition}
\label{prop:ROneNeg-absex}
The sequence below is absolutely exact in $\Cob(\mathord-,\mathord-)$:
\begin{equation}
\label{eq:ROneNeg-absex:seq}
0\to\diagFiH\xrightarrow{\Delta}\diagFiV\xrightarrow{\mu-\mu\Delta\otimes\varepsilon}\diagFiH\to 0
\quad.
\end{equation}
More precisely, the morphisms \eqref{eq:ROnePos-absex:contr} also defines a chain contraction for \eqref{eq:ROneNeg-absex:seq}.
\end{proposition}

The exact sequences \cref{prop:ROnePos-absex} and \cref{prop:ROneNeg-absex} compose to yield a longer exact sequence.
Namely, by the $4Tu$ relations, we have
\[
\begin{split}
&(\Delta-\mu\Delta\otimes\eta)(\mu-\mu\Delta\otimes\varepsilon) \\
&= \BordROneDeltaP\BordROneDeltaN - \BordROneNSndV\BordROneDeltaN - \BordROneDeltaP\BordROnePBarFstVarV + \BordROneNSndV\BordROnePBarFstVarV\\
&= \BordFIPhiFst - \BordFIPhiSnd
= \Phi
\quad,
\end{split}
\]
here the last term is the morphism given in \eqref{eq:Phi}.

\begin{proposition}
\label{prop:FI-absex}
The solid part of the following sequence is absolutely exact in $\Cob(\mathord-,\mathord-)$:
\begin{equation}
\label{eq:FI-absex:seq}
\begin{tikzcd}
0 \ar[r] & \diagFiH \ar[r,shift left,"\Delta"] & \diagFiV \ar[r,shift left,"\Phi"] \ar[l,dashed,shift left,bend left,"\mathrm{id}_I\otimes\varepsilon"] & \diagFiV \ar[r,shift left,"\mu"] \ar[l,dashed,shift left,bend left,"\mathrm{id}_I\otimes\eta\varepsilon"] & \diagFiH \ar[r] \ar[l,dashed,shift left,bend left,"\mathrm{id}_I\otimes\eta"] & 0
\end{tikzcd}
\quad,
\end{equation}
where the dashed arrows give a chain contraction.
\end{proposition}

Note that the sequence \eqref{eq:FI-absex:seq} equals (a shift of) the universal Khovanov complex
\begin{equation}
\label{eq:KhOfFI}
\dblBrac*{\diagFiSing}
\quad.
\end{equation}
It follows that \eqref{eq:KhOfFI} is contractible, which is exactly the FI relation of the universal Khovanov complex (\cite[Main Theorem~C]{ItoYoshida2020UKH}).

\subsection{Long exact sequences on homologies}%
\label{sec:absex:longex}

Let $\mathcal A$ be an additive category.
By a \emph{$2$-fold complex}, we mean a $\{\mathsf V,\mathsf H\}$-fold complex for formal symbols $\mathsf V$ and $\mathsf H$.
If $X^\bullet$ is a $2$-fold complex, we write
\[
X^{i,j}\coloneqq X^{i\mathsf H+j\mathsf V}
\quad;
\]
hence the differentials are written as $d^{i,j}_{\mathsf H}:X^{i,j}\to X^{i+1,j}$ and $d^{i,j}_{\mathsf V}:X^{i,j}\to X^{i,j+1}$, which we call \emph{horizontal} and \emph{vertical} differentials respectively.
For each integer $i\in\mathbb Z$, we write $X^{i,\bullet}=\{X^{i,j},d_{\mathsf V}^{i,j}\}_j$ the restricted chain complex.
In this point of view, the family $d_{\mathsf H}^i=\{d_{\mathsf H}^{i,j}\}_j$ defines a morphism of chain complexes $d_{\mathsf H}^i:X^{i,\bullet}\to X^{i+1,\bullet}$, where we often drop indices.
Consequently, the $2$-fold complex $X^\bullet$ can be seen as a sequence
\begin{equation}
\label{eq:horizontal-seq}
\cdots
\xrightarrow{d_{\mathsf H}} X^{i-1,\bullet}
\xrightarrow{d_{\mathsf H}} X^{i,\bullet}
\xrightarrow{d_{\mathsf H}} X^{i+1,\bullet}
\xrightarrow\cdots
\end{equation}
of morphisms of complexes in $\mathcal A$ with $d^2_{\mathsf H}=0$.

\begin{remark}
Some authors use the word ``bicomplex'' (or even ``double complex'') for our $2$-fold complex though they sometimes refer to a family $\{X^{i,j}\}$ with differential $d_{\mathsf H}:X^{i,j}\to X^{i+1,j}$ and $d_{\mathsf V}:X^{i,j}\to X^{i,j+1}$ such that $d_{\mathsf H}d_{\mathsf V}+d_{\mathsf V}d_{\mathsf H}=0$ instead of $d_{\mathsf H}d_{\mathsf V}=d_{\mathsf V}d_{\mathsf H}$.
For this reason, we avoid these terminologies in this paper.
\end{remark}

Especially, we are interested in the case where \eqref{eq:horizontal-seq} is absolutely exact in each degree; in this case, we see that it gives rise to long exact sequences.
More precisely, the goal of the subsection is to prove the following result.

\begin{proposition}
\label{prop:dblcplx-longex}
Let $X^\bullet=\{X^{i,j},d_{\mathsf H},d_{\mathsf V}\}$ be a $2$-fold complex in a $k$-linear category $\mathcal A$ such that the sequence \eqref{eq:horizontal-seq} is absolutely exact in each degree.
Then, for every pair of integers $p,q\in\mathbb Z$ with $q-p>1$, there is a morphism of chain complexes $\Xi:\operatorname{Tot}(\sigma^{\ge q}_{\mathsf H}X^\bullet)\to\operatorname{Tot}(\sigma^{\le p}_{\mathsf H}X^\bullet)[1]$ together with an isomorphism
\[
\operatorname{Tot}(\sigma^{\ge p+1}_{\mathsf H}\sigma^{\le q-1}_{\mathsf H}X^\bullet)
\cong \operatorname{Cone}\left(\Xi\right)
\quad.
\]
\end{proposition}

\begin{remark}
In view of \cref{cor:tensor-ES-alt}, for a bounded $2$-fold complex $X^\bullet=\{X^{i,j},d_{\mathsf H}^{i,j},d_{\mathsf V}^{i,j}\}$, we realize its total complex as
\[
\operatorname{Tot}(X^\bullet)^n
= \bigoplus_{i+j=n}X^{i,j}
\]
with differential $d:\operatorname{Tot}(X^\bullet)^n\to\operatorname{Tot}(X^\bullet)^{n+1}$ given by
\[
d^{i,j}_{\mathsf H}+(-1)^i d^{i,j}_{\mathsf V}:X^{i,j}\to X^{i+1,j}\oplus X^{i,j+1}
\quad.
\]
\end{remark}

In order to prove \cref{prop:dblcplx-longex}, we construct a morphism $\Xi$ in the statement.

\begin{definition}
\label{def:Theta}
Let $\{X^{i,j},d_{\mathsf V}^{i,j},d_{\mathsf H}^{i,j}\}_{i,j}$ be a $2$-fold complex in $\mathcal A$ such that the sequence \eqref{eq:horizontal-seq} is absolutely exact in each degree; i.e.~for each $j\in\mathbb Z$, there is a family of morphisms $\theta_{\mathsf H}^j=\{\theta^{i,j}_{\mathsf H}:X^{i,j}\to X^{i-1,j}\}_i$ such that
\begin{gather}
\label{eq:theta-contraction}
d_{\mathsf H}^{i-1,j}\theta_{\mathsf H}^{i,j} + \theta_{\mathsf H}^{i+1,j}d_{\mathsf H}^{i,j} = \mathrm{id}_{X^{i,j}}
\quad,\\[1ex]
\label{eq:theta-sqzero}
\theta_{\mathsf H}^{i-1,j}\theta_{\mathsf H}^{i,j}=0
\quad,
\end{gather}
for each $i,j\in\mathbb H$ (see~\cref{lem:theta-sqzero}).
In this case, for each positive integer $r\in\mathbb Z$, we define a morphism $\Theta_r^{i,j}:X^{i,j}\to X^{i-r,j+r-1}$ inductively by
\begin{equation}
\label{eq:Theta}
\Theta_1^{i,j}\coloneqq\theta^{i,j}_{\mathsf H}
\ ,\quad
\Theta_{r+1}^{i,j}\coloneqq\theta_{\mathsf H}d_{\mathsf H}\Theta_r^{i,j}
\quad.
\end{equation}
\end{definition}

\begin{remark}
Even if the sequence~\eqref{eq:horizontal-seq} is absolutely exact in each degree, it is not necessarily an absolute exact sequence in $\mathbf{Ch}(\mathcal A)$.
Indeed, the family $\theta_{\mathsf H}^j$ in \cref{def:Theta} is not a morphism of chain complexes in general.
\end{remark}

\begin{lemma}
\label{lem:Theta-diff}
Let $\{X^{i,j},d_{\mathsf V}^{i,j},d_{\mathsf H}^{i,j}\}$ be as in \cref{def:Theta}.
Then, for each $i,j\in\mathbb Z$, we have
\begin{equation}
\label{eq:Theta-diff:dvh}
d_{\mathsf H}\Theta_{r+1}-d_{\mathsf V}\Theta_r
= (-1)^{r+1}(\Theta_{r+1}d_{\mathsf H} - \Theta_r d_{\mathsf V})
\end{equation}
\end{lemma}
\begin{proof}
We prove the statement by induction on $r$.
For the base case $r=1$, we have
\[
\begin{multlined}[b]
d_{\mathsf H}\Theta_2 - d_{\mathsf V}\theta_{\mathsf H}
= d_{\mathsf H}\theta_{\mathsf H}d_{\mathsf V}\theta_{\mathsf H}
-(d_{\mathsf H}\theta_{\mathsf H}+\theta_{\mathsf H}d_{\mathsf H})d_{\mathsf V}\theta_{\mathsf H} \\
= -\theta_{\mathsf H}d_{\mathsf V}d_{\mathsf H}\theta_{\mathsf H}
= \theta_{\mathsf H}d_{\mathsf V}\theta_{\mathsf H}d_{\mathsf H}
- \theta_{\mathsf H}d_{\mathsf V}
= \Theta_2d_{\mathsf H} - \theta_{\mathsf H}d_{\mathsf V}
\end{multlined}
\quad.
\]
Hence, the equation~\eqref{eq:Theta-diff:dvh} holds for $r=1$.
On each induction step for $r\ge 2$, notice that the induction hypothesis implies
\[
d_{\mathsf H}\Theta_r
= d_{\mathsf V}\Theta_{r-1}+(-1)^r\left(\Theta_rd_{\mathsf H} - \Theta_{r-1}d_{\mathsf V}\right)
\quad.
\]
Thus, we obtain
\[
\begin{split}
d_{\mathsf H}\Theta_{r+1} - d_{\mathsf V}\Theta_r
&= d_{\mathsf H}\theta_{\mathsf H}d_{\mathsf V}\Theta_r
- (d_{\mathsf H}\theta_{\mathsf H}+\theta_{\mathsf H}d_{\mathsf H})d_{\mathsf V}\Theta_r \\
&= -\theta_{\mathsf H}d_{\mathsf V}d_{\mathsf H}\Theta_r \\
&= (-1)^{r+1}\left(\theta_{\mathsf H}d_{\mathsf V}\Theta_rd_{\mathsf H} - \theta_{\mathsf H}d_{\mathsf V}\Theta_{r-1}d_{\mathsf V}\right) \\
&= (-1)^{r+1}\left(\Theta_{r+1}d_{\mathsf H} - \Theta_rd_{\mathsf V}\right)
\quad,
\end{split}
\]
which completes the induction.
\end{proof}

\begin{remark}
In the following arguments, it is convenient to consider a normalized version
\begin{equation}
\label{eq:Theta-tilde}
\widetilde\Theta_r^{i,j}\coloneqq(-1)^{\frac12(r+1)(r+2i)}\Theta_r^{i,j}
:X^{i,j}\to X^{i-r,j+r-1}
\end{equation}
instead of $\Theta_r$ itself.
In terms of this, the equation~\eqref{eq:Theta-diff:dvh} is equivalent to the following:
\begin{equation}
\label{eq:Theta-tilde-dvh}
d_{\mathsf H}\widetilde\Theta^{i,j}_{r+1}
+(-1)^{i-r}d_{\mathsf V}\widetilde\Theta^{i,j}_r
= -\widetilde\Theta_{r+1}^{i+1,j}d_{\mathsf H}
- (-1)^i\widetilde\Theta^{i,j+1}_r d_{\mathsf V}
\quad.
\end{equation}
\end{remark}

Now, let $p$ and $q$ be integers with $q-p>1$.
For a bounded $2$-fold complex $X^\bullet=\{X^{i,j},d_{\mathsf H},d_{\mathsf V}\}$, we define a morphism $\Xi^n:\operatorname{Tot}(\sigma_{\mathsf H}^{\ge q}X^\bullet)^n\to \operatorname{Tot}(\sigma_{\mathsf H}^{\le p}X^\bullet)^{n-1}$ by
\[
-\sum_{r=i-p}^\infty \widetilde\Theta_r^{i,j}:
X^{i,j}\to\bigoplus_{\substack{i+j=n-1\\i\le p}} X^{i,j}
\quad.
\]

\begin{lemma}
\label{lem:Xi-chain}
In the situation above, the family $\Xi=\{\Xi^n\}_n$ forms a morphism of chain complexes $\Xi:\operatorname{Tot}(\sigma_{\mathsf H}^{\ge q}X^\bullet)\to\operatorname{Tot}(\sigma_{\mathsf V}^{\le p}X^\bullet)[1]$.
\end{lemma}
\begin{proof}
For each $i,j\in\mathbb Z$ with $i\ge q$, we have
\[
\begin{split}
d_{[1]}\circ\left.\Xi\right|_{X^{i,j}}
&= \sum_{r=i-p+1}^\infty d_{\mathsf H}\circ\widetilde\Theta_r
\sum_{r=i-p}^\infty(-1)^{i-r}d_{\mathsf V}\circ\widetilde\Theta_r
\\
&= \sum_{r=i-p}^\infty\left(d_{\mathsf H}\widetilde\Theta_{r+1}+(-1)^{i-r}d_{\mathsf V}\widetilde\Theta_r\right)
\quad.
\end{split}
\]
On the other hand, we also have
\[
\begin{split}
\Xi\circ\left.d\right|_{X^{i,j}}
&= -\sum_{r=i-p+1}^\infty\widetilde\Theta_r\circ d_{\mathsf H}
- \sum_{r=i-p}^\infty\widetilde\Theta_r\circ(-1)^id_{\mathsf V} \\
&= -\sum_{r=i-p}^\infty\left(\Theta_{r+1}d_{\mathsf H}+(-1)^i\widetilde\Theta_rd_{\mathsf V}\right)
\quad.
\end{split}
\]
Thus, the result follows from \eqref{eq:Theta-tilde-dvh}.
\end{proof}

As a consequence of \cref{lem:Xi-chain}, we can consider the mapping cone $\operatorname{Cone}(\Xi)$ as a complex in $\mathcal A$.
More explicitly, for each $n\in\mathbb Z$, we have
\begin{equation}
\label{eq:CXi-explicit}
\begin{split}
\operatorname{Cone}(\Xi)^n
&=\operatorname{Tot}(\sigma_{\mathsf H}^{\le p}X)[1]^n\oplus\operatorname{Tot}(\sigma_{\mathsf H}^{\ge q}X)^{n+1} \\[1ex]
&\cong \bigoplus_{\substack{i+j=n-1\\i\le p}}X^{i,j}
\oplus\bigoplus_{\substack{i+j=n+1\\i\ge q}}X^{i,j}
\quad.
\end{split}
\end{equation}
Specifically, it comes equipped with the following exact sequence of morphisms of chain complexes in $\mathcal A$:
\begin{equation}
\label{eq:coneXi-ex}
0
\to\operatorname{Tot}(\sigma_{\mathsf H}^{\le p}X)[1]
\xrightarrow{\iota} \operatorname{Cone}(\Xi)
\xrightarrow{\pi} \operatorname{Tot}(\sigma_{\mathsf H}^{\ge q}X)[-1]
\to 0
\quad.
\end{equation}
On the other hand, in view of \cref{lem:mcplx-cone}, we also have morphisms of chain complexes below:
\begin{equation}
\label{eq:Xi-trmorph}
\begin{gathered}
\widehat\varphi_{(p)}:\operatorname{Tot}(\sigma_{\mathsf H}^{\le p}X)[1]\to\operatorname{Tot}(\sigma_{\mathsf H}^{\ge p+1}\sigma_{\mathsf H}^{\le q-1}X)
\quad,\\
\widehat\varphi_{(q)}:\operatorname{Tot}(\sigma_{\mathsf H}^{\ge p+1}\sigma_{\mathsf H}^{\le q-1}X)\to \operatorname{Tot}(\sigma_{\mathsf H}^{\ge q}X)[-1]
\quad.
\end{gathered}
\end{equation}
These morphisms together with the morphism $\Theta_r$ defined in \cref{eq:Theta} (or $\widetilde\Theta_r$ in \cref{eq:Theta-tilde}) are the main ingredients of the chain homotopy equivalences stated in \cref{prop:dblcplx-longex}.

\begin{lemma}
\label{lem:CXi2tot}
In the situation above, for each $n\in\mathbb Z$, we define a morphism $\alpha^n:\operatorname{Cone}(\Xi)^n\to\operatorname{Tot}(\sigma_{\mathsf H}^{\ge p+1}\sigma_{\mathsf H}^{\le q-1} X)^n$ in $\mathcal A$ by
\[
\left.\alpha^n\right|_{X^{i,j}}
\coloneqq
\begin{cases*}
\widehat\varphi_{(p)} & on $\operatorname{Tot}(\sigma_{\mathsf H}^{\le p}X)^{n-1}$,\\[1ex]
\displaystyle \sum_{r=i-q+1}^{i-p-1}\widetilde\Theta_r^{i,j} & on $X^{i,j}$ with $i \ge q$.
\end{cases*}
\]
Then, it defines a morphism of chain complexes $\operatorname{Cone}(\Xi)\to\operatorname{Tot}(\sigma_{\mathsf H}^{\ge p+1}\sigma_{\mathsf H}^{\le q-1}X)$.
\end{lemma}
\begin{proof}
Since $\alpha$ is a morphism of chain complexes on the subcomplex $\operatorname{Tot}(\sigma^{\le p}X)[1]\subset\operatorname{Cone}(\Xi)$, it will be enough to show that it commutes with the differentials on each $X^{i,j}$ with $i\ge q$.
In this case, we have
\begin{equation}
\label{eq:prf:CXi2tot:da}
\begin{split}
d\circ\left.\alpha\right|_{X^{i,j}}
&= \sum_{r=i-q+2}^{i-p-1}d_{\mathsf H}\widetilde\Theta_r^{i,j} + \sum_{r=i-q+1}^{i-p-1}(-1)^{i-r}d_{\mathsf V}\widetilde\Theta_r^{i,j} \\
&= \sum_{r=i-q+1}^{i-p-2}\left(d_{\mathsf H}\widetilde\Theta_{r+1}^{i,j}+(-1)^{i-r}d_{\mathsf V}\widetilde\Theta_r^{i,j}\right) + (-1)^{p+1}d_{\mathsf V}\widetilde\Theta_{i-p-1}^{i,j}
\quad.
\end{split}
\end{equation}
On the other hand,
\begin{equation}
\label{eq:prf:CXi2tot:ad}
\begin{split}
\alpha\circ\left.d\right|_{X^{i,j}}
&= \alpha\circ\biggl(-d_{\mathsf H}-(-1)^id_{\mathsf V}-\sum_{r=i-p}^\infty\widetilde\Theta_r^{i,j}\biggr) \\
&= -\sum_{r=i-q+2}^{i-p}\widetilde\Theta^{i+1,j}_rd_{\mathsf H}
- (-1)^i\sum_{r=i-q+1}^{i-p-1}\widetilde\Theta^{i,j}_rd_{\mathsf V}
- d_{\mathsf H}\widetilde\Theta_{i-p}^{i,j} \\
&= -\sum_{r=i-q+1}^{i-p-1}\bigl(\widetilde\Theta^{i+1,j}_{r+1}d_{\mathsf H}+(-1)^i\widetilde\Theta^{i,j}_rd_{\mathsf V}\bigr)-d_{\mathsf H}\widetilde\Theta_{i-p}^{i,j}
\end{split}
\end{equation}
By virtue of the equation~\eqref{eq:Theta-tilde-dvh}, the comparison of \eqref{eq:prf:CXi2tot:da} with \eqref{eq:prf:CXi2tot:ad} shows that $\alpha$ commutes with the differentials, which completes the proof.
\end{proof}

\begin{lemma}
\label{lem:tot2CXi}
In the situation above, for each $n\in\mathbb Z$, we define a morphism $\beta^n:\operatorname{Tot}(\sigma^{\ge p+1}_{\mathsf H}\sigma_{\mathsf H}^{\le q-1}X)^n\to\operatorname{Cone}(\Xi)^n$ as follows: for $p+1\le i\le q-1$ with $i+j=n$,
\begin{equation}
\label{eq:tot2CXi-mor}
\left.\beta^n\right|_{X^{i,j}}
\coloneqq\sum_{r=i-p}^\infty\widetilde\Theta_r^{i,j}+\widehat\varphi_{(q)}:
X^{i,j}\to \bigoplus_{\substack{i+j=n-1\\i\le p}}X^{i,j}\oplus\operatorname{Tot}(\sigma_{\mathsf H}^{\ge q}X)^{n+1}
\quad.
\end{equation}
Then, it defines a morphism of chain complexes $\operatorname{Tot}(\sigma^{\ge p+1}_{\mathsf H}\sigma^{\le q-1}_{\mathsf H}X)\to\operatorname{Cone}(\Xi)$.
\end{lemma}
\begin{proof}
We show that $\beta$ commutes with the differentials on each $X^{i,j}$ with $p+1\le i\le q-1$.
First, for $i=q-1$, then, by \eqref{eq:Theta-tilde-dvh}, we have
\[
\begin{split}
&d\circ\left.\beta\right|_{X^{q-1,j}} \\
&=
-\sum_{\mathclap{r=q-p}}^\infty d_{\mathsf H}\widetilde\Theta_r^{q-1,j}
-\sum_{\mathclap{r=q-p-1}}^\infty(-1)^{q-r-1}d_{\mathsf V}\widetilde\Theta_r^{q-1,j}
- (-1)^qd_{\mathsf V}d_{\mathsf H}
- \sum_{\mathclap{r=q-p}}^\infty\widetilde\Theta_r^{q,j}d_{\mathsf H}
\\
&= -\sum_{\mathclap{r=q-p-1}}^\infty\;\left(d_{\mathsf H}\widetilde\Theta_{r+1}^{q-1,j}+(-1)^{q-r-1}d_{\mathsf V}\widetilde\Theta_r^{q-1,j}+\widetilde\Theta^{q,j}_{r+1}d_{\mathsf H}\right) - (-1)^qd_{\mathsf V}^{q,j}d_{\mathsf H}^{q-1,j}
\\
&= \sum_{r=q-p-1}^\infty(-1)^{q-1}\widetilde\Theta^{q-1,j+1}_r d_{\mathsf V} + (-1)^{q-1}d_{\mathsf H}^{q-1,j+1}d_{\mathsf V}^{q-1,j}
\\
&= \beta\circ \left.d\right|_{X^{q-1,j}}
\quad.
\end{split}
\]
Next, if $p+1\le i\le q-2$, the similar computation shows
\begin{equation}
\label{eq:prf:tot2CXi:db}
d\circ\left.\beta\right|_{X^{i,j}}
=-\sum_{\mathclap{r=i-p}}^\infty\;\left(d_{\mathsf H}\widetilde\Theta_{r+1}^{i,j}+(-1)^{i-r}d_{\mathsf V}\widetilde\Theta_r^{i,j}\right)
\quad.
\end{equation}
On the other hand, using the equation $\widehat\varphi_{(q)}\circ d_{\mathsf H}=0$, we also have
\begin{equation}
\label{eq:prf:tot2CXi:bd}
\begin{split}
\beta\circ \left.d\right|_{X^{i,j}}
&= \sum_{r=i-p+1}^\infty\widetilde\Theta^{i+1,j}_rd_{\mathsf H} + \sum_{r=i-p}^\infty(-1)^i\widetilde\Theta^{i,j+1}_rd_{\mathsf V} \\
&= \sum_{r=i-p}^\infty\left(\widetilde\Theta^{i+1,j}_{r+1}d_{\mathsf H}+(-1)^i\widetilde\Theta^{i,j+1}_r d_{\mathsf V}\right)
\quad.
\end{split}
\end{equation}
Comparing \eqref{eq:prf:tot2CXi:db} and \eqref{eq:prf:tot2CXi:bd}, one obtains $d\beta=\beta d$ on $X^{i,j}$ with $p+1\le i\le q-2$ by \eqref{eq:Theta-tilde-dvh}.
This together with the first case proves that $\beta$ is a morphism of chain complexes, as required.
\end{proof}

\begin{proof}[Proof of \cref{prop:dblcplx-longex}]
As a result of the arguments above, we now have a pair of morphisms of complexes below:
\[
\alpha:\operatorname{Cone}(\Xi)\rightleftarrows \operatorname{Tot}(\sigma^{\ge p+1}_{\mathsf H}\sigma_{\mathsf H}^{\le q-1}X):-\beta
\quad.
\]
Hence, it suffices to show that these morphisms are mutually chain homotopy inverses.

We first show that $-\alpha\beta$ is chain homotopic to the identity.
Notice that, for each $p+1\le i\le q-1$, we have
\begin{equation}
\label{eq:prf:dblcplx-longex:ab}
\alpha\circ\left.\beta\right|_{X^{i,j}}
=
\begin{cases}
d_{\mathsf H}\widetilde\Theta^{i,j}_{i-p} &\quad p+1\le i \le q-2\quad, \\[1ex]
\displaystyle d_{\mathsf H}\widetilde\Theta^{i,j}_{q-p-1}+\sum_{r=1}^{q-p-1}\widetilde\Theta^{q,j}_rd_{\mathsf H} &\quad i=q-1\quad.
\end{cases}
\end{equation}
We define a morphism $\Psi^n:\operatorname{Tot}(\sigma_{\mathsf H}^{\ge p+1}\sigma_{\mathsf H}^{\le q-1}X)^n\to \operatorname{Tot}(\sigma_{\mathsf H}^{\ge p+1}\sigma_{\mathsf H}^{\le q-1}X)^{n+1}$ so that, for each $p+1\le i\le q-1$ with $i+j=n$,
\[
\left.\Psi^n\right|_{X^{i,j}}
= \sum_{r=1}^{i-p-1}\widetilde\Theta_r^{i,j}:X^{i,j}\to\bigoplus_{i+j=n+1}X^{i,j}
\quad.
\]
For $p+1\le i\le q-1$, we have
\begin{equation}
\label{eq:prf:dblcplx-longex:dPsi}
\begin{split}
&\left.(d\Psi^n+\Psi^{n+1}d)\right|_{X^{i,j}} +\left.\alpha\beta\right|_{X^{i,j}} \\
&= \sum_{r=1}^{i-p-1}(d_{\mathsf H}+(-1)^{i-r}d_{\mathsf V})\widetilde\Theta^{i,j}_r
+ \sum_{r=1}^{i-p}\widetilde\Theta^{i+1,j}_rd_{\mathsf H}
+ \sum_{r=1}^{i-p-1}(-1)^i\widetilde\Theta^{i,j+1}_rd_{\mathsf V}
+ d_{\mathsf H}\widetilde\Theta^{i,j}_{i-p}\\
&=
\begin{multlined}[t]
\sum_{r=1}^{i-p-1}\left(d_{\mathsf H}\widetilde\Theta_{r+1}^{i,j}+(-1)^{i-r}d_{\mathsf V}\widetilde\Theta^{i,j}_r + \widetilde\Theta_{r+1}^{i+1,j}d_{\mathsf H}+(-1)^i\widetilde\Theta_r^{i,j+1}d_{\mathsf V}\right) \\
+ d_{\mathsf H}\widetilde\Theta_1^{i,j} + \widetilde\Theta_1^{i+1,j}d_{\mathsf H}
\quad.
\end{multlined}
\end{split}
\end{equation}
Since $\widetilde\Theta_1=-\theta$, by virtue of \eqref{eq:theta-contraction}, \eqref{eq:Theta-tilde-dvh}, and \eqref{eq:prf:dblcplx-longex:ab}, the equations~\eqref{eq:prf:dblcplx-longex:dPsi} yield the equation $d\Psi+\Psi d = -\alpha\beta-\mathrm{Id}$; in other words, $\Psi$ is a chain homotopy from the identity to $-\alpha\beta$.

It remains to show that $-\beta\alpha$ is also chain homotopic to the identity.
For $r,s\ge 1$, the equation~\eqref{eq:theta-sqzero} yields $\Theta_s\Theta_r=0$.
Hence, we obtain
\begin{equation}
\label{eq:prf:dblcplx-longex:ba}
\beta\circ\left.\alpha\right|_{X^{i,j}}=
\begin{cases}
0 & i< p\ ,\\[1ex]
\displaystyle\sum_{r=1}^\infty\widehat\Theta^{p+1,j}_r d_{\mathsf H} & i=p\ ,\\[3ex]
\displaystyle d_{\mathsf H}\widetilde\Theta^{i,j}_{i-q+1} & i\ge q\ .
\end{cases}
\end{equation}
We define a morphism $\Phi^n:\operatorname{Cone}(\Xi)^n\to\operatorname{Cone}(\Xi)^{n-1}$ by
\[
\left.\Phi^n\right|_{X^{i,j}}\coloneqq
\begin{cases}
\displaystyle \sum_{r=1}^{i-q}\widetilde\Theta_r^{i,j} & i\ge q\,\text{and}\,i+j=n+1\ ,\\[3ex]
\displaystyle \sum_{r=1}^\infty\widetilde\Theta_r^{i,j} & i\le p\,\text{and}\,i+j=n-1\ .
\end{cases}
\]
We show that the family $\Phi=\{\Phi^n\}$ forms a chain homotopy from the identity to $-\beta\alpha$.
Notice that, by the equation~\eqref{eq:theta-sqzero}, the compositions $\Xi\Phi$ and $\Phi\Xi$ both vanish.
Hence, for $i\le p$, we obtain
\begin{equation}
\label{eq:prf:dblcplx-longex:dPhi-cod}
\begin{split}
&\left.(d\Phi+\Phi d)\right|_{X^{i,j}} + \left.\beta\alpha\right|_{X^{i,j}} \\
&= \sum_{r=1}^\infty(d_{\mathsf H}+(-1)^{i-r})\widetilde\Theta_r^{i,j}
+ \sum_{r=1}^\infty\widetilde\Theta_r^{i+1,j}d_{\mathsf H} + \sum_{r=1}^\infty(-1)^i\widetilde\Theta_r^{i,j+1}d_{\mathsf V} \\
&= 
\begin{multlined}[t]
\sum_{r=1}^\infty\left(d_{\mathsf H}\widetilde\Theta^{i,j}_{r+1}+(-1)^{i-r}\widetilde\Theta_r^{i,j}+\widetilde\Theta^{i+1,j}_{r+1}d_{\mathsf H}+(-1)^i\widetilde\Theta^{i,j+1}_rd_{\mathsf V}\right) \\
+ d_{\mathsf H}\widetilde\Theta^{i,j}_1+ \widetilde\Theta^{i+1,j}_1d_{\mathsf H}
\end{multlined}
\end{split}
\end{equation}
On the other hand, for $i\ge q$, we obtain
\begin{equation}
\label{eq:prf:dblcplx-longex:dPhi-dom}
\begin{split}
&\left.(d\Phi+\Phi d)\right|_{X^{i,j}} + \left.\beta\alpha\right|_{X^{i,j}} \\
&= \sum_{r=1}^{i-q}(d_{\mathsf H}+(-1)^{i-r})\widetilde\Theta_r^{i,j}
+ \sum_{r=1}^{i-q+1}\widetilde\Theta_r^{i+1,j}d_{\mathsf H} + \sum_{r=1}^{i-q}(-1)^i\widetilde\Theta_r^{i,j+1}d_{\mathsf V} + d_{\mathsf H}\widetilde\Theta^{i,j}_{i-q+1} \\
&=
\begin{multlined}[t]
\sum_{r=1}^{i-q}\left(d_{\mathsf H}\widetilde\Theta^{i,j}_{r+1}+(-1)^{i-r}\widetilde\Theta_r^{i,j}+\widetilde\Theta^{i+1,j}_{r+1}d_{\mathsf H}+(-1)^i\widetilde\Theta^{i,j+1}_rd_{\mathsf V}\right) \\
+ d_{\mathsf H}\widetilde\Theta^{i,j}_1+ \widetilde\Theta^{i+1,j}_1d_{\mathsf H}
\end{multlined}
\end{split}
\end{equation}
Hence, by \eqref{eq:theta-contraction}, \eqref{eq:Theta-tilde-dvh}, and \eqref{eq:prf:dblcplx-longex:ab}, we obtain $d\Phi+\Phi d = -\beta\alpha-\mathrm{Id}$.
This implies $-\beta\alpha$ is chain homotopic to the identity, which completes the proof.
\end{proof}

\section{The first Vassiliev derivative}\label{sec:fstder}

In this section, we compute the universal Khovanov homology for links with single double points.
Note that it can be seen as the \emph{first Vassiliev derivative} of the universal Khovanov homology from the viewpoint that the last isomorphism in \cref{prop:skein-mcone} categorifies the Vassiliev skein relation.

\subsection{Twisted $S^1$-modules}\label{sec:fstder:twist}

We introduce ``twisted'' action and coaction of $S^1$ on the interval $I$.
To be more precise, we note that, in terms of the functor \eqref{eq:disj-tensor}, we have a $k$-linear functor
\begin{equation}
\label{eq:S1-freemod}
(\blank)\otimes S^1:\Cob(Y_0,Y_1)\to\Cob(Y_0,Y_1)
\end{equation}
for every oriented $0$-manifold $Y_0$ and $Y_1$.
As $S^1$ is a Frobenius monoid object in $\Cob(\varnothing,\varnothing)$ in the canonical way, we say an object $W\in\Cob(Y_0,Y_1)$ is an \emph{$S^1$-module} (resp.~\emph{$S^1$-comodule}) if it is equipped with morphisms $\mu^{}_W:W\otimes S^1\to W$ (resp.~$\Delta_W:W\to W\otimes S^1$) which makes the diagrams below commute:
\[
\begin{gathered}
\begin{tikzcd}
W\otimes S^1\otimes S^1 \ar[r,"\mu^{}_W\otimes S^1"] \ar[d,"W\otimes\mu"'] & W\otimes S^1 \ar[d,"\mu^{}_W"] \\
W\otimes S^1 \ar[r,"\mu^{}_W"] & W
\end{tikzcd}
\ ,\quad
\begin{tikzcd}[column sep=1em]
& W \ar[dl,"W\otimes\eta"'] \ar[dr,equal] & \\
W\otimes S^1 \ar[rr,"\mu^{}_W"] && W
\end{tikzcd}
\\
\mathllap{\biggr(resp.~\quad}
\begin{tikzcd}
W \ar[r,"\Delta_W"] \ar[d,"\Delta_W"'] & W\otimes S^1 \ar[d,"\Delta_W\otimes S^1"] \\
W\otimes S^1 \ar[r,"W\otimes S^1"] & W\otimes S^1\otimes S^1
\end{tikzcd}
\ ,\quad
\begin{tikzcd}[column sep=1em]
W \ar[rr,"\Delta_W"] \ar[dr,equal] && W\otimes S^1 \ar[dl,"W\otimes\varepsilon"] \\
& W &
\end{tikzcd}
\mathrlap{\quad\biggr).}
\end{gathered}
\]

Actually, on every object $W\in\Cob(Y_0,Y_1)$, $S^1$-module structures and $S^1$-comodule structures correspond in one-to-one.
Indeed, the functor~\eqref{eq:S1-freemod} is self-adjoint; namely, for cobordisms $W,W':Y_0\to Y_1$, we have the following bijection:
\begin{equation}
\label{eq:S1-selfadj}
\begin{array}[t]{ccc}
  \Cob(Y_0,Y_1)(W_0\otimes S^1,W_1) &\to& \Cob(Y_0,Y_1)(W_0,W_1\otimes S^1) \\
  S &\mapsto& (S\otimes S^1)(W_0\otimes\Delta\eta)
\end{array}
\quad.
\end{equation}
As easily verified, this adjunction associates an $S^1$-module structure $\mu^{}_W:W\otimes S^1\to W$ with an $S^1$-comodule structure $\Delta_W:W\to W\otimes S^1$.

\begin{remark}
The correspondence above can be generalized to the one between algebras and coalgebras over an arbitrary Frobenius monad (see \cite{Street2004} for details).
\end{remark}

\begin{definition}
\label{def:S1mod-compatible}
For an object $W\in\Cob(Y_0,Y_1)$, we say that an $S^1$-module structure and an $S^1$-comodule structure on $W$ are \emph{compatible} if they correspond with each other in the sense above.
\end{definition}

The compatibility implies not only that two structures correspond in a special way but that a variant of the Frobenius relation holds.

\begin{lemma}
\label{lem:S1mod-unique}
Let $W\in\Cob(Y_0,Y_1)$ be an object equipped with an $S^1$-module structure $\mu^{}_W:W\otimes S^1\to W$ and an $S^1$-comodule structure $\Delta_W:W\to W\otimes S^1$ which are compatible with each other.
Then, the following equations hold:
\[
(\mu^{}_W\otimes S^1)(W\otimes\Delta)
= \Delta_W\mu^{}_W
= (W\otimes\mu)(\Delta_W\otimes S^1)
:W\otimes S^1\to W\otimes S^1
\quad.
\]
\end{lemma}
\begin{proof}
It suffices to show that the three morphisms are identified by the map~\eqref{eq:S1-selfadj}.
Using the Frobenius relation on $S^1$ and the compatibility of $\mu_W$ and $\Delta_W$, one can see that the first and the third morphisms are mapped to the morphism $(W\otimes\Delta)(\Delta_W\otimes S^1)$ while the second to $(\Delta_W\otimes S^1)(W\otimes\Delta)$.
Therefore, the result follows from the co-associativity of $\Delta$ and $\Delta_W$.
\end{proof}

For example, the morphisms~\eqref{eq:interval-S1mod} exhibit the interval $I\in\Cob(\mathord-,\mathord-)$ simultaneously as an $S^1$-module and an $S^1$-comodule.
One can easily verify that these structures are compatible in the sense of \cref{def:S1mod-compatible}.
We further introduce ``twisted'' versions of these structures.

\begin{definition}
\label{def:twist-S1}
Define morphisms
\[
\begin{gathered}
\widetilde\mu\coloneqq\mu-\mu\Delta\otimes\varepsilon
:\diagFiV\to\diagFiH
\quad,\\
\widetilde\Delta\coloneqq\Delta-\mu\Delta\otimes\eta
:\diagFiH\to\diagFiV
\quad.
\end{gathered}
\]
\end{definition}

\begin{lemma}
\label{lem:twist-S1modcomod}
The morphisms in \cref{def:twist-S1} define structures of an $S^1$-module and an $S^1$-comodule on $I$ which are compatible with each other.
\end{lemma}
\begin{proof}
We have
\[
\begin{split}
(\widetilde\mu\otimes S^1)\circ(I\otimes\Delta\eta)
&= (\mu\otimes S^1)\circ(I\otimes\Delta\eta) - (\mu\Delta\otimes\varepsilon\otimes S^1)\circ(I\otimes\Delta\eta) \\
&= \Delta - \mu\Delta\otimes\eta = \widetilde\Delta
\quad.
\end{split}
\]
Hence, it suffices to show that $\widetilde\mu$ is an $S^1$-module structure on $I$.
By the relation $(S)$, we have
\[
\widetilde\mu\circ(I\otimes\eta) = \mathrm{id}_I
\quad.
\]
On the other hand, by the relation $(4Tu)$, we also have
\[
\begin{split}
\widetilde\mu\circ(\widetilde\mu\otimes S^1)
&=
\begin{multlined}[t]
\mu\circ(\mu\otimes S^1) - \mu\circ(\mu\Delta\otimes\varepsilon\otimes S^1) - (\mu\Delta\otimes\varepsilon)\circ(\mu\otimes S^1)\\ + (\mu\Delta\otimes\varepsilon)\circ(\mu\Delta\otimes\varepsilon\otimes S^1)
\end{multlined}
\\
&= \mu\circ(I\otimes\mu) - (\mu\Delta\otimes\varepsilon)\circ(I\otimes\mu) \\
&= \widetilde\mu\circ(I\otimes\mu)
\quad.
\end{split}
\]
They respectively imply the unitality and the associativity of the action $\widetilde\mu$, so the result follows.
\end{proof}

\subsection{The crux complex}
\label{sec:fstder:crux}

To concentrate on links, among tangles, by \emph{singular link-like graphs}, we mean a singular tangle-like graph $G$ with no univalent vertices.
Throughout this subsection, we fix a singular link-like graph $G$ with a unique double point, say $c^\sharp(G)=\{b_0\}$.
We also write $c(G)\coloneqq V^4(G)\setminus c^\sharp(G)$.
The goal of the subsection is to construct a $c(G)$-fold complex $\operatorname{Crx}(G^\chi)$.

\begin{notation}
In the situation above, we regard a map $\alpha:c(G)\to\mathbb Z$ as a map out of $V^4(G)$ with $\alpha(b_0)=0$.
This enables us to consider the $\alpha$-smoothing $G_\alpha$ of $G$ provided $\alpha$ lies in the effective range in the sense of \cref{sec:UKH-mulcplx:cob}.
In particular, $G_\alpha$ resolves the double point $b_0$ into two smooth edges (see \cref{tab:quad-repl}).
\end{notation}

\begin{definition}
A map $\alpha:c(G)\to\mathbb Z$ is called a \emph{$G$-crux map} if it satisfies the following conditions:
\begin{enumerate}[label=\upshape(\roman*)]
  \item it lies in the effective range;
  \item the two edges in $G_\alpha$ involved in the double point resolution belong to the same connected component of the $1$-manifold $|G_\alpha|$.
\end{enumerate}
\end{definition}

For a checkerboard coloring $\chi$ on the complement of $G$, we define an object $|G_\alpha|^\chi_{\mathsf{crx}}\in\Cob(\varnothing,\varnothing)$ by
\[
|G_\alpha|^\chi_{\mathsf{crx}}\coloneqq\begin{cases*}
|G_\alpha|^\chi & if $\alpha$ is a $G$-crux map, \\
0 & otherwise.
\end{cases*}
\]
If $\alpha:c(G)\to\mathbb Z$ is a $G$-crux map, then we specify a subarc of the closed $1$-manifold $|G_\alpha|^\chi_{\mathsf{crx}}$ as depicted as the blue line in \cref{fig:twistarc}, which we call the \emph{twisted arc}.
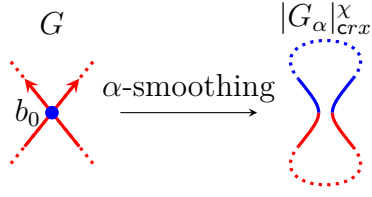
\begin{figure}[t]
\centering
\begin{tikzpicture}
\begin{scope}[shift={(-2,0)}]
\node[above] at (0,1) {$G$};
\draw[red,very thick,dotted] (-.6,-.75) -- (.6,.75);
\draw[red,very thick,dotted] (.6,-.75) -- (-.6,.75);
\draw[red,very thick,-stealth] (-.4,-.5) -- (.4,.5);
\draw[red,very thick,-stealth] (.4,-.5) -- (-.4,.5);
\fill[blue] (0,0) circle (.1) node[left,black]{$b_0$};
\end{scope}
\begin{scope}[shift={(2,0)}]
\node[above] at (0,1) {$|G_\alpha|_{\mathsf crx}^\chi$};
\draw[red,very thick] (.1,0) .. controls+(0,-.125)and+(-.2,.25) .. (.4,-.5);
\draw[red,very thick,dotted] (.4,-.5) .. controls+(.6,-.75)and+(-.6,-.75) .. (-.4,-.5);
\draw[red,very thick] (-.4,-.5) .. controls+(.2,.25)and+(0,-.125) .. (-.1,0);
\draw[blue,very thick] (.1,0) .. controls+(0,.125)and+(-.2,-.25) .. (.4,.5);
\draw[blue,very thick,dotted] (.4,.5) .. controls+(.6,.75)and+(-.6,.75) .. (-.4,.5);
\draw[blue,very thick] (-.4,.5) .. controls+(.2,-.25)and+(0,.125) .. (-.1,0);
\end{scope}
\draw[-stealth] (-1,0) -- node[above]{$\alpha$-smoothing} (1,0);
\end{tikzpicture}
\caption{The twisted arc in a smoothing (depicted in blue)}
\label{fig:twistarc}
\end{figure}
We also say an edge of $G_\alpha$ is \emph{twisted} if it lies in the twisted arc, and, in what follows, we depict them as blue arrows in pictures.
Hence, on a neighborhood of each crossing of $G$, $G_\alpha$ is depicted as in \cref{tab:twistsm}, where $G_{\alpha}$ and $G_{\alpha\pm v}$ are put in the same row provided both $\alpha$ and $\alpha\pm v$ are $G$-crux maps.
\begin{table}[tp]
\centering
\begin{tabular}{c||c|c}
  $G$ & $G_\alpha$ ($\alpha(v)=0$) & $G_\alpha$ ($\alpha(v)=\pm1$) \\\hline\rule[-3.5ex]{0pt}{8ex}
  & \diagSmoothUp & \diagSmoothW \\\cline{2-3}\rule[-3.5ex]{0pt}{8ex}
  \raisebox{1ex}{\smash{$\begin{matrix}\diagCrossNegUpWith{v}\\[2ex]\text{or}\\\diagCrossPosUpWith{v}\end{matrix}$}} & \diagSmoothUptwB & \diagSmoothWtwU{} or \diagSmoothWtwD \\\cline{2-3}\rule[-3.5ex]{0pt}{8ex}
  & \diagSmoothUptwL{} or \diagSmoothUptwR & \diagSmoothWtwB
\end{tabular}
\caption{Smoothings respecting twisted edges.}
\label{tab:twistsm}
\end{table}
Fix a checkerboard coloring $\chi$ on the complement of $G$, and let $v\in c(G)$ be a crossing of $G$ and $\alpha:c(G)\to\mathbb Z$ a $G$-crux map with $\alpha(v)=0$.
If $v$ is negative and $\alpha-v$ is a $G$-crux map, we define a morphism $\delta_v^{\mathsf{tw}}:|G_{\alpha-v}|_{\mathsf crx}^\chi\to|G_\alpha|_{\mathsf crx}^\chi\in\Cob(\varnothing,\varnothing)$ so that it is depicted as follows on a neighborhood of $v$:
\[
\delta_v^{\mathsf{tw}}\coloneqq
\begin{cases}
\delta_v:\:\left|\diagSmoothW\right|^\chi\to\left|\diagSmoothUp\right|^\chi\quad,\\[3.5ex]
\widetilde\mu:\:\left|\diagSmoothWtwUDA\right|^\chi\to\left|\diagSmoothUptwBDA\right|^\chi
\;\text{or}\quad
\left|\diagSmoothWtwDUA\right|^\chi\to\left|\diagSmoothUptwBUA\right|^\chi\quad,\\[5ex]
\widetilde\Delta:\:
\left|\diagSmoothWtwBLA\right|^\chi\to\left|\diagSmoothUptwRLA\right|^\chi
\;\text{or}\quad
\left|\diagSmoothWtwBRA\right|^\chi\to\left|\diagSmoothUptwLRA\right|^\chi\quad,
\end{cases}
\]
here we consider the twisted $S^1$-module structure on the twisted arc discussed in \cref{sec:fstder:twist}.
Similarly, if $v$ is a positive crossing with $\alpha+v$ being a $G$-crux map, then we define $\delta_v^{\mathsf{tw}}:|G_\alpha|_{\mathsf crx}^\chi\to|G_{\alpha+v}|_{\mathsf crx}^\chi$ by
\[
\delta_v^{\mathsf{tw}}\coloneqq
\begin{cases}
\delta_v:\:\left|\diagSmoothUp\right|^\chi\to\left|\diagSmoothW\right|^\chi\quad,\\[3.5ex]
\widetilde\mu:\:\left|\diagSmoothUptwRLA\right|^\chi\to\left|\diagSmoothWtwBLA\right|^\chi
\;\text{or}\;\:
\left|\diagSmoothUptwLRA\right|^\chi\to\left|\diagSmoothWtwBRA\right|^\chi\quad,\\[5ex]
\widetilde\Delta:\:
\left|\diagSmoothUptwBDA\right|^\chi\to\left|\diagSmoothWtwUDA\right|^\chi
\;\text{or}\quad
\left|\diagSmoothUptwBUA\right|^\chi\to\left|\diagSmoothWtwDUA\right|^\chi\quad.
\end{cases}
\]
For a general map $\alpha:c(G)\to\mathbb Z$, and for each crossing $v\in c(G)$, we define a morphism $d_v:|G_\alpha|^\chi\to |G_{\alpha+v}|^\chi\in\Cob(\varnothing,\varnothing)$ by $d_v\coloneqq\delta_v^{\mathsf{tw}}$ if it is defined and $d_v\coloneqq0$ otherwise.

\begin{lemma}
\label{lem:twdiff-comm}
Let $G$ and $\chi$ be as above, and suppose $\alpha:c(G)\to\mathbb Z$ is a map.
Then, for two distinct crossings $v,w\in c(G)$, the diagram in $\Cob(\varnothing,\varnothing)$ below commutes:
\begin{equation}
\label{eq:twdiff-comm:sq}
\begin{tikzcd}
{|G_\alpha|_{\mathsf{crx}}^\chi} \ar[r,"d_v"] \ar[d,"d_w"'] & {|G_{\alpha+v}|_{\mathsf{crx}}^\chi} \ar[d,"d_w"] \\
{|G_{\alpha+w}|_{\mathsf{crx}}^\chi} \ar[r,"d_w"] & {|G_{\alpha+v+w}|_{\mathsf{crx}}^\chi}
\end{tikzcd}
\quad.
\end{equation}
\end{lemma}
\begin{proof}
In the case where $|G_\alpha|_{\mathsf{crx}}^\chi$, $|G_{\alpha+v}|_{\mathsf{crx}}^\chi$, $|G_{\alpha+w}|_{\mathsf{crx}}^\chi$, and $|G_{\alpha+v+w}|_{\mathsf{crx}}^\chi$ are all non-zero, the result follows from \cref{lem:twist-S1modcomod} and \cref{lem:S1mod-unique}.
It remains to show that the diagrams below commute:
\[
\begin{tikzcd}
\diagTrefoilCruxlO \ar[r,"\Delta"] \ar[d] & \diagTrefoilCruxll \ar[d,"\widetilde\mu"] \\
0 \ar[r] &  \diagTrefoilCruxOl
\end{tikzcd}
\ ,\quad
\begin{tikzcd}
\diagTrefoilCruxOl \ar[r,"\widetilde\Delta"] \ar[d] & \diagTrefoilCruxll \ar[d,"\mu"] \\
0 \ar[r] & \diagTrefoilCruxlO
\end{tikzcd}
\ .
\]
They immediately follow from \cref{prop:ROnePos-absex} and \cref{prop:ROneNeg-absex}.
\end{proof}

\begin{definition}
Let $G$ be a singular link-like graph with a unique double point, and let $\chi$ be a checkerboard coloring of the complement of $G$.
We define a $c(G)$-fold complex $\operatorname{Crx}(G^\chi)^\bullet$ in $\Cob(\varnothing,\varnothing)$ by $\operatorname{Crx}(G^\chi)^\alpha\coloneqq|G_\alpha|^\chi_{\mathsf{crx}}$ with the differential $d_v$ defined above.
We call it the \emph{crux complex} of $G$.
\end{definition}

\subsection{Long exact sequence}
\label{sec:fstder:longex}

Let $G$ be a singular link-like graph with a unique double point.
We compute the Khovanov homology of $G$ in terms of the crux complex of $G$.
For this purpose, we need to know the relation of crux complexes and Khovanov complexes.
We denote by $\Hsmooth{G}$ and $\Vsmooth{G}$ a link-like graph obtained from $G$ by replacing the double point with a wide edge and by resolving it respectively; i.e.
\[
G=\diagSingUp
\ ,\quad \Hsmooth{G}=\diagSmoothW
\ ,\quad \Vsmooth{G}=\diagSmoothUp
\quad.
\]
Under the identification
\[
\mathbf{MCh}_{V^4(G)}^{\mathsf b}\bigl(\Cob(\varnothing,\varnothing)\bigr)
\cong \mathbf{Ch}^{\mathsf b}\left(\mathbf{MCh}_{c(G)}(\Cob(\varnothing,\varnothing))\right)
\]
mentioned in \cref{ex:mulfold-cplx-disj}, the $V^4(G)$-fold complex $\operatorname{Sm}(G^\chi)^\bullet$ defined in \cref{sec:UKH-mulcplx:UKH-mcplx} can be regarded as the following chain complex in the additive category $\mathbf{MCh}_{c(G)}(\Cob(\varnothing,\varnothing))$:
\begin{equation}
\label{eq:sing-chain}
\operatorname{Sm}(\Hsmooth{G}^\chi)^\bullet
\xrightarrow{-\delta_-}\operatorname{Sm}(\Vsmooth{G}^\chi)^\bullet
\xrightarrow{\Phi}\operatorname{Sm}(\Vsmooth{G}^\chi)^\bullet
\xrightarrow{-\delta_+}\operatorname{Sm}(\Hsmooth{G}^\chi)^\bullet
\quad.
\end{equation}
On the other hand, we define morphisms $\iota:\operatorname{Crx}(G^\chi)^\bullet\to\operatorname{Sm}(\Hsmooth{G}^\chi)^\bullet$ and $\pi:\operatorname{Sm}(\Hsmooth{G}^\chi)^\bullet\to\operatorname{Crux}(G^\chi)^\bullet$ as follows: notice first that, if $\alpha:c(G)\to\mathbb Z$ is a $G$-crux map, with regard to the connected components involved with the unique double point of $G$, we can depict $|G_\alpha|^\chi_{\mathsf{crx}}$ and $|\Hsmooth[\alpha]{G}|^\chi$ as
\[
|G_\alpha|^\chi_{\mathsf{crx}}\:=\: \diagCruxCompCrx
\ ,\quad
|\Hsmooth[\alpha]{G}|^\chi\:=\: \diagCruxCompH
\quad.
\]
In this case, we set $\iota$ and $\pi$ to be the morphisms
\begin{equation}
\label{eq:crux-iotapi}
\iota\coloneqq\widetilde\Delta:\diagCruxCompCrx\to\diagCruxCompH
\ ,\quad\pi\coloneqq\widetilde\mu:\diagCruxCompH\to\diagCruxCompCrx
\quad,
\end{equation}
where $\widetilde\Delta$ and $\widetilde\mu$ are the twisted coaction and action of the top circle on the bottom.
For the other maps $\alpha:c(G)\to\mathbb Z$, we set $\iota$ and $\pi$ to be zero.

\begin{lemma}
\label{lem:iotapi-chmorph}
In the situation above, $\iota$ and $\pi$ define morphisms of $c(G)$-fold complexes.
\end{lemma}
\begin{proof}
We define two singular link-like graphs $\Nresol{H}$ and $\Presol{H}$ by replacing the double point of $G$ with the following pictures:
\[
G=\diagSingUp
\ ,\quad \Nresol{H}={\ooalign{\ensuremath{\diagRvSingN}\crcr\hss\raisebox{1.5ex}{\ensuremath{\mathrlap{\;\;w_-}}}\hss}}
\ ,\quad \Presol{H}={\ooalign{\ensuremath{\diagRvSingP}\crcr\hss\raisebox{1.5ex}{\ensuremath{\mathrlap{\;\;w_+}}}\hss}}
\quad.
\]
Hence, we may regard $V^4(\PNresol{H})=V^4(G)\cup\{w_{\pm}\}$ and $c(\PNresol{H})=c(G)\cup\{w_{\pm}\}$.
For a map $\alpha:c(G)\to\mathbb Z$, we think of it a map $\alpha:c(H^\pm)\to\mathbb Z$ with $\alpha(w_\pm)=0$.
In this case, we have canonical identifications
\[
\begin{gathered}
\operatorname{Crx}(\Nresol{H}^\chi)^\alpha
= \operatorname{Crx}(\Presol{H}^\chi)^\alpha
= \operatorname{Crx}(G^\chi)^\alpha
\quad,\\
\operatorname{Crx}(\Nresol{H}^\chi)^{\alpha-w_-}
= \operatorname{Crx}(\Presol{H}^\chi)^{\alpha+w_+}
= \operatorname{Sm}(\Hsmooth{G}^\chi)^\alpha
\quad.
\end{gathered}
\]
Under these identifications, the morphisms $\iota$ and $\pi$ are identified with the differentials $d_{w_-}$ and $d_{w_+}$ on $\operatorname{Crx}(\Nresol{H}^\chi)^\bullet$ and $\operatorname{Crx}(\Presol{H}^\chi)^\bullet$ respectively.
Thus, the result follows from \cref{lem:twdiff-comm}.
\end{proof}

\begin{proposition}
\label{prop:crux-sm-exact}
Let $G$, $\Hsmooth{G}$, and $\Vsmooth{G}$ be as above, and let $\chi$ be a checkerboard coloring on the complement of $G$.
Then, the following sequence is absolutely exact in $\Cob(\varnothing,\varnothing)$ for each $\alpha:c(G)\to\mathbb Z$:
\begin{equation}
\label{eq:crux-sm-exact:seq}
\begin{tikzcd}[column sep=2em,row sep=3ex]
0 \ar[r] & \operatorname{Crx}(G^\chi)^\alpha \ar[r,"\iota"] & \operatorname{Sm}(\Hsmooth{G}^\chi)^\alpha \ar[r,"-\delta_-"] \ar[d,phantom,""{coordinate,name=C}] & \operatorname{Sm}(\Vsmooth{G}^\chi)^\alpha \ar[dll,"\Phi",rounded corners=1.5ex,to path={%
  -- ([xshift=1.5em]\tikztostart.east) [pos=.5]\tikztonodes
  |- (C)
  -| ([xshift=-1.5em]\tikztotarget.west)
  -- (\tikztotarget)}] & \\
& \operatorname{Sm}(\Vsmooth{G}^\chi)^\alpha \ar[r,"-\delta_+"] & \operatorname{Sm}(\Hsmooth{G}^\chi)^\alpha \ar[r,"\pi"] & \operatorname{Crx}(G^\chi)^\alpha \ar[r] & 0
\end{tikzcd}
\quad.
\end{equation}
\end{proposition}
\begin{proof}
Unwinding the definition of the morphisms in \eqref{eq:crux-sm-exact:seq}, we are reduced to prove the following sequences are both absolutely exact:
\begin{gather}
\label{eq:prf:crux-sm-exact:noncrux}
0
\to \:\diagNonCruxH\:
\xrightarrow{-\delta_-} \:\diagNonCruxV\:
\xrightarrow{\Phi} \:\diagNonCruxV\:
\xrightarrow{-\delta_+} \:\diagNonCruxH\:
\to 0
\quad,\displaybreak[1]\\[2ex]
\label{eq:prf:crux-sm-exact:crux}
\begin{tikzcd}[ampersand replacement=\&,row sep=2.5ex]
0 \ar[r]
\& \diagCruxCompCrx \ar[r,"\widetilde\Delta"]
\& \diagCruxCompH \ar[r,"-\delta_-"] \ar[d,phantom,""{coordinate,name=C}]
\& \diagCruxCompV \ar[dll,"\Phi",rounded corners=1em,to path={%
  -- ([xshift=1.5em]\tikztostart.east) [pos=.5]\tikztonodes
  |- (C)
  -| ([xshift=-1.5em]\tikztotarget.west)
  -- (\tikztotarget)}] \& \\
\& \diagCruxCompV \ar[r,"-\delta_+"] \& \diagCruxCompH \ar[r,"\widetilde\mu"] \& \diagCruxCompCrx \ar[r] \& 0
\end{tikzcd}
\quad.
\end{gather}
Notice that, in the sequence~\eqref{eq:prf:crux-sm-exact:crux}, the morphism $\Phi$ vanishes.
Therefore, the result follows from \cref{prop:ROnePos-absex}, \cref{prop:ROneNeg-absex}, and \cref{prop:FI-absex}.
\end{proof}

It is finally revealed that the Khovanov homology of $G$ is computed in terms of the crux complex.
We denote by $\dblBrac{G^\chi}_{\mathsf{crx}}$ the total complex of the crux complex $\operatorname{Crx}(G^\chi)^\bullet$.

\begin{theorem}
\label{theo:Kh-crux}
Let $G$ be a singular link-like graph with a unique double point.
Then, there is a morphism of chain complexes
\[
\Xi:\dblBrac{G^\chi}_{\mathsf{crx}}[2]\to\dblBrac{G^\chi}_{\mathsf{crx}}[-2]
\]
in the $k$-linear category $\Cob(\varnothing,\varnothing)$ together with a chain homotopy equivalence
\[
\dblBrac{G^\chi}
\simeq\operatorname{Cone}(\Xi)
\quad.
\]
Furthermore, in each cohomological degree $i\in\mathbb Z$, the morphism $\Xi^i:\dblBrac{G^\chi}_{\mathsf{crx}}^{i-2}\to\dblBrac{G^\chi}_{\mathsf{crx}}^{i+2}$ is homogeneous in Euler grading $2$.
\end{theorem}
\begin{proof}
Taking the total complexes in each term of the sequence~\eqref{eq:crux-sm-exact:seq}, we obtain the $2$-fold complex
\begin{equation}
\label{eq:prf:Kh-crux:hexseq}
\begin{tikzcd}[column sep=2em]
\cdots \ar[r] & 0 \ar[r]
& \dblBrac{G^\chi}_{\mathsf{crx}} \ar[r,"\iota"]
& \dblBrac{\Hsmooth{G}^\chi} \ar[r,"-\delta_-"] \ar[d,phantom,""{coordinate,name=C}]
& \dblBrac{\Vsmooth{G}^\chi} \ar[dll,"\Phi",rounded corners=.75em,to path={%
  -- ([xshift=1.5em]\tikztostart.east) [pos=.5]\tikztonodes
  |- (C)
  -| ([xshift=-1.5em]\tikztotarget.west)
  -- (\tikztotarget)}] && \\
&& \dblBrac{\Vsmooth{G}^\chi} \ar[r,"-\delta_+"]
& \dblBrac{\Hsmooth{G}^\chi} \ar[r,"\pi"]
& \dblBrac{G^\chi}_{\mathsf{crx}} \ar[r]
& 0 \ar[r] & \cdots
\end{tikzcd}
\quad,
\end{equation}
which we denote by $X^\bullet=\{X^{i,j}\}$ with the horizontal degrees so that $X^{-3,\bullet}=X^{2,\bullet}=\dblBrac{G^\chi}_{\mathsf{crx}}$.
We hence obtain a canonical isomorphisms
\[
\sigma^{\le -3}_{\mathsf H}X^\bullet
\cong \dblBrac{G^\chi}_{\mathsf{crx}}[-3]
\ ,\quad
\sigma^{\ge -2}_{\mathsf H}\sigma^{\le 1}_{\mathsf H}X^\bullet
\cong \dblBrac{G^\chi}
\ ,\quad
\sigma^{\ge 2}_{\mathsf H}X^\bullet
\cong \dblBrac{G^\chi}_{\mathsf{crx}}[2]
\quad.
\]
On the other hand, by virtue of \cref{prop:crux-sm-exact}, the sequence~\eqref{eq:prf:Kh-crux:hexseq} is absolutely exact in each vertical degree.
Therefore, the first statement directly follows from \cref{prop:dblcplx-longex}.

To verify the last statement, note that a chain contraction on the sequence~\cref{eq:prf:Kh-crux:hexseq} is explicitly given in \cref{prop:ROnePos-absex}, \cref{prop:ROneNeg-absex}, \cref{prop:FI-absex}.
Thus, one can describe $\Xi$ explicitly as in the proof of \cref{prop:dblcplx-longex}.
The last statement then follows from the direct computation.
\end{proof}

Consequently, \cref{theo:Kh-crux} yields a long exact sequence on variants of Khovanov homology.

\begin{corollary}
\label{cor:crux-homology-longex}
Let $G$ and $\chi$ be as in \cref{theo:Kh-crux}.
For every $k$-linear functor $Z:\Cob(\varnothing,\varnothing)\to\mathbf{Mod}_k$, there is a long exact sequence
\begin{equation}%
\label{eq:crux-homology-longex:seq}
\cdots
\to H^{i-2}Z_{h,t}\dblBrac{G^\chi}_{\mathsf{crx}}
\xrightarrow{\Xi_\ast} H^{i+2}Z_{h,t}\dblBrac{G^\chi}_{\mathsf{crx}}
\to H^iZ_{h,t}\dblBrac{G^\chi}
\to H^{i-1}Z_{h,t}\dblBrac{G^\chi}_{\mathsf{crx}}
\xrightarrow{\Xi_\ast}\cdots
\quad.
\end{equation}
\end{corollary}

\Cref{cor:crux-homology-longex} is mainly applied in the case $Z=Z_{h,t}:\Cob(\varnothing,\varnothing)\to\mathbf{Mod}_k$ in \cref{def:TQFT-ht}.
Specifically, in Euler-graded case in the sense of \cref{rem:Euler-gr}, we have a further refinement.
Indeed, in this case, we endow the chain complex $Z_{h,t}\dblBrac{G^\chi}_{\mathsf{crx}}$ with a grading in the same manner as in \cref{rem:Kh-grading}.

\begin{corollary}
\label{cor:crux-longex-gr}
Let $G$ and $\chi$ be as in \cref{theo:Kh-crux}.
If the functor $Z_{h,t}$ defined in \cref{def:TQFT-ht} is Euler-graded, then for each $j\in\mathbb Z$, there is a long exact sequence as below:
\begin{equation}
\label{eq:crux-longex-gr:seq}
\begin{tikzcd}[column sep=1em]
\cdots \ar[r]
& H^{i-2}Z_{h,t}\left(\dblBrac{G^\chi}_{\mathsf{crx}}\right)^{\bullet,j-2} \ar[r,"\Xi_\ast"]
& H^{i+2}Z_{h,t}\left(\dblBrac{G^\chi}_{\mathsf{crx}}\right)^{\bullet,j+4} \ar[r] \ar[d,phantom,""{coordinate,name=C}]
& H^iZ_{h,t}\left(\dblBrac{G^\chi}\right)^{\bullet,j} \ar[dl,rounded corners=.75em,to path={%
  -- ([xshift=1.5em]\tikztostart.east)
  |- (C)
  -| ([xshift=-1.5em]\tikztotarget.west)
  -- (\tikztotarget)}] \\
&& H^{i-1}Z_{h,t}\left(\dblBrac{G^\chi}_{\mathsf{crx}}\right)^{\bullet,j-2} \ar[r,"\Xi_\ast"]
& \cdots
\end{tikzcd}
\quad.
\end{equation}
\end{corollary}

In the case $h=t=0$, the sequence~\eqref{eq:crux-longex-gr:seq} in particular induces a long exact sequence containing Khovanov homology of $G$ (see~\cref{rem:Kh-grading}).
Hence, taking Euler characteristics, we obtain an identity on Jones polynomials.
In the next result, we denote by $V(L)$ the Jones polynomial for a link $L$.

\begin{corollary}
\label{cor:crux-Jones}
Suppose $D$ is a singular link diagram with a unique double point, and set $L_+$ and $L_-$ to be the links associated with the diagrams obtained from $D$ by resolving the double point into positive and negative crossings.
Then, for every field $k$, the following identity holds:
\begin{equation}\label{eq:crux-Jones:skein}
\left.(V(L_+)-V(L_-))\right|_{\sqrt{t}=-q}
= \frac{q^2-q^{-4}}{q+q^{-1}}\sum_{i,j}(-1)^i q^j \dim_k H^iZ_{0,0}\left(\dblBrac{D^\chi}_{\mathsf{crx}}\right)^{\bullet,j}
\end{equation}
\end{corollary}
\begin{proof}
Taking the Euler characteristics in \eqref{eq:crux-longex-gr:seq}, we obtain
\begin{equation}
\label{eq:prf:crux-Jones:crux}
\sum_{i,j}(-1)^iq^j\dim_k \mathit{Kh}^{i,j}(D;k)
= (q^2-q^{-4})\dim_k H^iZ_{0,0}\left(\dblBrac{D^\chi}_{\mathsf{crx}}\right)^{\bullet,j}
\quad.
\end{equation}
On the other hand, by \cref{prop:skein-mcone}, we also have
\begin{multline}
\label{eq:prf:crux-Jones:Vassilliev}
\sum_{i,j}(-1)^iq^j\dim_k \mathit{Kh}^{i,j}(D;k) \\
= \sum_{i,j}(-1)^iq^j\dim_k\mathit{Kh}^{i,j}(L_+;k)
- \sum_{i,j}(-1)^iq^j\dim_k\mathit{Kh}^{i,j}(L_-;k)
\quad.
\end{multline}
Since the graded Euler characteristic of Khovanov homology is exactly the unnormalized Jones polynomial \cite{Khovanov2000}, we hence obtain the result by combining \eqref{eq:prf:crux-Jones:crux} and \eqref{eq:prf:crux-Jones:Vassilliev}.
\end{proof}

\begin{remark}
It is known that, if $\zeta_3\coloneqq e^{2\pi\sqrt{-1}/3}$ is the primitive cubic root of the unity, for every link $L$, we have $V_L(\zeta_3)=(-1)^{\#\pi_0(L)-1}$ (e.g.~see\cite[Table~16.3]{Lickorish1997}).
Actually, this identity also follows from \cref{cor:crux-Jones}.
Indeed, the equation~\eqref{eq:crux-Jones:skein} implies that any crossing change is trivial on Jones polynomial modulo the ideal generated by $1-t^3$.
Specifically, $V_L(t)$ equals the Jones polynomial of the trivial link with the same number of components as $L$, which is $(-1)^{\#\pi_0(L)-1}$.
Therefore, the identity $V_L(\zeta_3)=(-1)^{\#\pi_0(L)-1}$ follows.
\end{remark}

\section{Applications}
\label{sec:app}

In this last section, we make use of the crux complex to compute the Khovanov homology of several links with single double points.
As they are homotopy cofibers of the crossing-change operation $\widehat\Phi$ in \cref{prop:skein-mcone}, the results can be used to investigate $\widehat\Phi$ itself.

\subsection{Khovanov homology of twist knots}
\label{sec:app:twist}

First, we give an example of Khovanov homology for which the crux complex drastically reduces the computation.
Namely, for a non-negative integer $r$, let $G(r)$ be the singular knot diagram depicted in \cref{fig:twist-dbl}.
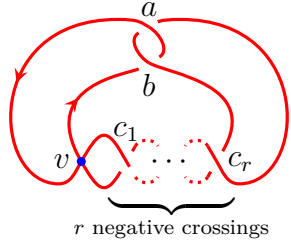
\begin{figure}[t]
\centering
\begin{tikzpicture}[scale=.66]
\useasboundingbox (-3,-3) rectangle (3,2);
\node[inner sep=3,circle] (CU) at (0,1.5) {};
\node[inner sep=3,circle] (CM) at (0,.7) {};
\coordinate (D) at (-1.5,-1.5);
\node[inner sep=3,circle] (CL) at (-.5,-1.5) {};
\node[inner sep=3,circle] (CR) at (1.5,-1.5) {};
\draw[red,very thick] (CU) .. controls+(30:.2)and+(-.5,0) .. (0.7,1.67) .. controls+(3,0)and+(1.5,0) .. (2,-2) to[out=180,in=-60] (CR.center) -- +(120:.4);
\draw[red,very thick,-stealth-] (CU.center) .. controls+(150:.2)and+(.5,0) .. (-0.7,1.67) .. controls+(-3,0)and+(-1.5,0) .. (-2,-2) to[out=0,in=-120] (D);
\draw[red,very thick] (CM.center) .. controls+(-30:.5)and+(60:2) .. (CR) -- +(-120:.4);
\draw[red,very thick,-stealth-] (D) .. controls+(120:2)and+(210:.5) .. (CM);
\draw[red,very thick] (CU.center) to[out=-30,in=30] (CM);
\draw[red,very thick] (CM.center) to[out=150,in=210] (CU);
\draw[red,very thick] (D) to[out=60,in=120,looseness=2] (CL) -- +(-60:.4);
\draw[red,very thick] (D) to[out=-60,in=240,looseness=2] (CL);
\draw[red,very thick] (CL) -- +(60:.4);
\draw[red,very thick,dotted] (CL) ++(60:.4) to[out=60,in=120] +(.5,0);
\draw[red,very thick,dotted] (CL) ++(-60:.4) to[out=-60,in=240] +(.5,0);
\draw[red,very thick,dotted] (CR) ++(120:.4) to[out=120,in=60] +(-.5,0);
\draw[red,very thick,dotted] (CR) ++(240:.4) to[out=240,in=-60] +(-.5,0);
\node at (.5,-1.5) {$\cdots$};
\node[below] at (.5,-2) {$\underbrace{\rule{4em}{0pt}}_{\text{$r$ negative crossings}}$};
\fill[blue] (D) circle(.1);
\node[left] at (D) {$v$};
\node[above=.7ex] at (CL) {$c_1$};
\node[right] at (CR) {$c_r$};
\node[above] at (CU) {$a$};
\node[below] at (CM) {$b$};
\end{tikzpicture}
\caption{A twist knot with a single double point}
\label{fig:twist-dbl}
\end{figure}%
We write
\[
c(G(r))=\{a,b,c_1,\dots,c_r\}
\]
as in the figure; hence, the crossings $a$ and $b$ are both positive (if $r$ is even) or negative (if $r$ is odd) while $c_1,\dots,c_r$ are always all negative crossings.

\begin{notation}
Throughout this subsection, we fix a checkerboard coloring for $G(r)$ and drop it from the notations.
\end{notation}

In the situation above, one can easily verify the following.

\begin{lemma}
\label{lem:twist-crux}
For the diagram $G(r)$ given above, a map $\alpha:c(G(r))\to\mathbb Z$ is a $G(r)$-crux map precisely if $\alpha(c_1)=\dots=\alpha(c_r)=-1$ and either of the following hold:
\begin{itemize}
  \item $r$ is even and
\[
(\alpha(a),\alpha(b))
= (-1,-1),\,(-1,0),\,(0,-1)
\quad;
\]
\item $r$ is odd and
\[
(\alpha(a),\alpha(b))
= (0,0),\,(0,1),\,(1,0)
\quad.
\]
\end{itemize}
\end{lemma}

It follows that the complex $\dblBrac{G(r)}_{\mathsf{crx}}$ is isomorphic to the total complex of the following double complex in $\Cob(\varnothing,\varnothing)$ (up to degree shifts):
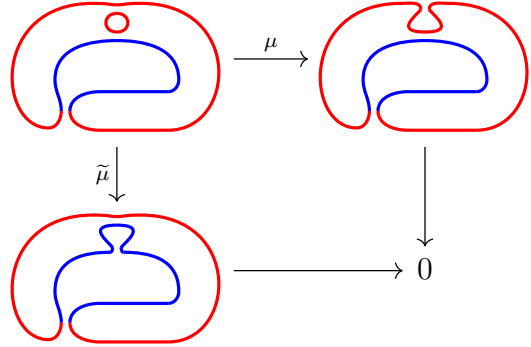
\[
\begin{tikzcd}
\diagTwistCruxOO \ar[r,"\mu"] \ar[d,"\widetilde\mu"'] & \diagTwistCruxlO \ar[d] \\
\diagTwistCruxOl \ar[r] & 0
\end{tikzcd}
\quad.
\]
Note that the morphism $\widetilde\mu:S^1\otimes S^1\to S^1$ is a split epimorphism in $\Cob(\varnothing,\varnothing)$ with kernel $\Delta:S^1\to S^1\otimes S^1$, one obtains the following result.

\begin{lemma}%
\label{lem:crux-twist}
In the situation above, the complex $\dblBrac{G(r)}_{\mathsf{crx}}$ is chain-homotopy equivalent to the complex
\begin{equation}
\label{eq:twist-cruxcore}
\cdots\to 0\to S^1\xrightarrow{\mu\Delta} S^1 \to 0 \to\cdots
\end{equation}
in $\Cob(\varnothing,\varnothing)$ concentrated in the degrees $-r-(-1)^r-1$ and $-r-(-1)^r$.
\end{lemma}

\begin{proposition}
\label{prop:kh-twistsing}
The complex $\dblBrac{G(r)}$ is chain homotopic to the following complex in $\Cob(\varnothing,\varnothing)$:
\begin{equation}
\label{eq:kh-twistsing:cplx}
\begin{tikzcd}[column sep=2em,row sep=.5ex]
& \mathclap{\scriptstyle-r-(-1)^r-3}\quad & \quad\mathclap{\scriptstyle-r-(-1)^r-2} & & \mathclap{\scriptstyle-r-(-1)^r}\quad & \quad\mathclap{\scriptstyle-r-(-1)^r+1} & \\
\cdots \to 0 \ar[r] & S^1\ar[r,"\mu\Delta"] & S^1 \ar[r] & 0 \ar[r] & S^1 \ar[r,"\mu\Delta"] & S^1 \ar[r] & 0 \to \cdots
\end{tikzcd}
\quad,
\end{equation}
where the labels above indicate the cohomological degrees.
\end{proposition}
\begin{proof}
By \cref{theo:Kh-crux}, $\dblBrac{G(r)}$ is chain homotopic to the mapping cone of the morphism $\Xi:\dblBrac{G(r)}_{\mathsf{crx}}[2]\to\dblBrac{G(r)}_{\mathsf{crx}}[-2]$, which is null-homotopic for degree reason.
This implies that there is a chain homotopy equivalence
\begin{equation}
\label{eq:prf:kh-twistsing:split}
\dblBrac{G(r)}
\simeq \dblBrac{G(r)}_{\mathsf{crx}}[-2]\oplus\dblBrac{G(r)}_{\mathsf{crx}}[1]
\quad.
\end{equation}
The result hence follows from \cref{lem:crux-twist}.
\end{proof}

\begin{example}
Let $C_{h,t}$ be the Frobenius algebra over $k$ defined in \cref{def:FrobCht}.
Applying the associated TQFT $Z_{h,t}:\Cob(\varnothing,\varnothing)\to\mathbf{Mod}_k$ to the complex~\eqref{eq:kh-twistsing:cplx}, we obtain
\[
H^iZ_{h,t}\dblBrac{G(r)}\cong
\begin{cases}
\operatorname{Ann}_{C_{h,t}}(2x-h) & i=-r-(-1)^r-3,\,-r-(-1)^r,\\
C_{h,t}/(2x-h) & i=-r-(-1)^r-2,\,-r-(-1)^r+1,\\
0 & \text{otherwise},
\end{cases}
\]
where $\operatorname{Ann}_{C_{h,t}}(2x-h)$ is the \emph{annihilator} of the element $2x-h\in C_{h,t}$.
Note that $2x-h$ is invertible in $C_{h,t}$ if and only if $h^2+4t$ is invertible in the coefficient ring $k$, and, in this case, the homology $H^\ast Z\dblBrac{G(r)}$ vanishes.
Specifically, it vanishes for Lee homology \cite{Lee2005} with $\frac12\in k$ and Bar-Natan homology \cite{BarNatan2005}.
\end{example}

We use the result of \cref{prop:kh-twistsing} to determine the homology of twist knots.
Let us denote by $\Presol{G(r)}$ and $\Nresol{G(r)}$ the diagrams obtained from $G(r)$ by resolving the double point $v$ into positive and negative crossings respectively.
Specifically, $\Nresol{G(r)}$ is a diagram for the twist knot with $r+1$ half twists, so we also write $D(r+1)\coloneqq \Nresol{G(r)}$.
In addition, we write $D(0)$ the diagram in \cref{fig:unknot-ppself}.
\begin{figure}[t]
\centering
\begin{tikzpicture}[scale=.66]
\useasboundingbox (-3,-2.2) rectangle (3,1.7);
\node[inner sep=3,circle] (CU) at (0,1.5) {};
\node[inner sep=3,circle] (CM) at (0,.7) {};
\node[inner sep=5,circle] (D) at (-1.5,-1.5) {};
\node[inner sep=5,circle] (CR) at (1.5,-1.5) {};
\draw[red,very thick] (CU) .. controls+(30:.2)and+(-.5,0) .. (0.7,1.67) .. controls+(3,0)and+(2,0) .. (CR.south);
\draw[red,very thick,-stealth-] (CU.center) .. controls+(150:.2)and+(.5,0) .. (-0.7,1.67) .. controls+(-3,0)and+(-2,0) .. (D.south);
\draw[red,very thick] (CM.center) .. controls+(-30:.5)and+(90:1) .. (1.81,-0.58) to[out=-90,in=0] (CR.north);
\draw[red,very thick,-stealth-] (D.north) to[out=180,in=-90] (-1.81,-0.58) .. controls+(0,1)and+(210:.5) .. (CM);
\draw[red,very thick] (CU.center) to[out=-30,in=30] (CM);
\draw[red,very thick] (CM.center) to[out=150,in=210] (CU);
\draw[red,very thick] (D.south) -- (CR.south);
\draw[red,very thick] (D.north) -- (CR.north);
\end{tikzpicture}
\caption{The unknot with two-fold positive twists}
\label{fig:unknot-ppself}
\end{figure}
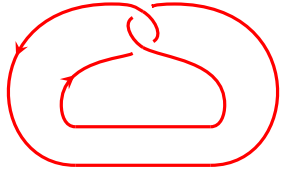
Note that the diagram $\Presol{G(r+1)}$ is equivalent to $D(r)$ via the Reidemeister move of type II.

\begin{proposition}
\label{prop:kh-twist}
For every non-negative integer $r$, there is a chain homotopy equivalence in the category $\Cob(\varnothing,\varnothing)$:
\begin{equation}%
\label{eq:kh-twist-isom}
\dblBrac{D(r)}
\simeq \begin{dcases*}
\dblBrac{\text{unknot}}\oplus\bigoplus_{i=1}^{r/2}\dblBrac{G(2i-1)}[1] & if $r$ is even,\\
\dblBrac{\text{trefoil}}\oplus\bigoplus_{i=1}^{(r-1)/2}\dblBrac{G(2i)}[1] & if $r$ is odd.
\end{dcases*}
\end{equation}
\end{proposition}

In order to prove \cref{prop:kh-twist}, we compute the sequence
\begin{equation}
\label{eq:Gr-triangle}
\dblBrac{G(r)}[1]
\twoheadrightarrow \dblBrac{\Nresol{G(r)}}
\xrightarrow{\widehat\Phi} \dblBrac{\Presol{G(r)}}
\hookrightarrow \dblBrac{G(r)}
\quad,
\end{equation}
which forms a distinguished triangle in the homotopy category of $\mathbf{Ch}^{\mathsf b}(\Cob(\varnothing,\varnothing))$.

\begin{lemma}
\label{lem:Gr-split}
For every non-negative integer $r\ge0$, the inclusion $\dblBrac{\Presol{G(r)}}\hookrightarrow\dblBrac{G(r)}$ is null-homotopic.
Consequently, the sequence~\eqref{eq:Gr-triangle} splits in the homotopy category.
\end{lemma}
\begin{proof}
By virtue of the homotopy equivalence~\eqref{eq:prf:kh-twistsing:split}, it suffices to show that the composition below is null homotopic:
\begin{equation}
\label{eq:prf:Gr-split:comp}
\dblBrac{\Presol{G(r)}}
\hookrightarrow\dblBrac{G(r)}
\simeq\dblBrac{G(r)}_{\mathsf{crx}}[-2]\oplus\dblBrac{G(r)}_{\mathsf{crx}}[1]
\quad.
\end{equation}
We denote by $\Hsmooth{G(r)}$ and $\Vsmooth{G(r)}$ the diagrams obtained from $G(r)$ by reducing and replacing the double point as in \cref{sec:fstder:longex}.
Hence, by \cref{prop:skein-mcone}, we have
\[
\Presol{G(r)}\cong \operatorname{Cone}\left(
\dblBrac{\Vsmooth{G(r)}}[1]
\xrightarrow{\widehat\delta_+[1]}
\dblBrac{\Hsmooth{G(r)}}[1]
\right)
\quad.
\]
Using this isomorphism, we represent the morphism~\eqref{eq:prf:Gr-split:comp} by the following matrix:
\[
\begin{bmatrix}
\widetilde\Theta_4 & \widetilde\Theta_3 \\ \pi[1] & 0
\end{bmatrix}
:\dblBrac{\Hsmooth{G(r)}}[1]\oplus\dblBrac{\Vsmooth{G(r)}}
\to \dblBrac{G(r)}_{\mathsf{crx}}[-2]\oplus\dblBrac{G(r)}_{\mathsf{crx}}[1]
\quad,
\]
where $\widetilde\Theta_i$ is the morphisms defined in \eqref{eq:Theta-tilde} and $\pi:\dblBrac{\Hsmooth{G(r)}}\to\dblBrac{G(r)}_{\mathsf{crx}}$ is the one given in \eqref{eq:crux-iotapi}.
Notice that, for degree reasons, both $\widetilde\Theta_4$ and $\widetilde\Theta_3$ are zero.
Therefore, it suffices to verify that $\pi$ is null-homotopic.

To describe the morphism $\pi$, we define a map $\alpha_0:c(G(r))\to\mathbb Z$ by
\[
\alpha_0(v)=
\begin{cases*}
0 & if $r$ is odd and $v=a,b$,\\
-1 & otherwise.
\end{cases*}
\quad.
\]
By \cref{lem:twist-crux}, a map $\alpha:c(G(r))\to\mathbb Z$ is a $G(r)$-crux map precisely if it is of the form $\alpha=\alpha_0+ia+jb+\sum_sk_s c_s$ for $i,j,k_s\in\{0,1\}$.
Specifically, we write
\[
E(i,j,k)\coloneqq E_{c(G(r))}(\alpha_0+ia+jb+kc_1)
\quad.
\]
Using this notation, we can describe $\pi$ componentwisely as follows:
\begin{align*}
\diagTwistCruxHOOO\otimes E(0,0,0)
& \xrightarrow{\widetilde\mu\otimes\mathrm{id}} \diagTwistCruxOO\otimes E(0,0,0)
\quad,\displaybreak[1]\\
\diagTwistCruxHlOO\otimes E(1,0,0)
& \xrightarrow{\widetilde\mu\otimes\mathrm{id}} \diagTwistCruxlO\otimes E(1,0,0)
\quad,\displaybreak[1]\\
\diagTwistCruxHOlO\otimes E(0,1,0)
& \xrightarrow{\widetilde\mu\otimes\mathrm{id}} \diagTwistCruxOl\otimes E(0,1,0)
\quad,\displaybreak[1]\\
\text{the others}
&\xrightarrow{\phantom{\widetilde\mu\otimes\mathrm{id}}} 0\quad.\hfill
\end{align*}
We consider the subcomplex $C$ of $\dblBrac{\Hsmooth{G(r)}}$ consisting of the components associated to the map $\alpha:c(G(r))\to\mathbb Z$ with
\[
\begin{gathered}
\alpha(a)\ge \alpha_0(a)+1
\,,\ \alpha(b)\ge\alpha_0(b)+1
\quad,\\
\alpha(c_1)\ge\alpha(c_1)
\,,\ \dots\,,\ 
\alpha(c_r)\ge\alpha(c_r)
\quad.
\end{gathered}
\]
Clearly, the morphism $\pi:\dblBrac{\Hsmooth{G(r)}}\to\dblBrac{G(r)}_{\mathsf{crx}}$ vanishes on the subcomplex $C$.
Since the quotient complex $Q\coloneqq\dblBrac{\Hsmooth{G(r)}}/C$ exists in the category $\Cob(\varnothing,\varnothing)$, it follows that the morphism $\pi$ factors through a morphism $\widetilde\pi:Q\to\dblBrac{G(r)}_{\mathsf{crx}}$.
Note that the complex $Q$ is isomorphic to a shift of the universal Khovanov complex of the graph depicted in \cref{fig:unknot-ppself}.
Thus, the two-fold Reidemeister moves of type I yield a chain homotopy equivalence between $Q$ and the following complex $\widetilde Q$:
\begin{equation}
\label{eq:prf:Gr-split:unknot}
\widetilde Q\coloneqq
\Bigl\{\:\cdots\to 0\to \overset{\mathclap{-r-(-1)^r-1}}{S^1\rule{0pt}{3ex}}\to 0\to\cdots\:\Bigr\}
\quad.
\end{equation}
By virtue of the description in \cite[Section~4.3]{BarNatan2005}, the chain homotopy equivalence $\widetilde Q\xrightarrow\simeq Q$ is given by the sum of cobordisms
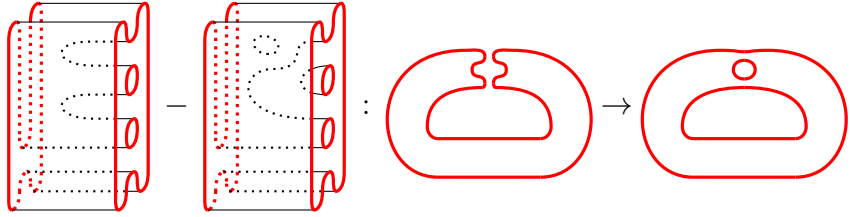
\[
\begin{tikzpicture}[xlen=.25pt,ylen=-.5pt,baseline=-.5ex]
\draw[red,very thick] (-155.41,50.47) to[quadratic={(-155.41,52.7)}] (-140.6,52.7) -- (99.37,52.7) to[quadratic={(119.4,52.7)}] (126.4,48.63) to[quadratic={(128.26,47.57)}] (128.26,46.78);
\draw[red,very thick,dotted] (128.26,46.78) to[quadratic={(128.26,44.57)}] (113.5,44.57) -- (-46.54,44.57) to[quadratic={(-66.54,44.57)}] (-59.5,40.5) to[quadratic={(-52.46,36.43)}] (-32.46,36.43) -- (127.5,36.43) to[quadratic={(127.9,36.43)}] (128.2,36.43);
\draw[red,very thick] (128.2,36.43) to[quadratic={(147.7,36.36)}] (154.6,32.37) to[quadratic={(156.41,31.31)}] (156.41,30.53);
\draw[red,very thick,dotted] (156.41,30.53) to[quadratic={(156.41,28.3)}] (141.6,28.3) -- (-97.66,28.3) -- (-98.37,28.3) to[quadratic={(-118.4,28.3)}] (-125.4,32.37) to[quadratic={(-132.5,36.43)}] (-112.5,36.43) to[quadratic={(-92.46,36.43)}] (-99.5,40.5) to[quadratic={(-106.5,44.57)}] (-126.5,44.57) to[quadratic={(-146.5,44.57)}] (-153.6,48.63) to[quadratic={(-155.41,49.69)}] (-155.41,50.47);
\draw[black] (-17.66,-41.34) -- (-17.66,-27.3);
\draw[black,thick,dotted]  (-17.66,-27.3) -- (-17.66,-11) .. controls +(0,30)and+(0,30) .. (18.66,-8) -- (18.66,-27.3);
\draw[black] (18.66,-27.3) -- (18.66,-37.66);
\draw[black] (62.34,-41.34) -- (62.34,-27.3);
\draw[black,thick,dotted] (62.34,-27.3) -- (62.34,-11) .. controls+(0,30)and+(0,30) .. (98.66,-8) -- (98.66,-27.3);
\draw[black] (98.66,-27.3) -- (98.66,-37.66);
\draw[black] (-155.4,50.48) -- (-155.4,-29.52);
\draw[black,thick,dotted] (-127.2,34.21) -- (-127.2,-35.43);
\draw[black] (-127.2,-35.43) -- (-127.2,-45.79);
\draw[black,thick,dotted] (-97.66,38.66) -- (-97.66,-27.3);
\draw[black] (-97.66,-27.3) -- (-97.66,-41.34);
\draw[black,thick,dotted] (-61.34,42.34) -- (-61.34,-27.3);
\draw[black] (-61.34,-27.3) -- (-61.34,-37.66);
\draw[black] (128.2,46.79) -- (128.2,6.792) -- (128.2,-33.21);
\draw[black] (156.4,30.52) -- (156.4,-9.478) -- (156.4,-49.48);
\draw[red,very thick] (-153.6,-31.37) to[quadratic={(-160.6,-27.3)}] (-140.6,-27.3) -- (99.37,-27.3) to[quadratic={(119.4,-27.3)}] (126.4,-31.37) to[quadratic={(133.5,-35.43)}] (113.5,-35.43) to[quadratic={(93.46,-35.43)}] (100.5,-39.5) to[quadratic={(107.5,-43.57)}] (127.5,-43.57) to[quadratic={(147.5,-43.57)}] (154.6,-47.63) to[quadratic={(161.6,-51.7)}] (141.6,-51.7) -- (-98.37,-51.7) to[quadratic={(-118.4,-51.7)}] (-125.4,-47.63) to[quadratic={(-132.5,-43.57)}] (-112.5,-43.57) to[quadratic={(-92.46,-43.57)}] (-99.5,-39.5) to[quadratic={(-106.5,-35.43)}] (-126.5,-35.43) to[quadratic={(-146.5,-35.43)}] (-153.6,-31.37);
\draw[red,very thick] (-59.5,-39.5) to[quadratic={(-66.54,-35.43)}] (-46.54,-35.43) to[quadratic={(-26.54,-35.43)}] (-19.5,-39.5) to[quadratic={(-12.46,-43.57)}] (-32.46,-43.57) to[quadratic={(-52.46,-43.57)}] (-59.5,-39.5);
\draw[red,very thick] (20.5,-39.5) to[quadratic={(13.46,-35.43)}] (33.46,-35.43) to[quadratic={(53.46,-35.43)}] (60.5,-39.5) to[quadratic={(67.54,-43.57)}] (47.54,-43.57) to[quadratic={(27.54,-43.57)}] (20.5,-39.5);
\end{tikzpicture}
\:-\:
\begin{tikzpicture}[xlen=.25pt,ylen=-.5pt,baseline=-.5ex]
\draw[red,very thick] (-155.41,50.47) to[quadratic={(-155.41,52.7)}] (-140.6,52.7) -- (99.37,52.7) to[quadratic={(119.4,52.7)}] (126.4,48.63) to[quadratic={(128.26,47.57)}] (128.26,46.78);
\draw[red,very thick,dotted] (128.26,46.78) to[quadratic={(128.26,44.57)}] (113.5,44.57) -- (-46.54,44.57) to[quadratic={(-66.54,44.57)}] (-59.5,40.5) to[quadratic={(-52.46,36.43)}] (-32.46,36.43) -- (127.5,36.43) to[quadratic={(127.9,36.43)}] (128.2,36.43);
\draw[red,very thick] (128.2,36.43) to[quadratic={(147.7,36.36)}] (154.6,32.37) to[quadratic={(156.41,31.31)}] (156.41,30.53);
\draw[red,very thick,dotted] (156.41,30.53) to[quadratic={(156.41,28.3)}] (141.6,28.3) -- (-97.66,28.3) -- (-98.37,28.3) to[quadratic={(-118.4,28.3)}] (-125.4,32.37) to[quadratic={(-132.5,36.43)}] (-112.5,36.43) to[quadratic={(-92.46,36.43)}] (-99.5,40.5) to[quadratic={(-106.5,44.57)}] (-126.5,44.57) to[quadratic={(-146.5,44.57)}] (-153.6,48.63) to[quadratic={(-155.41,49.69)}] (-155.41,50.47);
\draw[black] (-17.66,-41.34) -- (-17.66,-27.3);
\draw[black,thick,dotted] (-17.66,-27.3) .. controls+(0,60)and+(0,25) .. ++(75,29.69) .. controls+(0,-30)and+(0,20) .. (98.66,-27.3);
\draw[black] (98.66,-27.3) -- (98.66,-37.66);
\draw[black] (-155.4,50.48) -- (-155.4,-29.52);
\draw[black,thick,dotted] (-127.2,34.21) -- (-127.2,-35.43);
\draw[black] (-127.2,-35.43) -- (-127.2,-45.79);
\draw[black,thick,dotted] (-97.66,38.66) -- (-97.66,-27.3);
\draw[black] (-97.66,-27.3) -- (-97.66,-41.34);
\draw[black,thick,dotted] (-61.34,42.34) -- (-61.34,-27.3);
\draw[black] (-61.34,-27.3) -- (-61.34,-37.66);
\draw[black] (128.2,46.79) -- (128.2,26.79) -- (128.2,6.792) -- (128.2,-33.21);
\draw[black] (156.4,30.52) -- (156.4,10.52) -- (156.4,-9.478) -- (156.4,-49.48);
\draw[black] (18.66,-37.66) .. controls (18.66,-33.46)and(19.75,-30.02) .. (21.57,-27.3);
\draw[black,thick,dotted] (21.57,-27.3) .. controls (28.81,-16.48)and(47.76,-17.08) .. (57.05,-27.3);
\draw[black] (57.05,-27.3) .. controls (60.28,-30.85)and(62.34,-35.55) .. (62.34,-41.34);
\draw[black,thick,dotted] (80,5) .. controls+(0,12)and+(0,12) .. ++(27,3.68) .. controls+(0,-12)and+(0,-12) .. cycle;
\draw[red,very thick] (-153.6,-31.37) to[quadratic={(-160.6,-27.3)}] (-140.6,-27.3) -- (99.37,-27.3) to[quadratic={(119.4,-27.3)}] (126.4,-31.37) to[quadratic={(133.5,-35.43)}] (113.5,-35.43) to[quadratic={(93.46,-35.43)}] (100.5,-39.5) to[quadratic={(107.5,-43.57)}] (127.5,-43.57) to[quadratic={(147.5,-43.57)}] (154.6,-47.63) to[quadratic={(161.6,-51.7)}] (141.6,-51.7) -- (-98.37,-51.7) to[quadratic={(-118.4,-51.7)}] (-125.4,-47.63) to[quadratic={(-132.5,-43.57)}] (-112.5,-43.57) to[quadratic={(-92.46,-43.57)}] (-99.5,-39.5) to[quadratic={(-106.5,-35.43)}] (-126.5,-35.43) to[quadratic={(-146.5,-35.43)}] (-153.6,-31.37);
\draw[red,very thick] (-59.5,-39.5) to[quadratic={(-66.54,-35.43)}] (-46.54,-35.43) to[quadratic={(-26.54,-35.43)}] (-19.5,-39.5) to[quadratic={(-12.46,-43.57)}] (-32.46,-43.57) to[quadratic={(-52.46,-43.57)}] (-59.5,-39.5);
\draw[red,very thick] (20.5,-39.5) to[quadratic={(13.46,-35.43)}] (33.46,-35.43) to[quadratic={(53.46,-35.43)}] (60.5,-39.5) to[quadratic={(67.54,-43.57)}] (47.54,-43.57) to[quadratic={(27.54,-43.57)}] (20.5,-39.5);
\end{tikzpicture}
\::\:\diagTwistCruxHllO\to\diagTwistCruxHOOO
\]
in degree $-r-(-1)^r-1$ and zero in the other degrees.
It then turns out that the composition $\widetilde Q\xrightarrow\simeq Q\xrightarrow{\widetilde\pi}\dblBrac{G(r)}_{\mathsf{crx}}$ factors through $\widetilde\mu\Delta:S^1\to S^1$, which is zero by \cref{prop:ROneNeg-absex}.
This implies that $\widetilde\pi$ is null-homotopic.
Consequently, $\pi$ is also null-homotopic, and this completes the proof.
\end{proof}

\begin{proof}[Proof of \cref{prop:kh-twist}]
We prove the result by induction on $r$.
Since $D(0)$ and $D(1)$ are diagrams for the unknot and the trefoil respectively, the chain homotopy equivalence~\eqref{eq:kh-twist-isom} is obvious in the cases $r=0$ and $r=1$.
On the other hand, if $r\ge 2$, by virtue of \cref{lem:Gr-split}, the sequence~\eqref{eq:Gr-triangle} homotopically splits; in other words, there is a chain homotopy equivalence
\[
\dblBrac{\Nresol{G(r-1)}}
\simeq\dblBrac{\Presol{G(r-1)}}\oplus\dblBrac{G(r-1)}[1]
\quad.
\]
Since we have $\dblBrac{\Nresol{G(r-1)}}=\dblBrac{D(r)}$ and $\dblBrac{\Presol{G(r-1)}}\simeq\dblBrac{D(r-2)}$, we obtain
\[
\dblBrac{D(r)}
\simeq \dblBrac{D(r-2)}\oplus\dblBrac{G(r-1)}[1]
\quad,
\]
which completes the induction step.
\end{proof}

Finally, \cref{main:twistknot} directly follows from \cref{prop:kh-twist} and \cref{prop:kh-twistsing}.

\begin{remark}
Shumakovitch \cite{Shumakovitch2018} showed that Khovanov homology groups of a non-splitting alternating link $L$ is determined by its Jones polynomial and its signature in the cases where the coefficient ring is either the rational number field $\mathbb Q$ \cite{Lee2005}, the field of characteristic $2$ \cite{Shumakovitch2014}, or the ring of integers $\mathbb Z$.
Since twist knots are prime, \cref{prop:kh-twist} is also checked by his result for the functor $Z=Z_{0,0}$ with the above coefficients.
\end{remark}

\subsection{Homologically trivial crossing changes}
\label{sec:app:triv-genus1}

In this last section, we see some typical examples for which the genus-one morphism induces an isomorphism on homology groups.
More precisely, we will give some pairs of link diagrams $\Nresol{D}$ and $\Presol{D}$ which differ only by the sign of a common crossing $c_0$ such that the map
\begin{equation}
\label{eq:genus1-isom}
\widehat\Phi_\ast:H^\ast Z\dblBrac{\Nresol{D}}\to H^\ast Z\dblBrac{\Presol{D}}
\end{equation}
induced by the genus-one morphism at $c_0$ is an isomorphism for a functor $Z:\Cob(\varnothing,\varnothing)\to\mathbf{Mod}_k$.
Notice that, if we denote by $\Dresol{D}$ the singular link diagram obtained by replacing $c_0$ with a double point, then the above statement is equivalent to saying that the homology group $H^\ast Z\dblBrac{\Dresol{D}}_{\mathsf{crx}}$ vanishes by virtue of \cref{prop:skein-mcone} and \cref{cor:crux-homology-longex}.

The first one is regarding ``homologically trivial'' links.

\begin{proposition}
\label{prop:hunknot-genus1}
Let $\Nresol{D}$ and $\Presol{D}$ be link diagrams which differ only by the sign of a common crossing $c_0$, where $\Nresol{D}$ and $\Presol{D}$ are negative and positive respectively.
Suppose $Z:\Cob(\varnothing,\varnothing)\to\mathbf{Mod}_k$ is a functor such that $H^iZ\dblBrac{\Nresol{D}}=H^iZ\dblBrac{\Presol{D}}=0$ for all integer $i$ except a single degree $i=i_0\in\mathbb Z$.
Then, the morphism \eqref{eq:genus1-isom} is an isomorphism.
\end{proposition}
\begin{proof}
Let $\Dresol{D}$ be as above.
Then, by the argument above, it suffices to show that $H^i Z\dblBrac{\Dresol{D}}_{\mathsf{crx}}=0$ for every $i\in\mathbb Z$.
Note that, in view of \cref{prop:skein-mcone}, we have an exact sequence
\[
H^iZ\dblBrac{\Presol{D}} \to H^iZ\dblBrac{\Dresol{D}}\to H^{i+1}Z\dblBrac{\Nresol{D}}
\quad.
\]
Thus, the assumption on $H^\ast Z\dblBrac{\Nresol{D}}$ and $H^\ast Z\dblBrac{\Presol{D}}$ implies $H^iZ\dblBrac{\Dresol{D}}=0$ for $i\le i_0-2$ and $i\ge i_0+1$.
On the other hand, by \cref{cor:crux-homology-longex}, we also have an exact sequence
\[
H^{i-1}Z\dblBrac{\Dresol{D}}
\to H^{i-2}Z\dblBrac{\Dresol{D}}_{\mathsf{crx}}
\xrightarrow{\Xi_\ast} H^{i+2}Z\dblBrac{\Dresol{D}}_{\mathsf{crx}}
\to H^iZ\dblBrac{\Dresol{D}}
\quad.
\]
It follows that the morphism $\Xi_\ast$ above is isomorphism for $i\le i_0-2$ and $i\ge i_0+2$.
Since the homology group $H^\ast Z\dblBrac{\Dresol{D}}_{\mathsf{crx}}$ is bounded, this implies that $H^i Z\dblBrac{\Dresol{D}}_{\mathsf{crx}}=0$ for every integer $i\in\mathbb Z$.
Therefore, The result follows.
\end{proof}

\begin{example}
Lee \cite{Lee2005} showed that her homology group is concentrated in degree $0$ for knots.
Hence, by \cref{prop:hunknot-genus1}, the genus-one morphism on knot diagrams always induces isomorphisms on Lee homology.
\end{example}

We next discuss crossing change at a \emph{reducible crossing}; i.e.~crossings $c_0$ in \cref{fig:reducible-cross}.
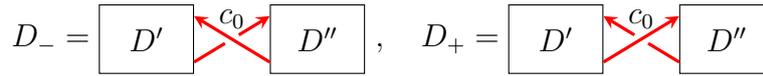
\begin{figure}[t]
\centering
\[
\Nresol{D}=
\begin{tikzpicture}[baseline=-.5ex]
\node[circle,inner sep=2] (C) at (0,0){};
\node[draw,rectangle,minimum height=5ex,minimum width=3em] (L) at (-1.25,0) {$D^\prime$};
\node[draw,rectangle,minimum height=5ex,minimum width=3em] (R) at (1.25,0) {$D^{\prime\prime}$};
\node[circle,inner sep=3] (LU) at (L.north east) {};
\node[circle,inner sep=3] (LD) at (L.south east) {};
\node[circle,inner sep=3] (RU) at (R.north west) {};
\node[circle,inner sep=3] (RD) at (R.south west) {};
\draw[red,very thick,-stealth] (LD.north) -- (C) (C) -- (RU.south);
\draw[red,very thick,-stealth] (RD.north) -- (LU.south);
\node[above] at (C) {$c_0$};
\end{tikzpicture}
\ ,\quad
\Presol{D}=
\begin{tikzpicture}[baseline=-.5ex]
\node[circle,inner sep=2] (C) at (0,0){};
\node[draw,rectangle,minimum height=5ex,minimum width=3em] (L) at (-1.25,0) {$D^\prime$};
\node[draw,rectangle,minimum height=5ex,minimum width=3em] (R) at (1.25,0) {$D^{\prime\prime}$};
\node[circle,inner sep=3] (LU) at (L.north east) {};
\node[circle,inner sep=3] (LD) at (L.south east) {};
\node[circle,inner sep=3] (RU) at (R.north west) {};
\node[circle,inner sep=3] (RD) at (R.south west) {};
\draw[red,very thick,-stealth] (LD.north) -- (RU.south);
\draw[red,very thick,-stealth] (RD.north) -- (C) (C) -- (LU.south);
\node[above] at (C) {$c_0$};
\end{tikzpicture}
\]
\caption{Reducible crossings}
\label{fig:reducible-cross}
\end{figure}

\begin{theorem}
\label{theo:reducible}
Let $\Nresol{D}$ and $\Presol{D}$ be diagrams as in \cref{fig:reducible-cross}.
Then, the genus-one morphism
\[
\widehat\Phi_\ast:\dblBrac{\Nresol{D}}\to \dblBrac{\Presol{D}}
\]
at the crossing $c_0$ is a chain homotopy equivalence.
\end{theorem}
\begin{proof}
We denote by $\Dresol{D}$ the singular link diagram obtained from $\Nresol{D}$ by replacing $c_0$ to a double point.
Then, it is easily seen that there is no $\Dresol{D}$-crux map, so $\dblBrac{\Dresol{D}}_{\mathsf{crx}}=0$ by definition.
Hence, by \cref{theo:Kh-crux}, $\dblBrac{\Dresol{D}}$ is null-homotopic.
In view of \cref{prop:skein-mcone}, this implies that $\widehat\Phi$ is a chain homotopy equivalence.
\end{proof}

\bibliographystyle{plain}
\bibliography{../../mybiblio}

\end{document}